\documentclass{amsart}
\usepackage{amsmath}
\usepackage{amssymb}
\usepackage{amsthm}
\usepackage[top=3cm,left=3cm,right=3cm,bottom=3cm]{geometry}
\usepackage[bitstream-charter]{mathdesign}
\usepackage[T1]{fontenc}
\usepackage{enumitem}
\setenumerate{label=\upshape{(\roman*)},leftmargin=*}
\usepackage[all]{xy}
\usepackage{xcolor}
\usepackage{hyperref}
\hypersetup{
    colorlinks=true,
    citecolor=black,
    linkcolor=black,
    urlcolor=black,
}
\swapnumbers
\theoremstyle{plain}
\newtheorem{theorem}{Theorem}[subsection]
\newtheorem{proposition}[theorem]{Proposition}
\newtheorem{lemma}[theorem]{Lemma}
\newtheorem{corollary}[theorem]{Corollary}
\theoremstyle{definition}
\newtheorem{definition}[theorem]{Definition}
\newtheorem{notation}[theorem]{Notation}
\newtheorem{construction}[theorem]{Construction}
\theoremstyle{remark}
\newtheorem{remark}[theorem]{Remark}
\newtheorem{example}[theorem]{Example}
\newtheorem{p}[theorem]{}
\numberwithin{equation}{theorem}

\newcommand{\mcal}[1]{\mathcal{#1}}
\newcommand{\mbb}[1]{\mathbb{#1}}
\newcommand{\mbf}[1]{\mathbf{#1}}
\newcommand{\mrm}[1]{\mathrm{#1}}
\newcommand{\mfk}[1]{\mathfrak{#1}}

\newcommand{\Ab}{\mathrm{Ab}}
\newcommand{\Add}{\mathrm{Add}}
\newcommand{\Alg}{\mathrm{Alg}}
\newcommand{\Ani}{\mathrm{Ani}}
\newcommand{\Assoc}{\mathrm{Assoc}}
\newcommand{\Bass}{\mathrm{Bass}}
\newcommand{\CAlg}{\mathrm{CAlg}}
\newcommand{\Cat}{\mathrm{Cat}}
\newcommand{\Comm}{\mathrm{Comm}}
\newcommand{\Fin}{\mathrm{Fin}}
\newcommand{\Gr}{\mathrm{Gr}}
\newcommand{\LMod}{\mathrm{LMod}}
\newcommand{\Mod}{\mathrm{Mod}}
\newcommand{\Op}{\mathrm{Op}}

\renewcommand{\Pr}{\mathrm{Pr}}
\newcommand{\Proj}{\mathrm{Proj}}
\newcommand{\PSh}{\mathrm{PSh}}
\newcommand{\Sel}{\mathrm{Sel}}

\newcommand{\Sch}{\mathrm{Sch}}
\newcommand{\Sh}{\mathrm{Sh}}
\newcommand{\Sm}{\mathrm{Sm}}
\newcommand{\Sp}{\mathrm{Sp}}
\newcommand{\Spec}{\mathrm{Spec}}
\newcommand{\St}{\mathrm{St}}
\newcommand{\Sym}{\mathrm{Sym}}
\newcommand{\Tel}{\mathrm{Tel}}
\newcommand{\Triv}{\mathrm{Triv}}
\newcommand{\V}{\mathcal{V}}

\newcommand{\et}{\mathrm{\acute{e}t}}
\newcommand{\ex}{\mathrm{ex}}
\newcommand{\fd}{\mathrm{fd}}
\newcommand{\grf}{\mathrm{grf}}
\newcommand{\id}{\mathrm{id}}
\newcommand{\lax}{\mathrm{lax}}
\newcommand{\op}{\mathrm{op}}
\newcommand{\pbf}{\mathrm{pbf}}
\newcommand{\pre}{\mathrm{pre}}

\newcommand{\sm}{\mathrm{sm}}

\newcommand{\univ}{\mathrm{univ}}

\newcommand{\OmegaP}{\Omega_{\mathbb{P}^1}}
\newcommand{\SigmaP}{\Sigma_{\mathbb{P}^1}}
\newcommand{\SpP}{\mathrm{Sp}_{\mathbb{P}^1}}
\newcommand{\TelP}{\mathrm{Tel}_{\mathbb{P}^1}}

\newcommand{\Pic}{{\mathcal{P}\mspace{-2.0mu}\mathrm{ic}}}
\newcommand{\Vect}{{\mathcal{V}\mspace{-2.0mu}\mathrm{ect}}}

\newcommand{\Fun}{\operatorname{Fun}}
\newcommand{\Map}{\operatorname{Map}}
\newcommand{\cofib}{\operatorname{cofib}}
\newcommand{\coker}{\operatorname{coker}}
\newcommand{\colim}{\operatorname*{colim}}
\newcommand{\cosk}{\operatorname{cosk}}
\newcommand{\fib}{\operatorname{fib}}
\newcommand{\intMap}{\underline{\Map}}

\title{Motivic spectra and universality of $K$-theory}

\author{Toni Annala}
\address{Department of Mathematics, University of Chicago, Eckhart Hall, 5734 S University Ave, Chicago, IL 60637, USA.}
\email{\href{mailto:tannala@uchicago.edu}{tannala@uchicago.edu}}

\author{Ryomei Iwasa}
\address{Department of Mathematical Sciences, University of Copenhagen, Universitetsparken 5, DK-2100 Copenhagen \O.}
\email{\href{mailto:ryomei@math.ku.dk}{ryomei@math.ku.dk}}

\begin{document}

\date{\today}

\begin{abstract}
We develop a theory of motivic spectra in a broad generality; in particular $\mbb{A}^1$-homotopy invariance is not assumed.
As an application, we prove that $K$-theory of schemes is a universal Zariski sheaf of spectra which is equipped with an action of the Picard stack and satisfies projective bundle formula.
\end{abstract}

\maketitle

\tableofcontents

\section{Introduction}\label{Int}

Algebraic $K$-theory is a spectrum-valued invariant of categories that is characterized by a universal property.
The point we would like to emphasize here is that, when restricted to schemes, $K$-theory has much richer structures, such as Adams operations and conjectural motivic filtrations, which lead to deep problems in algebraic geometry.
This paper was born out of the motivation to understand those further structures.
We develop a theory of motivic spectra without assuming $\mbb{A}^1$-homotopy invariance and apply it to $K$-theory.

\subsection{Universality of $K$-theory}

Our main result on $K$-theory is an algebraic analogue of Snaith's theorem for topological $K$-theory in \cite{Sn79,Sn83} and a non-$\mbb{A}^1$-localized refinement of the main theorem in \cite{GS,SO}.
To fix the notation, let $\St$ denote the $\infty$-topos of Zariski sheaves on smooth schemes ($\St$ stands for ``stacks'').
Let $\Pic$ denote the Picard stack which we regard as an $\mbb{E}_\infty$-monoid in $\St$.
Then its stabilization $\mbb{S}[\Pic]$ is an $\mbb{E}_\infty$-algebra in $\Sp(\St)$.
We say that an $\mbb{S}[\Pic]$-module $E$ in $\Sp(\St)$ \textit{satisfies projective bundle formula} if, for every $n\ge 1$ and every smooth scheme $X$, the map
\[
	\sum_{i=0}^n \beta^i \colon \bigoplus_{i=0}^n E(X) \to E(\mbb{P}^n_X)
\]
is an equivalence, where $\beta$ is the Bott element $1-[\mcal{O}(-1)]$.
By abstract reason, there exists a localization
\[
	L_\pbf \colon \Mod_{\mbb{S}[\Pic]}(\Sp(\St)) \to \Mod_{\mbb{S}[\Pic]}(\Sp(\St))
\]
whose essential image is spanned by $\mbb{S}[\Pic]$-modules which satisfy projective bundle formula.
Let $K$ denote the non-connective $K$-theory which we regard as an $\mbb{E}_\infty$-algebra in $\Sp(\St)$.
Note that we have a canonical morphism of $\mbb{E}_\infty$-algebras $\mbb{S}[\Pic]\to K$ and it factors through the localization $L_\pbf\mbb{S}[\Pic]$ since $K$-theory satisfies projective bundle formula.
Then the main theorem is stated as follows.

\begin{theorem}\label{thm1}
The canonical map
\[
	L_\pbf\mbb{S}[\Pic] \to K
\]
is an equivalence of $\mbb{E}_\infty$-algebras in $\Sp(\St)$.
\end{theorem}

In the body of the paper, we discuss and prove the case over an arbitrary qcqs derived scheme.
Also, a universality of the Selmer $K$-theory is established.
The basic idea of the proof is to regard the projective bundle formula as $\mbb{P}^1$-periodicity and work in a category where $\mbb{P}^1$ is formally inverted.
This leads to our formulation of motivic spectra, which we explain next.

\subsection{Motivic spectra}

The crucial idea of the theory of motives is to invert the pointed projective line $\mbb{P}^1$, as Grothendieck first considered in his construction of the category of pure motives.
Voevodsky's stable motivic homotopy category (cf.\ \cite{Vo,MV}) is based on the same idea and has been studied extensively in the last decades, but it completely relies on $\mbb{A}^1$-homotopy invariance.
We would like to propose more general definition in this paper.
See \cite{Bin,BDO,KMSY} for other approaches to non-$\mbb{A}^1$-local theory of motives.

We define the \textit{$\infty$-category of motivic spectra} to be the formal inversion of $\mbb{P}^1$ in $\St_*$
\[
	\SpP := \St_*[(\mbb{P}^1)^{-1}],
\]
where $\mbb{P}^1$ is pointed by $\infty$.
More precisely, $\SpP$ is a universal presentably symmetric monoidal $\infty$-category together with a symmetric monoidal functor $\SigmaP^\infty\colon\St_*\to\SpP$ which carries $\mbb{P}^1$ to an invertible object.
More generally, for an $\infty$-category $\V$ presentably tensored over $\St$ (i.e., an $\St$-module object in $\Pr^L$), we define the \textit{$\infty$-category of motivic spectra in $\V$} by
\[
	\SpP(\V) := \V_*[(\mbb{P}^1)^{-1}] \simeq \SpP\otimes_\St\V,
\]
where the tensor product is taken in the $\infty$-category $\Pr^L$ of presentable $\infty$-categories.
One can think of this construction as an analogue of stabilization of $\infty$-categories, replacing the $\infty$-topos $\Ani$ of anima with $\St$ and the circle $S^1$ with $\mbb{P}^1$.

We warn that, contrary to the usual stabilization, the $\infty$-category $\SpP(\V)$ may not be equivalent to the sequential colimit in $\Pr^L$
\[
	\TelP(\V) := \colim(\V\xrightarrow{\mbb{P}^1\otimes-}\V\xrightarrow{\mbb{P}^1\otimes-}\V\xrightarrow{\mbb{P}^1\otimes-}\cdots).
\]
However, there is still a canonical functor $\SpP(\V)\to\TelP(\V)$ and it is conservative.
More serious problem is that the $\infty$-category $\SpP(\V)$ may not be stable.
To overcome this difficulty, we extract special type of motivic spectra.
We say that a motivic spectrum $E$ in $\V$ is \textit{fundamental} if the canonical map
\[
	\mbb{P}^1\otimes E \to S^1\otimes\mbb{G}_m\otimes E
\]
admits a retraction.
Let $\SpP(\V)^\fd$ denote the full subcategory of $\SpP(\V)$ spanned by fundamental motivic spectra.
Roughly speaking, a motivic spectrum is fundamental if and only if it satisfies Bass fundamental exact sequence, and then we employ the idea of Bass construction as in \cite{TT} to prove the following.

\begin{theorem}\label{thm2}
The adjunction
\[
	\Sigma^\infty \colon \SpP \rightleftarrows \SpP(\Sp) \colon \Omega^\infty
\]
restricts to an adjoint equivalence
\[
	\Sigma^\infty \colon \SpP^\fd \overset{\sim}{\rightleftarrows} \SpP(\Sp)^\fd \colon \Omega^\infty.
\]
\end{theorem}

To move further on, we develop the theory of orientation for motivic spectra in parallel with the theory of complex orientation in topology.
We say that a motivic spectrum $E$ is \textit{orientable} if the map
\[
	[\mcal{O}(1)]\otimes\id_E \colon \mbb{P}^1\otimes E \to \Pic\otimes E
\]
admits a retraction.
We remark that if a motivic spectrum is orientable then it is fundamental.
Then we formulate projective bundle formula for oriented motivic spectra, develop a theory of Chern classes, and calculate the cohomology of the moduli stack $\Vect_n$ of rank $n$ vector bundles by adopting the argument in \cite{AI}.

\begin{theorem}\label{thm3}
Let $E$ be a homotopy commutative oriented motivic ring spectrum which satisfies projective bundle formula.
Then there is a natural ring isomorphism
\[
	E^{*,*}(\Vect_{n,S}) \simeq E^{*,*}(S)[[c_1,\dotsc,c_n]]
\]
for every qcqs scheme $S$.
\end{theorem}

\begin{remark}
All cohomology theories treated in \cite{AI} were assumed to have finite quasi-smooth transfers.
We have succeeded in removing this assumption by introducing the notion of oriented motivic spectra.
Although oriented motivic spectra are expected to admit transfers in good generality, we do not discuss this problem in this paper.
\end{remark}

Let us go back to $K$-theory.
We see that $K$-theory of schemes is represented by a motivic spectrum $K$, which is canonically oriented, satisfies projective bundle formula, and periodic, i.e., $\SigmaP K\simeq K$.
In particular, it has a unique infinite delooping as motivic spectra by Theorem \ref{thm2}, which recovers the non-connective $K$-theory.
Furthermore, we prove an equivalence of motivic spectra 
\[
	K \simeq \colim(\SigmaP^\infty\Omega^\infty K \xrightarrow{\beta}
		\SigmaP^{\infty-1}\Omega^\infty K \xrightarrow{\beta} \SigmaP^{\infty-2}\Omega^\infty K \xrightarrow{\beta} \cdots).
\]
Consequently, we get an equivalence
\[
	\Map(K,E) \simeq \lim_n\Map(\Omega^\infty K,\OmegaP^{\infty-n} E)
\]
for a motivic spectrum $E$, and each term in the limit is calculated by Theorem \ref{thm3} if $E$ satisfies projective bundle formula.
Then, by proceeding with calculation, we obtain Theorem \ref{thm1}.

\subsection{Organization of the paper}

Section \ref{IST} deals with formal inversion in a purely categorical setting.
In Section \ref{MSp}, we define motivic spectra and prove Theorem \ref{thm2} in a more general form.
In Section \ref{OPB}, we discuss orientations and projective bundle formula for motivic spectra.
In Section \ref{Vec}, we develop a theory of Chern classes for oriented motivic spectra and prove Theorem \ref{thm3}.
In Section \ref{MSn}, we prove Theorem \ref{thm1} and its variant for Selmer $K$-theory.
Appendix \ref{CPr} collects some categorical preliminaries.
Each section begins with a brief summary.

\subsection{Acknowledgement}

The formulation of Theorem \ref{thm1} is due to Dustin Clausen.
We would like to thank him for the essential remark that our previous work \cite{AI} may be useful in proving Theorem \ref{thm1}, and for many helpful discussions.
We are very grateful to Marc Hoyois for pointing out a mistake in the proof of \cite{AI} and helping us to correct it in this paper.
We also thank Lars Hesselholt, Markus Land, and Shuji Saito for helpful discussions.
The first author was supported by the Vilho, Yrj\"o and Kalle V\"ais\"al\"a Foundation of the Finnish Academy of Science and Letters.
The second author was supported by the European Union's Horizon 2020 research and innovation programme under the Marie Sk\l{}odowska-Curie grant agreement No.\ 896517.

\subsection{Developments since the first and second versions}

Since the first version of this paper in April 2022, there have been several developments of motivic spectra in the sense of this paper.
The two papers \cite{AHI,AHI2} written with Marc Hoyois can be considered as sequels.
When \cite{AHI} was announced in March 2023, we made minor changes to this paper accordingly, which is the second version.
In this third version, we have made some more minor changes, but the numbering has not been changed from the second version.

Here are some of these later developments that are relevant to the content of this paper.
First of all, notations: the category $\mrm{MS}_S$ of motivic spectra over a derived scheme $S$ in the sense of \cite{AHI} is exactly $\SpP(\Sp(\St_S^\ex))$ in the sense of this paper.
The category $\mrm{MS}_S$ of motivic spectra in the sense of \cite{AHI2} is essentially the same but uses the Nisnevich topology instead of the Zariski topology, i.e., $\SpP(\Sp(\St_S^{\mrm{Nis},\ex}))$.
We now have the feeling that the Nisnevich $\mrm{MS}$ is the fundamental playground for motivic stable homotopy theory; but all the results in this paper and \cite{AHI} hold Zariski locally.

The most notable technique made available since then is the \textit{projective bundle homotopy invariance} in $\mrm{MS}$, \cite[Theorem 4.1]{AHI}.
Using this, we proved an equivalence $\Gr_n\simeq\Vect_n$ in $\mrm{MS}$, \cite[Theorem 5.3]{AHI}, generalizing Theorem \ref{Che} below to the non-oriented case.
This makes part of Section \ref{Vec} obsolete.
Nevertheless, we would like to keep our original argument, as the proof provides a different perspective to study the pbf-localization of oriented theories.
To get the updated perspective, the reader could skim Section \ref{Vec} and read Section 4 and 5 of \cite{AHI} at the same time.

The other sections are still fresh as they are.
The following are some noteworthy results proved later based on the results here.
The fundamental stability (Theorem \ref{FSt}) combined with the projective bundle homotopy invariance gives an equivalence
\[
	\mrm{MS}_S = \SpP(\Sp(\St_S^\ex)) \simeq \SpP(\St_S^\ex)^\fd, 
\]
in other words, every object in $\mrm{MS}_S$ is fundamental; see \cite[Corollary 4.13]{AHI}.
Algebraic Conner–Floyd isomorphism (\cite[Theorem 8.11]{AHI}) is established using the universality of $K$-theory (Theorem \ref{thm:AlK}).
The universality of $K$-theory (Theorem \ref{thm:AlK}) also allows us to compute cohomology operations of $K$-theory as in Section 9 of \cite{AHI2}.

\newpage

\section{Formal inversion, spectra, and telescopes}\label{IST}

Let $\mcal{C}$ be a presentably symmetric monoidal $\infty$-category and $c$ an object in $\mcal{C}$.
Then we will see that there exists a presentably symmetric monoidal $\infty$-category $\mcal{C}[c^{-1}]$ which is obtained by formally inverting $c$ in $\mcal{C}$.
More generally, for an $\infty$-category $\mcal{D}$ presentably tensored over $\mcal{C}$, the \textit{formal inversion} $\mcal{D}[c^{-1}]$ is well-defined as an $\infty$-category presentably tensored over $\mcal{C}[c^{-1}]$.
The purpose of this section is to present basic tools for studying these $\infty$-categories.

Let us briefly recall localization of modules over commutative rings.
Let $R$ be a commutative ring, $r$ an element in $R$, and $M$ an $R$-module.
Then the localization $M[r^{-1}]$ is modeled by the sequential colimit
\[
	M[r^{-1}] \simeq \colim(M\xrightarrow{r}M\xrightarrow{r}M\xrightarrow{r}\cdots).
\]
One might wonder if there would be an analogue for formal inversion of $\infty$-categories.
However, it turns out that the analogous construction does not give a correct model in general.
We can consider the sequential colimit in $\Pr^L$
\[
	\Tel_c(\mcal{D}) := \colim(\mcal{D}\xrightarrow{c\otimes-}\mcal{D}\xrightarrow{c\otimes-}\mcal{D}\xrightarrow{c\otimes-}\cdots),
\]
but it is not equivalent to the formal inversion $\mcal{D}[c^{-1}]$ in general.
The point is that we need to incorporate more symmetricity in order to obtain a correct model.
This is achieved by what we call \textit{$c$-spectra}; in other literature, it is often referred to as symmetric $c$-spectra.
We will define the $\infty$-category $\Sp_c(\mcal{D})$ of $c$-spectra and prove that it is equivalent to the formal inversion $\mcal{D}[c^{-1}]$ completely in general (Proposition \ref{prop:cSp}).

We refer to the sequential colimit $\Tel_c(\mcal{D})$ as the $\infty$-category of $c$-telescopes.
$c$-telescopes are structurally simpler than $c$-spectra and play a complementary role for studying $c$-spectra.
The upshot is that there is a canonical conservative functor $\Sp_c(\mcal{D})\to\Tel_c(\mcal{D})$ and it is an equivalence under a certain symmetricity on $c$ (Proposition \ref{prop:CST}).
See \cite{Ho,Ro} for related works.

\subsection{Categorical conventions}\label{CCo}

We generally follow the notation in \cite{HTT,HA}.
See also \S\ref{Mod} for the theory of modules over commutative algebras.
The following is a glossary of terms that may require further explanations.

\begin{p}[Anima]
We adopt the term ``anima'' following \cite{CS} and let $\Ani$ denote the $\infty$-category of anima, which is the $\infty$-category of spaces in the sense of \cite{HTT}.
\end{p}

\begin{p}[$\infty$-category of $\infty$-categories]
Let $\Cat_\infty$ denote the $\infty$-category of possibly large $\infty$-categories and $\Cat_\infty^\sm$ denote its full subcategory spanned by small $\infty$-categories.
We suppose that $\Cat_\infty$ and $\Cat_\infty^\sm$ are equipped with the cartesian symmetric monoidal structures.
\end{p}

\begin{p}[Tensored $\infty$-category]
Let $\mcal{C}$ be a monoidal $\infty$-category.
Recall from \cite[4.2.1.19]{HA} that an $\infty$-category $\mcal{D}$ is \textit{left-tensored over $\mcal{C}$} if we are supplied with an $\mbb{LM}$-monoidal $\infty$-category $\mcal{D}^\otimes$, an equivalence of monoidal $\infty$-categories $\mcal{D}^\otimes_\mfk{a}\simeq\mcal{C}$, and an equivalence of $\infty$-categories $\mcal{D}^\otimes_\mfk{m}\simeq\mcal{D}$.
Note that an $\infty$-category left-tensored over $\mcal{C}$ is identified with a left $\mcal{C}$-module object in $\Cat_\infty$.

When $\mcal{C}$ underlies a symmetric monoidal $\infty$-category, we omit the prefix ``left'' from the notation, because left or right does not make any difference.
Actually, we can replace the operad $\mbb{LM}^\otimes$ by a simpler operad $\mbb{M}^\otimes$ to deal with this case, cf.\ \S\ref{Mod}.
\end{p}

\begin{p}[Presentably symmetric/tensored $\infty$-category]
Let $\Pr^L$ denote the $\infty$-category of presentable $\infty$-categories and colimit-preserving functors.
We suppose that $\Pr^L$ is equipped with the symmetric monoidal structure as in \cite[4.8.1.15]{HA}.
We refer to a commutative algebra object in $\Pr^L$ as a \textit{presentably symmetric monoidal $\infty$-category}.
Given a presentably symmetric monoidal $\infty$-category $\mcal{C}$, we refer to a $\mcal{C}$-module object in $\Pr^L$ as an \textit{$\infty$-category presentably tensored over $\mcal{C}$}, cf.\ \S\ref{PrM}.
\end{p}

\begin{p}[Linear functor]
Let $\mcal{D}$ and $\mcal{D}'$ be $\infty$-categories left-tensored over a monoidal category $\mcal{C}$.
Recall from \cite[4.6.2.7]{HA} that a \textit{(lax) $\mcal{C}$-linear functor} $\mcal{D}\to\mcal{D}'$ is a (lax) $\mbb{LM}$-monoidal functor $\mcal{D}^\otimes\to\mcal{D}^{\prime\otimes}$ which is the identity on $\mcal{C}$.
Note that a $\mcal{C}$-linear functor is identified with a morphism in $\LMod_\mcal{C}(\Cat)$. 
\end{p}

\begin{p}[Exponential object]
Let $\mcal{C}$ be a presentably symmetric monoidal $\infty$-category and $\mcal{D}$ an $\infty$-category presentably tensored over $\mcal{C}$.
For $X,Y\in\mcal{D}$ and $A\in\mcal{C}$, the \textit{exponential objects} $X^Y\in\mcal{C}$ and $X^A\in\mcal{D}$ are defined by the adjunctions
\[
	-\otimes Y \colon \mcal{C} \rightleftarrows \mcal{D} \colon (-)^Y \qquad
	A\otimes- \colon \mcal{D} \rightleftarrows \mcal{D} \colon (-)^A.
\]
We also write $\intMap(Y,X):=X^Y$ for $X,Y\in\mcal{D}$.
\end{p}

\begin{p}[Smashing localization]
Recall that a localization $L\colon\mcal{D}\to\mcal{D}$ is a functor of the form $L=G\circ F$ for some functor $F\colon\mcal{D}\to\mcal{D}'$ which admits a fully faithful right adjoint $G\colon\mcal{D}'\to\mcal{D}$.
Suppose that an $\infty$-category $\mcal{D}$ is tensored over a symmetric monoidal category $\mcal{C}$.
Then we say that a localization $L\colon\mcal{D}\to\mcal{D}$ is \textit{smashing} if it has the form $L=A\otimes -$ for some commutative algebra object $A$ in $\mcal{C}$.
\end{p}

\subsection{Formal inversion}\label{Inv}

\begin{p}
We fix a presentably symmetric monoidal $\infty$-category $\mcal{C}$ and an object $c$ in $\mcal{C}$ throughout this section.
We usually denote by $\mcal{D}$ an $\infty$-category presentably tensored over $\mcal{C}$.
\end{p}

\begin{proposition}\label{prop:Inv}
There exists a smashing localization
\[
	(-)[c^{-1}] \colon \Mod_\mcal{C}(\Pr^L) \to \Mod_\mcal{C}(\Pr^L)
\]
whose essential image is spanned by $\infty$-categories presentably tensored over $\mcal{C}$ on which $c$ acts as an equivalence.
\end{proposition}
\begin{proof}
This is \cite[Proposition 2.9]{Ro}.
In what follows, we will construct a concrete model of the localization $(-)[c^{-1}]$, which independently proves its existence, cf.\ Proposition \ref{prop:cSp}.
Then the assertion that it is smashing is a formal consequence of the obvious fact that its essential image is both an ideal and a co-ideal of $\Mod_\mcal{C}(\Pr^L)$, cf. Lemma \ref{lem:SmL}.
\end{proof}

\begin{p}[Formal inversion]
For an $\infty$-category $\mcal{D}$ presentably tensored over $\mcal{C}$, we refer to $\mcal{D}[c^{-1}]$ as the \textit{formal inversion of $c$ in $\mcal{D}$}.
\end{p}

\begin{remark}
Since the localization $(-)[c^{-1}]$ is smashing, the unit map $u\colon\mcal{C}\to\mcal{C}[c^{-1}]$ exhibits $\mcal{C}[c^{-1}]$ as an idempotent object in $\Mod_\mcal{C}(\Pr^L)$, and thus $\mcal{C}[c^{-1}]$ admits a unique presentably symmetric monoidal structure for which $u$ is (uniquely) promoted to a symmetric monoidal functor.
Then the restriction of scalars along $u$ induces an equivalence
\[
	\Mod_{\mcal{C}[c^{-1}]}(\Pr^L) \xrightarrow{\sim} \Mod_\mcal{C}(\Pr^L)[c^{-1}].
\]
In particular, the formal inversion $\mcal{D}[c^{-1}]$ is presentably tensored over $\mcal{C}[c^{-1}]$ in a canonical way.
\end{remark}

\begin{lemma}\label{lem:Inv}
There is a natural $\mcal{C}$-linear equivalence
\[
	(\mcal{D}[c^{-1}])[d^{-1}] \simeq \mcal{D}[(c\otimes d)^{-1}]
\]
for every $\infty$-category $\mcal{D}$ presentably tensored over $\mcal{C}$ and for every $c,d\in\mcal{C}$.
\end{lemma}
\begin{proof}
It is straightforward to check that both sides have the same universal property.
\end{proof}

\subsection{$c$-spectra}\label{cSp}

\begin{construction}\label{cons:cSp}
Let $B\Sigma_\mbb{N}$ be the free commutative monoid in $\Ani$ with a single generator $e$.
We consider the lax symmetric monoidal functor
\[
	(-)^\Sigma := \Fun(B\Sigma_\mbb{N},-) \colon \Pr^L \to \Pr^L,
\]
which encodes the Day convolution (Construction \ref{cons:Day}), and apply it to a presentably symmetric monoidal $\infty$-category $\mcal{C}$ and an $\infty$-category $\mcal{D}$ presentably tensored over $\mcal{C}$.
Then $\mcal{C}^\Sigma$ is a presentably symmetric monoidal $\infty$-category and $\mcal{D}^\Sigma$ is presentably tensored over $\mcal{C}^\Sigma$.
We consider the following natural transformations:
\begin{itemize}[label={---},leftmargin=*]
\item Let $F\colon\id\to(-)^\Sigma$ be the natural transformation obtained as the left Kan extension along the morphism $*\to B\Sigma_\mbb{N}$ of commutative monoids.
\item Let $s_+\colon(-)^\Sigma\to(-)^\Sigma$ be the natural transformation obtained as the left Kan extension along the morphism $e\colon B\Sigma_\mbb{N}\to B\Sigma_\mbb{N}$ of $B\Sigma_\mbb{N}$-modules.
\end{itemize}
Then $F\colon\mcal{C}\to\mcal{C}^\Sigma$ is symmetric monoidal, $F\colon\mcal{D}\to\mcal{D}^\Sigma$ is $\mcal{C}$-linear, and $s_+\colon\mcal{D}^\Sigma\to\mcal{D}^\Sigma$ is $\mcal{C}^\Sigma$-linear.
Furthermore, we have canonical equivalences (Lemma \ref{lem:Day})
\[
	\mcal{D}^\Sigma \simeq \mcal{C}^\Sigma\otimes_\mcal{C}\mcal{D} \qquad
	(\mcal{D}\xrightarrow{F}\mcal{D}^\Sigma) \simeq (\mcal{C}\xrightarrow{F}\mcal{C}^\Sigma)\otimes_\mcal{C}\mcal{D} \qquad
	(\mcal{D}^\Sigma\xrightarrow{s_+}\mcal{D}^\Sigma) \simeq (\mcal{C}^\Sigma\xrightarrow{s_+}\mcal{C}^\Sigma)\otimes_\mcal{C}\mcal{D},
\]
where the tensor products are taken in $\Mod_\mcal{C}(\Pr^L)$.
We consider the adjunctions
\[
	F \colon \mcal{D} \rightleftarrows \mcal{D}^\Sigma \colon U \qquad
	s_+ \colon \mcal{D}^\Sigma \rightleftarrows \mcal{D}^\Sigma \colon s_-,
\]
where $U$ is the pre-composition by $*\to B\Sigma_\mbb{N}$ and $s_-$ is the pre-composition by $e\colon B\Sigma_\mbb{N}\to B\Sigma_\mbb{N}$.
For $n\ge 0$, we write $F_n:=(s_+)^{\circ n}\circ F$ and $U_n:=U\circ(s_-)^{\circ n}$. 
\end{construction}

\begin{remark}
We illustrate the previous construction in a more concrete way.
An object in $\mcal{D}^\Sigma$ is given by a sequence $Y=(Y_0,Y_1,\dotsc)$ in  $\mcal{D}$ with a $\Sigma_n$-action on $Y_n$.
For $X\in\mcal{C}^\Sigma$ and $Y\in\mcal{D}^\Sigma$, we have a formula
\[
	(X\otimes Y)_n = \bigoplus_{p+q=n} \Sigma_{p+q}\otimes_{\Sigma_p\times\Sigma_q}(X_p\otimes Y_q).
\]
The functor $U\colon\mcal{D}^\Sigma\to\mcal{D}$ carries $Y$ to $Y_0$ and the functor $F\colon\mcal{D}\to\mcal{D}^\Sigma$ carries $d$ to $(d,*,*,\cdots)$, where $*$ is an initial object of $\mcal{D}$.
For $Y\in\mcal{D}^\Sigma$, we have $s_-(Y)_n=Y_{n+1}$ with the restricted action of $\Sigma_n$ on $Y_{n+1}$ and 
\[
	s_+(Y)_n =
	\begin{cases}
		* &\text{if }n=0 \\
		\Sigma_n\otimes_{\Sigma_{n-1}}Y_{n-1} &\text{if }n>0.
	\end{cases}
\]
In other words, the functor $s_+\colon\mcal{D}^\Sigma\to\mcal{D}^\Sigma$ is the multiplication by $s_+(1)=(*,1,*,*,\dotsc)$.
\end{remark}

\begin{lemma}\label{lem:cSp}
\leavevmode
\begin{enumerate}
\item The functors $U_n\colon\mcal{D}^\Sigma\to\mcal{D}$ and $s_-\colon\mcal{D}^\Sigma\to\mcal{D}^\Sigma$ are $\mcal{C}$-linear.
\item The natural transformation $\id_\mcal{D}\to U\circ F$ is an equivalence.
\item The family of functors $\{U_n\colon\mcal{D}^\Sigma\to\mcal{D}\}_{n\ge 0}$ is conservative.
\end{enumerate}
\end{lemma}
\begin{proof}
For (i), note that these functors are clearly lax $\mcal{C}$-linear, but then an easy inspection shows that they are actually $\mcal{C}$-linear.
(ii) and (iii) are obvious.
\end{proof}

\begin{definition}[Lax $c$-spectrum]\label{def:cSp}
Let $S_c$ be the free commutative algebra in $\mcal{C}^\Sigma$ generated by $F_1(c)$.
For an $\infty$-category $\mcal{D}$ presentably tensored over $\mcal{C}$, we define
\[
	\Sp^\lax_c(\mcal{D}) := \Mod_{S_c}(\mcal{D}^\Sigma)
\]
and call it the \textit{$\infty$-category of lax $c$-spectra in $\mcal{D}$}.
Then $\Sp^\lax_c(\mcal{C})$ admits a presentably symmetric monoidal structure in a canonical way and $\Sp^\lax_c(\mcal{D})$ is presentably tensored over $\Sp^\lax_c(\mcal{C})$.
\end{definition}

\begin{p}[Adjunction]\label{p:cSp}
We consider the following adjunctions:
\begin{itemize}[label={---},leftmargin=*]
\item The adjunction $(F,U)$ together with $S_c\otimes-$ induces an adjunction
\[
	F_c := S_c\otimes F \colon \mcal{D} \rightleftarrows \Sp^\lax_c(\mcal{D}) \colon U.
\]
Then $F_c\colon\mcal{C}\to\Sp^\lax_c(\mcal{C})$ is symmetric monoidal and $F_c\colon\mcal{D}\to\Sp^\lax_c(\mcal{D})$ is $\mcal{C}$-linear.
\item The adjunction $(s_+,s_-)$ induces an adjunction 
\[
	s_+ \colon \Sp^\lax_c(\mcal{D}) \rightleftarrows \Sp^\lax_c(\mcal{D}) \colon s_-.
\]
Then $s_+\colon\Sp^\lax_c(\mcal{D})\to\Sp^\lax_c(\mcal{D})$ is $\Sp^\lax_c(\mcal{C})$-linear.
\end{itemize}
\end{p}

\begin{lemma}\label{lem:cSp2}
\leavevmode
\begin{enumerate}
\item We have $U_n(S_c)\simeq c^{\otimes n}$ for each $n\ge 0$.
\item The functors $U_n\colon\Sp^\lax_c(\mcal{D})\to\mcal{D}$ and $s_-\colon\Sp^\lax_c(\mcal{D})\to\Sp^\lax_c(\mcal{D})$ are $\mcal{C}$-linear.
\item The natural transformation $\id_\mcal{D}\to U\circ F_c$ is an equivalence.
\item The family of functors $\{U_n\colon\Sp^\lax_c(\mcal{D})\to\mcal{D}\}_{n\ge 0}$ is conservative.
\end{enumerate}
\end{lemma}
\begin{proof}
We have
\[
	U_n(S_c) \simeq U_n(\Sym^n(F_1(c))) \simeq (\Sigma_n\otimes c^{\otimes n})_{h\Sigma_n}.
\]
Here the $\Sigma_n$-action on $\Sigma_n\otimes c^{\otimes n}$ is the diagonal action, and thus the homotopy orbit is equivalent to $c^{\otimes n}$, which proves (i). 
The other assertions are immediate from (i) and Lemma \ref{lem:cSp}.
\end{proof}

\begin{construction}\label{cons:cSp2}
For each lax $c$-spectrum $E$ in $\mcal{D}$, we have natural equivalences
\[
	c\otimes s_+E \simeq s_+(c\otimes E) \simeq F_1(c)\otimes E.
\]
Hence, the multiplication by $F_1(c)$ yields a morphism of lax $c$-spectra
\[
	\sigma_E \colon s_+(c\otimes E) \to E.
\]
We write $\sigma^\#_E\colon E\to (s_-E)^c$ for the adjoint of $\sigma_E$.
\end{construction}

\begin{definition}[$c$-spectrum]\label{def:cSp2}
A \textit{$c$-spectrum in $\mcal{D}$} is a lax $c$-spectrum $E$ in $\mcal{D}$ such that the map $\sigma^\#_E\colon E\to(s_-E)^c$ is an equivalence.
Let $\Sp_c(\mcal{D})$ denote the full subcategory of $\Sp^\lax_c(\mcal{D})$ spanned by $c$-spectra.
\end{definition}

\begin{lemma}\label{lem:cSp3}
The $\infty$-category $\Sp_c(\mcal{D})$ is an accessible localization of $\Sp^\lax_c(\mcal{D})$ with respect to all the maps $\sigma_E\colon s_+(c\otimes E)\to E$ for $E\in\Sp^\lax_c(\mcal{D})$.
\end{lemma}
\begin{proof}
Since $\Sp^\lax_c(\mcal{D})$ is presentable and $s_+$ and $c\otimes-$ preserve all small colimits, the class of maps $\{\sigma_E\}_E$ is indeed generated by a small set.
Then the assertion follows immediately.
\end{proof}

\begin{remark}
Let $L$ denote the localization $\Sp^\lax_c(\mcal{D})\to\Sp_c(\mcal{D})$.
Then it follows from \cite[2.2.1.9]{HA} (see also \cite[4.1.7.4]{HA}) that $\Sp_c(\mcal{C})$ admits a unique presentably symmetric monoidal structure for which $L$ is symmetric monoidal and that $\Sp_c(\mcal{D})$ is presentably tensored over $\Sp_c(\mcal{C})$ in a unique way so that $L$ is $\Sp^\lax_c(\mcal{C})$-linear. 
\end{remark}

\begin{lemma}\label{lem:cSp+}
There is a natural $\mcal{C}$-linear equivalence
\[
	\Sp_c(\mcal{D}) \simeq \Sp_c(\mcal{C})\otimes_\mcal{C}\mcal{D},
\]
where the tensor product is taken in $\Mod_\mcal{C}(\Pr^L)$.
\end{lemma}
\begin{proof}
It follows from Lemma \ref{lem:cSp3} that $\Sp_c(\mcal{C})\otimes_\mcal{C}\mcal{D}$ is a localization of $\Sp^\lax_c(\mcal{C})\otimes_\mcal{C}\mcal{D}$ with respect to morphisms of the form $\sigma_E\otimes d$ for some $E\in\Sp^\lax_c(\mcal{C})$ and $d\in\mcal{D}$.
Then it is identified with $\Sp_c(\mcal{D})$ under the equivalence
\[
	\Sp^\lax_c(\mcal{D}) \simeq \Sp^\lax_c(\mcal{C})\otimes_\mcal{C}\mcal{D},
\]
which holds by definition and Lemma \ref{lem:PrM}.
\end{proof}

\begin{p}[Adjunction]\label{p:cSp2}
The adjunctions in \ref{p:cSp} derive the following adjunctions:
\[
	LF_c \colon \mcal{D} \rightleftarrows \Sp_c(\mcal{D}) \colon U \qquad
	Ls_+ \colon \Sp_c(\mcal{D}) \rightleftarrows \Sp_c(\mcal{D}) \colon s_-.
\]
Then $LF_c\colon\mcal{C}\to\Sp_c(\mcal{C})$ is symmetric monoidal, $LF_c\colon\mcal{D}\to\Sp_c(\mcal{D})$ is $\mcal{C}$-linear, and $Ls_+\colon\Sp_c(\mcal{D})\to\Sp_c(\mcal{D})$ is $\Sp_c(\mcal{C})$-linear.
\end{p}

\begin{lemma}\label{lem:cSp4}
There are natural equivalences
\[
	c\otimes E \simeq s_-E \qquad E^c \simeq Ls_+E
\]
for every $c$-spectrum $E$ in $\mcal{D}$, and $c$ acts as an equivalence on $\Sp_c(\mcal{D})$.
\end{lemma}
\begin{proof}
By definition, we have natural equivalences 
\[
	E \simeq (s_-E)^c \simeq s_-(E^c)
\]
for every $c$-spectrum $E$ in $\mcal{D}$.
Hence, $s_-$ and $(-)^c$ are inverse of each other, from which the assertion follows immediately.
\end{proof}

\begin{proposition}\label{prop:cSp}
There is a natural $\mcal{C}$-linear equivalence
\[
	\mcal{D}[c^{-1}] \simeq \Sp_c(\mcal{D})
\]
for every $\infty$-category $\mcal{D}$ presentably tensored over $\mcal{C}$.
\end{proposition}
\begin{proof}
It suffices to show the following:
\begin{enumerate}
\item $c$ acts as an equivalence on $\Sp_c(\mcal{D})$.
\item If $c$ acts as an equivalence on $\mcal{D}$, then $LF_c\colon\mcal{D}\to\Sp_c(\mcal{D})$ is an equivalence.
\end{enumerate}
Indeed, $\Sp_c$ is a monad in $\Mod_\mcal{C}(\Pr^L)$ by Lemma \ref{lem:cSp+} and (i) and (ii) imply that $\Sp_c$ is an idempotent monad; then it is readily identified with the formal inversion $(-)[c^{-1}]$ by referring to (i) and (ii) again.

We have seen (i) in Lemma \ref{lem:cSp4}.
To show (ii), assume that $c$ acts as an equivalence on $\mcal{D}$.
Note that this assumption implies that a lax $c$-spectrum $E$ is a $c$-spectrum if and only if the canonical map $c\otimes E\to s_-E$ is an equivalence, where the tensor product is taken in $\Sp^\lax_c(\mcal{D})$.

We first prove that $U\colon\Sp_c(\mcal{D})\to\mcal{D}$ is conservative.
For a $c$-spectrum $E$, we have natural equivalences
\[
	U_nE \simeq U_{n-1}(s_-E) \simeq U_{n-1}(c\otimes E) \simeq c\otimes (U_{n-1}E),
\]
where the tensor product $c\otimes E$ is calculated in $\Sp^\lax_c(\mcal{D})$ and thus the last equivalence holds by Lemma \ref{lem:cSp2} (ii).
Since $\{U_n\}_{n\ge 0}$ is conservative by Lemma \ref{lem:cSp2} (iv), we conclude that $U=U_0$ is conservative.

It remains to show that $\id\simeq U\circ LF_c$.
By Lemma \ref{lem:cSp2} (iii), it suffices to show that $F_c\simeq LF_c$, that is, $F_c(d)=S_c\otimes F(d)$ is a $c$-spectrum for each $d\in\mcal{D}$.
We need to show that the canonical map $c\otimes (S_c\otimes F(d))\to s_-(S_c\otimes F(d))$ is an equivalence, and it is reduced to showing that the canonical map $c\otimes S_c\to s_-S_c$ is an equivalence.
This follows from Lemma \ref{lem:cSp2} (i) and the conservativity of $\{U_n\}_{n\ge 0}$.
\end{proof}

\subsection{$c$-telescopes}\label{cTe}

We develop the theory of $c$-telescopes in a parallel way with the theory of $c$-spectra.

\begin{construction}\label{cons:cTe}
Here is a parallel construction with Construction \ref{cons:cSp}.
We regard $\mbb{N}$ as a commutative monoid in $\Ani$ and consider the lax symmetric monoidal functor
\[
	(-)^\mbb{N} := \Fun(\mbb{N},-) \colon \Pr^L \to \Pr^L,
\]
which encodes the Day convolution, and apply it to a presentably symmetric monoidal $\infty$-category $\mcal{C}$ and an $\infty$-category $\mcal{D}$ presentably tensored over $\mcal{C}$.
Then $\mcal{C}^\mbb{N}$ is a presentably symmetric monoidal $\infty$-category and $\mcal{D}^\mbb{N}$ is presentably tensored over $\mcal{C}^\mbb{N}$.
We consider the following natural transformations:
\begin{itemize}[label={---},leftmargin=*]
\item Let $G\colon\id\to(-)^\mbb{N}$ be the natural transformation obtained as the left Kan extension along the morphism $*\to\mbb{N}$ of commutative monoids.
\item Let $s_+\colon(-)^\mbb{N}\to(-)^\mbb{N}$ be the natural transformation obtained as the left Kan extension along the morphism $+1\colon\mbb{N}\to\mbb{N}$ of $\mbb{N}$-modules.
\end{itemize}
Then $G\colon\mcal{C}\to\mcal{C}^\mbb{N}$ is symmetric monoidal, $G\colon\mcal{D}\to\mcal{D}^\mbb{N}$ is $\mcal{C}$-linear, and $s_+\colon\mcal{D}^\mbb{N}\to\mcal{D}^\mbb{N}$ is $\mcal{C}^\mbb{N}$-linear.
We consider the adjunctions
\[
	G \colon \mcal{D} \rightleftarrows \mcal{D}^\mbb{N} \colon U \qquad
	s_+ \colon \mcal{D}^\mbb{N} \rightleftarrows \mcal{D}^\mbb{N} \colon s_-,
\]
where $U$ is the pre-composition by $*\to\mbb{N}$ and $s_-$ is the pre-composition by $+1\colon\mbb{N}\to\mbb{N}$.
For $n\ge 0$, we write $G_n:=(s_+)^{\circ n}\circ G$ and $U_n:=U\circ(s_-)^{\circ n}$.
\end{construction}

\begin{definition}[Lax $c$-telescope]\label{def:cTe}
Let $S_c$ be the free $\mbb{E}_1$-algebra in $\mcal{C}^\mbb{N}$ generated by $G_1(c)$.
For an $\infty$-category $\mcal{D}$ presentably tensored over $\mcal{C}$, we define
\[
	\Tel_c^\lax(\mcal{D}) := \LMod_{S_c}(\mcal{D}^\mbb{N})
\]
and call it the \textit{$\infty$-category of lax $c$-telescopes in $\mcal{D}$}.
Then $\Tel_c^\lax(\mcal{D})$ is presentably tensored over $\mcal{C}^\mbb{N}$ in a canonical way.
\end{definition}

\begin{lemma}\label{lem:cTe}
Consider the $\mbb{N}$-indexed diagram
\[
	\mcal{D} \xrightarrow{c\otimes} \mcal{D} \xrightarrow{c\otimes} \mcal{D} \xrightarrow{c\otimes} \cdots 
\]
and let $p\colon\mcal{E}\to\mbb{N}$ be the cocartesian fibration which classifies this diagram.
Then there is a natural $\mcal{C}^\mbb{N}$-linear equivalence
\[
	\Tel^\lax_c(\mcal{D}) \simeq \Fun_{\mbb{N}}(\mbb{N},\mcal{E}),
\]
where the right hand side is tensored over $\mcal{C}^\mbb{N}$ as in Construction \ref{cons:CoM2}.
\end{lemma}
\begin{proof}
We see that the canonical functor $\Fun_{\mbb{N}}(\mbb{N},\mcal{E})\to\mcal{D}^\mbb{N}$ exhibits $\Fun_{\mbb{N}}(\mbb{N},\mcal{E})$ as monadic over $\mcal{D}^\mbb{N}$ by using \cite[4.7.3.5]{HA}.
Hence, it suffices to show that the resulting monad is equivalent to $S_c$ as monads, but this is straightforward to check.
\end{proof}

\begin{p}[Adjunction]\label{p:cTe}
We consider the following adjunctions:
\begin{itemize}[label={---},leftmargin=*]
\item The adjunction $(G,U)$ together with $S_c\otimes-$ induces an adjunction
\[
	G_c:=S_c\otimes G \colon \mcal{D} \rightleftarrows \Tel^\lax_c(\mcal{D}) \colon U.
\]
Then $G_c$ is $\mcal{C}$-linear.
\item The adjunction $(s_+,s_-)$ induces an adjunction 
\[
	s_+ \colon \Tel^\lax_c(\mcal{D}) \rightleftarrows \Tel^\lax_c(\mcal{D}) \colon s_-.
\]
Then $s_+$ is $\mcal{C}^\mbb{N}$-linear.
\end{itemize}
\end{p}

\begin{construction}\label{cons:cTe2}
For each $d\in\mcal{D}$, we have a natural equivalence
\[
	U(s_-G_c(d)) \simeq c\otimes d
\]
and by adjunction we obtain a morphism of lax $c$-telescopes
\[
	\sigma_d \colon s_+(c\otimes G_c(d)) \to G_c(d).
\]
As its dual, we obtain a map $\sigma^\#_E\colon U(E)\to U(s_-E)^c$ for each lax $c$-telescope $E$ in $\mcal{D}$.
\end{construction}

\begin{remark}
Construction \ref{cons:cSp2} and Construction \ref{cons:cTe2} represent the major difference between $c$-spectra and $c$-telescopes.
Unlike lax $c$-spectra, there is no obvious way to construct a natural morphism $s_+(c\otimes E)\to E$ that extends $\sigma_d$ for a lax $c$-telescope $E$.
\end{remark}

\begin{definition}[$c$-telescope]\label{def:cTe2}
A \textit{$c$-telescope in $\mcal{D}$} is a lax $c$-telescope $E$ in $\mcal{D}$ such that the map $\sigma^\#_E\colon U_nE\to U_{n+1}E^c$ is an equivalence for every $n\ge 0$.
Let $\Tel_c(\mcal{D})$ denote the full subcategory of $\Tel^\lax_c(\mcal{D})$ spanned by $c$-telescopes.
\end{definition}

\begin{lemma}\label{lem:cTe2}
The $\infty$-category $\Tel_c(\mcal{D})$ is an accessible localization of $\Tel^\lax_c(\mcal{D})$ with respect to all the maps $\sigma_d\colon s_+^{n+1}(c\otimes G_c(d))\to s_+^nG_c(d)$ for $d\in\mcal{D}$ and $n\ge 0$.
\end{lemma}
\begin{proof}
It is proved in the same way as Lemma \ref{lem:cSp3}.
\end{proof}

\begin{remark}
Let $L$ denote the localization $\Tel^\lax_c(\mcal{D})\to\Tel_c(\mcal{D})$.
Then $\Tel_c(\mcal{D})$ is presentably tensored over $\mcal{C}^\mbb{N}$ in a unique way so that $L$ is $\mcal{C}^\mbb{N}$-linear.
\end{remark}

\begin{lemma}\label{lem:cTe+}
There is a natural $\mcal{C}$-linear equivalence
\[
	\Tel_c(\mcal{D}) \simeq \Tel_c(\mcal{C})\otimes_\mcal{C}\mcal{D},
\]
where the tensor product is taken in $\Mod_\mcal{C}(\Pr^L)$.
\end{lemma}
\begin{proof}
It is proved in the same way as Lemma \ref{lem:cSp+}.
\end{proof}

\begin{p}[Adjunction]\label{p:cTe2}
The adjunctions in \ref{p:cTe} derive the following adjunctions:
\[
	LG_c \colon \mcal{D} \rightleftarrows \Tel_c(\mcal{D}) \colon U \qquad
	Ls_+ \colon \Tel_c(\mcal{D}) \rightleftarrows \Tel_c(\mcal{D}) \colon s_-.
\]
Then $LG_c$ is $\mcal{C}$-linear and $Ls_+$ is $\mcal{C}^\mbb{N}$-linear.
\end{p}

\begin{lemma}\label{lem:cTe3}
There is a natural $\mcal{C}$-linear equivalence 
\[
	 \Tel_c(\mcal{D}) \simeq \colim(\mcal{D} \xrightarrow{c\otimes} \mcal{D} \xrightarrow{c\otimes} \mcal{D} \xrightarrow{c\otimes} \cdots),
\]
where the colimit is taken in $\Mod_\mcal{C}(\Pr^L)$.
\end{lemma}
\begin{proof}
By Lemma \ref{lem:cTe}, it suffices to show that a lax $c$-telescope $E$ is a $c$-spectrum if and only if the corresponding section $E\colon\mbb{N}\to\mcal{E}$ is cartesian, but this is a simple paraphrase.
\end{proof}

\begin{corollary}\label{cor:cTe}
The adjonction
\[
	Ls_+ \colon \Tel_c(\mcal{D}) \rightleftarrows \Tel_c(\mcal{D}) \colon s_-
\]
is an adjoint equivalence.
\end{corollary}
\begin{proof}
Let $X\colon\mbb{N}\to\Pr^L$ denote the diagram $\mcal{D}\xrightarrow{c\otimes}\mcal{D}\xrightarrow{c\otimes}\cdots$.
Since the functor $+1\colon\mbb{N}\to\mbb{N}$ is cofinal, it induces an equivalence $\colim X\xrightarrow{\sim}\colim(X\circ (+1))$, but this functor is identified with $s_-$ under the equivalence in Lemma \ref{lem:cTe3}.
\end{proof}

\subsection{Formal properties of spectra and telescopes}\label{FPr}

\begin{lemma}\label{lem:FPr2}
Let $\mcal{D}$ be an $\infty$-category presentably tensored over $\mcal{C}$.
Then $\Sp_c(\mcal{D})$ is generated under colimits by objects the form $Ls_+^nLF_c(d)$ for $d\in\mcal{D}$ and $n\ge 0$.
\end{lemma}
\begin{proof}
Since $\Sp_c(\mcal{D})$ is a localization of $\Sp_c^\lax(\mcal{D})$, we are reduced to showing that $\Sp^\lax_c(\mcal{D})$ is generated under colimits by objects of the form $s_+^nF_c(d)$ for $d\in\mcal{D}$ and $n\ge 0$.
Since $\Sp^\lax_c(\mcal{D})=\Mod_{S_c}(\mcal{D}^\Sigma)$ is generated under colimits by free $S_c$-modules (cf.\ the proof of \cite[5.3.2.12]{HA}), we are reduced to showing that $\mcal{D}^\Sigma$ is generated under colimits by $F_n(d)$ for $d\in\mcal{D}$ and $n\ge 0$.
This is true by the same reason, that is, modules are generated under colimits by free modules.
\end{proof}

\begin{lemma}\label{lem:FPr3}
Let $\mcal{D}$ be an $\infty$-category presentably tensored over $\mcal{C}$.
Assume that $c$ is compact in $\mcal{D}$, i.e., $(-)^c\colon\mcal{D}\to\mcal{D}$ preserves filtered colimits.
Then the functor $U_0\colon\Sp_c(\mcal{D})\to\mcal{D}$ preserves filtered colimits.
\end{lemma}
\begin{proof}
Since $U_0\colon\Sp^\lax_c(\mcal{D})\to\mcal{D}$ preserves filtered colimits (in fact all small colimits), it suffices to show that the forgetful functor $\Sp_c(\mcal{D})\to\Sp^\lax_c(\mcal{D})$ preserves filtered colimits.
We have to show that, if $\{E_i\}$ is a filtered family of $c$-spectra in $\mcal{D}$, then the colimit $E:=\colim E_i$ taken in $\Sp^\lax_c(\mcal{D})$ is a $c$-spectrum, i.e., $E\to(s_-E)^c$ is an equivalence.
This is true since $s_-$ and $(-)^c$ preserve filtered colimits.
\end{proof}

\begin{corollary}\label{cor:FPr}
Let $\mcal{D}$ be an $\infty$-category presentably tensored over $\mcal{C}$.
Assume that $\mcal{D}$ is compactly generated and that $c$ is compact in $\mcal{D}$.
Then $\Sp_c(\mcal{D})$ is compactly generated.
\end{corollary}
\begin{proof}
By Lemma \ref{lem:FPr2}, it suffices to show that $LF_c\colon\mcal{D}\to\Sp_c(\mcal{D})$ preserves compact objects, which formally follows from Lemma \ref{lem:FPr3}.
\end{proof}

\begin{lemma}\label{lem:FPr4}
Let $L\colon\mcal{D}\to\mcal{D}'$ be a $\mcal{C}$-linear localization between $\infty$-categories presentably tensored over $\mcal{C}$.
Then the induced functor $\Sp_c(\mcal{D})\to\Sp_c(\mcal{D}')$ is an $\Sp_c(\mcal{C})$-linear localization with respect to all the maps of the form $Ls_+^nLF_c(f)$ for an $L$-equivalence $f$ and $n\ge 0$.
\end{lemma}
\begin{proof}
By Lemma \ref{lem:cSp+} and Lemma \ref{lem:FPr2}, we only have to show that the induced functor
\[
	L\otimes\id \colon \mcal{D}\otimes_\mcal{C}\Sp_c(\mcal{C}) \to \mcal{D}'\otimes_\mcal{C}\Sp_c(\mcal{C}) 
\]
is a localization with respect to morphisms of the form $f\otimes E$ for some $L$-equivalence $f$ in $\mcal{D}$ and $E\in\Sp_c(\mcal{C})$.
This holds formally, cf.\ the proof of \cite[4.8.1.15]{HA}.
\end{proof}

\begin{remark}
Lemma \ref{lem:FPr2}, Lemma \ref{lem:FPr3}, Corollary \ref{cor:FPr}, and Lemma \ref{lem:FPr4} have evident analogues for telescopes, which are proved in the same way.
\end{remark}

\subsection{Comparison of spectra and telescopes}\label{CST}

\begin{construction}\label{cons:CST}
Let $\bar{F}\colon(-)^\mbb{N}\to(-)^\Sigma$ be the natural transformation between endofunctors on $\Pr^L$ obtained as the left Kan extension along the canonical morphism $\mbb{N}\to B\Sigma_\mbb{N}$ of $\mbb{E}_1$-monoids.
Then $\bar{F}\colon\mcal{C}^\mbb{N}\to\mcal{C}^\Sigma$ is monoidal and $\bar{F}\colon\mcal{D}^\mbb{N}\to\mcal{D}^\Sigma$ is $\mcal{C}^\mbb{N}$-linear for an $\infty$-category $\mcal{D}$ presentably tensored over $\mcal{C}$.
We consider the adjunction
\[
	\bar{F} \colon \mcal{D}^\mbb{N} \rightleftarrows \mcal{D}^\Sigma \colon \bar{U},
\]
where $\bar{U}$ is the pre-composition by $\mbb{N}\to B\Sigma_\mbb{N}$.
Note that $\bar{F}$ commutes with $s_+$ and $\bar{U}$ commutes with $s_-$.

Since the functor $\bar{U}\colon\mcal{C}^\Sigma\to\mcal{C}^\mbb{N}$ is lax monoidal, it carries $\mbb{E}_1$-algebras to $\mbb{E}_1$-algebras.
In particular, $\bar{U}(S_c)$ is an $\mbb{E}_1$-algebra and it is identified with $S_c$ in $\mcal{C}^\mbb{N}$.
It follows that we obtain an induced adjunction
\[
	\bar{F}_c:=S_c\otimes_{\bar{F}(S_c)}\bar{F} \colon \Tel^\lax_c(\mcal{D}) \rightleftarrows \Sp^\lax_c(\mcal{D}) \colon \bar{U},
\]
and $\bar{F}_c$ is $\mcal{C}^\mbb{N}$-linear.
We see that a lax $c$-spectrum $E$ is a $c$-spectrum if and only if $\bar{U}(E)$ is a $c$-telescope.
Therefore, we obtain an induced adjunction
\[
	L\bar{F}_c \colon \Tel_c(\mcal{D}) \rightleftarrows \Sp_c(\mcal{D}) \colon \bar{U},
\]
and $L\bar{F}_c$ is $\mcal{C}^\mbb{N}$-linear.
\end{construction}

\begin{lemma}\label{lem:CST}
The functor $\bar{U}\colon\Sp_c(\mcal{D})\to\Tel_c(\mcal{D})$ is conservative.
\end{lemma}
\begin{proof}
It follows from the conservativity of $\{U_n\colon\Sp_c(\mcal{D})\to\mcal{D}\}_{n\ge 0}$, cf.\ Lemma \ref{lem:cSp2}.
\end{proof}

\begin{proposition}\label{prop:CST}
Assume that the cyclic permutation on $c^{\otimes 3}$ is homotopic to the identity.
Then the functor
\[
	\bar{U} \colon \Sp_c(\mcal{D}) \to \Tel_c(\mcal{D})
\]
is an equivalence for every $\infty$-category $\mcal{D}$ presentably tensored over $\mcal{C}$.
\end{proposition}
\begin{proof}
It suffices to show that $\Tel_c(\mcal{D})$ has the same universal property with $\Sp_c(\mcal{D})$, that is:
\begin{enumerate}
\item $c$ acts as an equivalence on $\Tel_c(\mcal{D})$.
\item If $c$ acts as an equivalence on $\mcal{D}$, then $LF_c\colon\mcal{D}\to\Tel_c(\mcal{D})$ is an equivalence.
\end{enumerate}
(i) follows from Lemma \ref{lem:cTe3} and \cite[Proposition C.3]{BNT} thanks to the cyclic triviality.
(ii) is immediate from Lemma \ref{lem:cTe3}.
\end{proof}

\newpage

\section{Motivic spectra and fundamental stability}\label{MSp}

Recall that the $\infty$-category $\Sp$ of spectra is the formal inversion of $S^1$ in $\Ani_*$ and, more generally, that the stabilization $\Sp(\mcal{C})$ of a presentable $\infty$-category $\mcal{C}$ is the formal inversion of $S^1$ in $\mcal{C}_*$.
We develop the theory of motivic spectra in parallel, replacing the $\infty$-topos $\Ani$ with the $\infty$-topos $\St$ of Zariski sheaves on smooth schemes and $S^1$ with the projective line $\mbb{P}^1$.

We define the $\infty$-category $\SpP$ of motivic spectra to be the formal inversion of $\mbb{P}^1$ in $\St_*$ or, equivalently, the $\infty$-category of $\mbb{P}^1$-spectra in $\St_*$ in the sense of Definition \ref{def:cSp2};
\[
	\SpP := \SpP(\St_*) \simeq \St_*[(\mbb{P}^1)^{-1}].
\]
More generally, for an $\infty$-category $\V$ presentably tensored over $\St$, the $\infty$-category $\SpP(\V)$ of motivic spectra in $\V$ is well-defined to be the formal inversion of $\mbb{P}^1$ in $\V_*$.
The basic issue is that it is not clear if the $\infty$-category $\SpP(\V)$ is stable.
We define the notion of \textit{fundamental motivic spectra} and establish a stability for them (Theorem \ref{thm:FSt}).
Roughly speaking, a motivic spectrum is fundamental if and only if it satisfies Bass fundamental exact sequence, and then we employ the idea of Bass construction to prove the stability.

\subsection{Algebro-geometric conventions}\label{AGC}

We refer to \cite{SAG} for the theory of derived schemes.

\begin{p}
For a derived qcqs scheme $S$, let $\Sm_S$ denote the $\infty$-category of qcqs smooth derived $S$-schemes.
We suppose that $\Sm_S$ is endowed with the Zariski topology by default.
In the case $S=\Spec(\mbb{Z})$, the prefix/subscript $S$ is omitted from the notation, and the same applies below.
\end{p}

\begin{p}[Stack]
We refer to a sheaf of anima on $\Sm_S$ as an \textit{$S$-stack}.
Let $\St_S$ denote the $\infty$-topos of sheaves of anima on $\Sm_S$.
We endow $\St_S$ with the cartesian symmetric monoidal structure.
Then $\St_S$ is a presentably symmetric monoidal $\infty$-category.
\end{p}

\begin{p}[Projective line]
We suppose that the projective line $\mbb{P}^1$ is pointed by $\infty$.
We write
\[
	\SigmaP := \mbb{P}^1\otimes- \qquad
	\OmegaP := (-)^{\mbb{P}^1}
\]
for the operations on $\infty$-categories presentably tensored over $\St_*$.
\end{p}

\begin{p}[Moduli stack of vector bundles]
For a non-negative integer $n$, let $\Vect_n$ denote the moduli stack of vector bundles of rank $n$, which yields an $S$-stack for each qcqs derived scheme $S$.
Since the moduli stack $\Vect_n$ is left Kan extended from smooth schemes, the base change functor $\St\to\St_S$ carries $\Vect_n$ to $\Vect_n$.
For $n=1$, we write $\Pic:=\Vect_1$, which is the Picard stack.
We often regard $\Pic$ as an $\mbb{E}_\infty$-monoid in $\St_S$ with respect to tensor products of line bundles.
\end{p}

\begin{p}[Grassmannian]
For non-negative integers $n$ and $N$, the $n$-th grassmannian $\Gr_n(\mcal{O}^N)$ of $\mcal{O}^N$ classifies all quotients $\mcal{O}^N\twoheadrightarrow\mcal{E}$, where $\mcal{E}$ is a vector bundle of rank $n$.
The projection $\mcal{O}^{N+1}\to\mcal{O}^N$ discarding the last factor induces an immersion $\Gr_n(\mcal{O}^N)\to\Gr_n(\mcal{O}^{N+1})$.
We write $\Gr_n:=\colim_N\Gr_n(\mcal{O}^N)$ and regard it as an ind-scheme or stack.
We write $\mbb{P}^\infty:=\Gr_1$, which is the infinite projective space.
\end{p}

\subsection{Definition of motivic spectra}\label{DMS}

\begin{p}\label{p:DMS}
Let $\V$ be an $\infty$-category presentably tensored over $\St$ throughout.
We assume that $\V$ is compactly generated and that $\mbb{P}^1$ is compact in $\V$, i.e., $(-)^{\mbb{P}^1}\colon\V\to\V$ preserves filtered colimits.
We say that $\V$ is \textit{multiplicative} if $\V$ is a presentably symmetric monoidal $\infty$-category together with a symmetric monoidal left adjoint $\St\to\V$, which we denote by $(-)_\V$.
For a qcqs derived scheme $S$, we say that $\V$ is \textit{defined over $S$} if $\V$ is presentably tensored over $\St_S$.
\end{p}

\begin{remark}
The assumption that $\V$ is compactly generated and that $\mbb{P}^1$ is compact is required only for the validity of Brown representability theorem, which we use only in the proof of Theorem \ref{thm:FSt}.
We can remove those assumptions unless Theorem \ref{thm:FSt} is involved.
\end{remark}

\begin{definition}[Motivic spectrum]\label{def:DMS}
We define a \textit{motivic spectrum in $\V$} to be a $\mbb{P}^1$-spectrum in $\V_*$ in the sense of Definition \ref{def:cSp2}.
Accordingly, the presentably symmetric monoidal $\infty$-category
\[
	\SpP := \SpP(\St_*)
\]
is defined and $\SpP(\V):=\SpP(\V_*)$ is defined as an $\infty$-category presentably tensored over $\SpP$.
\end{definition}

\begin{remark}
By Proposition \ref{prop:cSp}, we have canonical equivalences
\[
	\SpP(\V) \simeq \SpP\otimes_\St\V \simeq \V_*[(\mbb{P}^1)^{-1}],
\]
where the tensor product is taken in $\Mod_\St(\Pr^L)$.
Moreover, $\SpP(\V)$ is compactly generated by Corollary \ref{cor:FPr}.
\end{remark}

\begin{remark}
Suppose that $\V$ is multiplicative.
Then $\SpP(\V)$ admits a unique presentably symmetric monoidal structure for which the evident functors $\V\to\SpP(\V)$ and $\SpP\to\SpP(\V)$ are symmetric monoidal.
In this case, we refer to a homotopy associative (co)algebra object in $\SpP(\V)$ as a \textit{motivic (co)ring spectrum in $\V$} and refer to an $\mbb{E}_k$-(co)algebra object as a \textit{motivic $\mbb{E}_k$-(co)ring spectrum in $\V$}.
\end{remark}

\begin{example}
For a qcqs derived $S$-scheme, we can take $\St_S$ as $\V$.
Then it is multiplicative and the presentably symmetric monoidal $\infty$-category
\[
	\SpP(S) := \SpP(\St_S)
\]
is defined.
We refer to a motivic spectrum in $\St_S$ as a \textit{motivic spectrum over $S$}.
In general, if $\V$ is defined over $S$, then $\SpP(\V)$ is presentably tensored over $\SpP(S)$.
\end{example}

\begin{notation}[Infinite suspension]
We write
\[
	\SigmaP^\infty \colon \V_* \rightleftarrows \SpP(\V) \colon \OmegaP^\infty \qquad
	s_+ \colon \SpP(\V) \rightleftarrows \SpP(\V) \colon s_-
\]
for the adjunctions in \ref{p:cSp2} ($s_+$ denotes the derived sift $Ls_+$ for simplicity).
For each integer $n$, we set
\[
	\SigmaP^{\infty-n} := (s_+)^{\circ n}\circ\SigmaP^\infty \qquad
	\OmegaP^{\infty-n} := \OmegaP^\infty\circ (s_-)^{\circ n}.
\]
Note that we have natural equivalences $\SigmaP\simeq s_-$ and $\OmegaP\simeq s_+$ as endofunctors of $\SpP(\V)$, cf.\ Lemma \ref{lem:cSp4}.
\end{notation}

\begin{notation}[Cohomology]\label{not:DMS}
Let $E,R$ be motivic spectra in $\V$ and $p,q,n$ integers with $2q-p\ge 0$.
We write
\[
	E(R) := \Map(R,E) \qquad
	E^{p,q}(R) := \pi_{2q-p}\Map(R,\SigmaP^qE) \qquad
	E^n(R) := E^{2n,n}(R).
\]
For an object $X$ in $\V$, we write $E(X):=E(\SigmaP^\infty X_+)$, and when $X$ is pointed, $\tilde{E}(X):=E(\SigmaP^\infty X)$.
\end{notation}

\begin{remark}
Let $E$ be a motivic spectrum in $\V$.
By definition, we have an isomorphism
\[
	E^{p,q}(\mbb{P}^1\otimes R) \simeq E^{p-2,q-1}(R)
\]
for every motivic spectrum $R$ in $\V$ and for integers $p,q$ with $2q-p\ge 0$.
This can be referred to as the \textit{$\mbb{P}^1$-suspension isomorphism}.
\end{remark}

\begin{p}[Change of coefficients]
Let $F\colon\V\to\V'$ be an $\St$-linear left adjoint between $\infty$-categories presentably tensored over $\St$.
Then we have an induced adjunction
\[
	F^* \colon \SpP(\V) \rightleftarrows \SpP(\V') \colon F_*,
\]
and $F^*$ is $\SpP$-linear.
If $\V$ and $\V'$ are multiplicative and $F\colon\V\to\V'$ is symmetric monoidal, then the induced left adjoint $F^*\colon\SpP(\V)\to\SpP(\V')$ is symmetric monoidal.
\end{p}

\begin{p}[Relation to $\mbb{A}^1$-local theory]
Let $S$ be a qcqs derived scheme.
Let $\St^{\mrm{Nis},\mbb{A}^1}_S$ be the full subcategory of $\St_S$ spanned by $\mbb{A}^1$-local Nisnevich sheaves.
Then Voevodsky's stable motivic homotopy category $\mrm{SH}(S)$ is defined as
\[
	\mrm{SH}(S) = \SpP(\St^{\mrm{Nis},\mbb{A}^1}_S).
\]
In particular, it is a localization of $\SpP(S)$ with respect to Nisnevich descent and $\mbb{A}^1$-homotopy invariance.
\end{p}

\subsection{Fundamental motivic spectra}\label{FMS}

\begin{p}
Consider the standard Zariski distinguished square
\[
\xymatrix{
	\mbb{G}_m \ar[r] \ar[d] & \mbb{A}^1 \ar[d] \\
	\mbb{A}^1 \ar[r] & \mbb{P}^1,
}
\]
which we regard as a square of pointed schemes, and let $\partial\colon\mbb{P}^1\to S^1\otimes\mbb{G}_m$ denote the boundary map in $\St_*$ (note that $S^1\otimes\mbb{G}_m\simeq*\sqcup_{\mbb{G}_m}*$).
\end{p}

\begin{definition}[Fundamental motivic spectrum]\label{def:FMS}
We say that a motivic spectrum $E$ in $\V$ is \textit{fundamental} if the map
\[
	\partial = \partial\otimes\id_E \colon \mbb{P}^1\otimes E \to S^1\otimes\mbb{G}_m\otimes E
\]
admits a retraction.
\end{definition}

\begin{lemma}\label{lem:FMS}
Let $E$ be a motivic spectrum in $\V$.
Then the following are equivalent:
\begin{enumerate}
\item $E$ is fundamental.
\item There exists an element $\nu\in E^{1,1}(\mbb{G}_m\otimes E)$ which lifts the identity $\id_E\in E^0(E)$ via the map
\[
	E^{1,1}(\mbb{G}_m\otimes E) \xrightarrow{\partial^*} E^{2,1}(\mbb{P}^1\otimes E) \simeq E^0(E).
\]
\item The map $\partial^*\colon E^{S^1\otimes\mbb{G}_m}\to E^{\mbb{P}^1}$ admits a section.
\item The motivic spectrum $\intMap(E,E)$ over $\Spec(\mbb{Z})$ is fundamental.
\end{enumerate}
Suppose that $\V$ is multiplicative and that $E$ is a motivic ring spectrum in $\V$, then these are further equivalent to the following:
\begin{enumerate}[resume]
\item The map $\partial\colon\mbb{P}^1\otimes E\to S^1\otimes\mbb{G}_m\otimes E$ admits a retraction as a morphism of left (or right) $E$-modules.
\item There exists an element $\nu\in\tilde{E}^{1,1}(\mbb{G}_{m,\V})$ which lifts the unit $\eta\in E^0(\mbf{1}_\V)$ via the map
\[
	\tilde{E}^{1,1}(\mbb{G}_{m,\V}) \xrightarrow{\partial^*} \tilde{E}^{2,1}(\mbb{P}^1_\V) \simeq E^0(\mbf{1}_\V).
\]
\end{enumerate}
\end{lemma}
\begin{proof}
By definition, we have
\[
	\pi_0\Map(S^1\otimes\mbb{G}_m\otimes E,\mbb{P}^1\otimes E) = E^{1,1}(\mbb{G}_m\otimes E)
\]
and this identification furnishes a one-to-one correspondence between retractions of $\partial$ and lifts of the identity $\id_E\in E^0(E)$.
This proves (i)$\Leftrightarrow$(ii).
Next note that we have equivalences
\[
\begin{split}
	\Map(S^1\otimes\mbb{G}_m\otimes E,\mbb{P}^1\otimes E)
	&\simeq \Map(E,s_-E^{S^1\otimes\mbb{G}_m}) \\
	&\simeq \Map(s_+E,E^{S^1\otimes\mbb{G}_m}) \\
	&\simeq \Map(E^{\mbb{P}^1},E^{S^1\otimes\mbb{G}_m}).
\end{split}
\]
This equivalence furnishes a one-to-one correspondence between retractions of $\partial$ and sections of $\partial^*$.
This proves (i)$\Leftrightarrow$(iii).

Note that $F:=\intMap(E,E)$ is a motivic ring spectrum over $\Spec(\mbb{Z})$ in a canonical way.
Then the condition (vi) for $F$ is identified with the condition (ii) for $E$.
Assuming (i)$\Leftrightarrow$(vi) for the moment, we see that $E$ is fundamental if and only if $F$ is fundamental, i.e., (i)$\Leftrightarrow$(iv).

Suppose that $\V$ is multiplicative and that $E$ is a motivic ring spectrum in $\V$.
Then the implications (v)$\Rightarrow$(i) and (iii)$\Rightarrow$(vi) are obvious, and thus it remains to show that (vi)$\Rightarrow$(v).
Suppose that we are given a lift $\nu\in\tilde{E}^{1,1}(\mbb{G}_{m,\V})$ as in (v).
Then $\nu$ is regarded as a map $\SigmaP^\infty(S^1\otimes\mbb{G}_m)\to\mbb{P}^1\otimes E$.
By taking the adjunction, we obtain a morphism $S^1\otimes\mbb{G}_m\otimes E\to\mbb{P}^1\otimes E$ of left (or right) $E$-modules, and it gives a retraction of the canonical map $\partial$.
This completes the proof.
\end{proof}

\begin{corollary}\label{cor:FMS}
Suppose that $\V$ is multiplicative.
Let $E$ be a fundamental motivic ring spectrum in $\V$.
Then every left or right $E$-module in $\SpP(\V)$ is fundamental.
\end{corollary}
\begin{proof}
This is immediate from the condition (v) in Lemma \ref{lem:FMS} for $E$.
\end{proof}

\begin{corollary}\label{cor:FMS2}
Let $F\colon\V\to\V'$ be an $\St$-linear left adjoint between $\infty$-categories presentably tensored over $\St$.
Then the induced functors
\[
	F^* \colon \SpP(\V) \to \SpP(\V') \qquad
	F_* \colon \SpP(\V') \to \SpP(\V) 
\]
preserve fundamental motivic spectra.
\end{corollary}
\begin{proof}
Note that $F^*$ preserves fundamental motivic spectra by definition and that $F_*$ preserves the condition (iii) in Lemma \ref{lem:FMS}.
\end{proof}

\begin{p}[Fundamental exact sequence]
Let $E$ be a fundamental motivic spectrum in $\V$.
Then, by Lemma \ref{lem:FMS}, we have a split exact sequence
\[
	0 \to E^{p+1,q+1}(\mbb{A}^1\otimes R)^{\times 2} \to E^{p+1,q+1}(\mbb{G}_m\otimes R) \xrightarrow{\partial^*} E^{p,q}(R) \to 0 
\]
for every motivic spectrum $R$ in $\V$ and for integers $p,q$ with $2q-p\ge 0$.
We refer to this sequence as the \textit{fundamental exact sequence}.
Warn that a priori this is an exact sequence of pointed sets when $2q-p=0$.
However, the next lemma says that this is canonically promoted to an exact sequence of abelian groups.
\end{p}

\begin{lemma}\label{lem:FMS2}
Let $E$ be a fundamental motivic spectrum in $\V$, $R$ a motivic spectrum in $\V$, and $q$ a non-negative integer.
Then $E^{2q-1,q}(R)$ is an abelian group and $E^{2q,q}(R)$ admits a natural abelian group structure which makes the fundamental exact sequence an exact sequence of abelian groups.
\end{lemma}
\begin{proof}
By definition, $E^{p,q}(R)$ is a group for $2q-p\ge 1$ and abelian for $2q-p\ge 2$.
The fundamental exact sequence implies that $E^{p,q}(R)$ is abelian for $2q-p=1$ as well.
In general, if we have a cartesian square
\[
\xymatrix{
	A \ar[r] \ar[d] & B \ar[d]^f \\
	C \ar[r]_g & D
}
\]
of pointed anima, then we have an exact sequence
\[
	\pi_1(B)\times\pi_1(C) \xrightarrow{f\cdot g} \pi_1(D) \xrightarrow{\partial} \pi_0(A).
\]
Moreover, if $x,y\in\pi_1(D)$ with $\partial x=\partial y$, then there exist $\beta\in B$ and $\gamma\in C$ such that $f(\beta)\cdot x=y\cdot g(\gamma)$.
We apply this observation to the fundamental exact sequence.
Then we see that $E^{2q,q}(R)$ inherits an abelian group structure from $E^{2q-1,q}(\mbb{G}_m\otimes R)$.
\end{proof}

\subsection{Fundamental stability}\label{FSt}

\begin{construction}\label{cons:FSt}
Let $E$ be a fundamental motivic spectrum in $\V$ and $R$ a motivic spectrum in $\V$.
We define abelian groups $E^{p,q}(R)$ for all integers $p,q$ as follows:
For $2q-p\ge 0$, it is as defined in Notation \ref{not:DMS} and Lemma \ref{lem:FMS2} and, for $2q-p<0$, it is defined by induction by the formula
\[
	E^{p,q}(R) := \coker(E^{p+1,q+1}(\mbb{A}^1\otimes R)^{\oplus 2} \to E^{p+1,q+1}(\mbb{G}_m\otimes R)).
\]
For $2q-p\ge 1$, we have a natural isomorphism
\[
	\delta \colon E^{p,q}(R) \xrightarrow{\sim} E^{p+1,q}(\Sigma R)
\]
and we extend it to all integers $p,q$ by induction by the commutative diagram
\[
\xymatrix{
	E^{p+1,q+1}(\mbb{A}^1\otimes R)^{\oplus 2} \ar[r] \ar[d]^\delta_\simeq &
		E^{p+1,q+1}(\mbb{G}_m\otimes R) \ar[r] \ar[d]^\delta_\simeq & E^{p,q}(R) \ar[r] \ar@{.>}[d]^\delta_\simeq & 0 \\
	E^{p+2,q+1}(\mbb{A}^1\otimes\Sigma R)^{\oplus 2} \ar[r] &
		E^{p+2,q+1}(\mbb{G}_m\otimes\Sigma R) \ar[r] & E^{p+1,q}(\Sigma R) \ar[r] & 0.
}
\]
Note that this construction is natural in $R$.
\end{construction}

\begin{lemma}\label{lem:FSt}
Let $E$ be a fundamental motivic spectrum in $\V$ and $R$ a motivic spectrum in $\V$.
Then the exact sequence
\[
	0 \to E^{p+1,q+1}(\mbb{A}^1\otimes R)^{\oplus 2} \to E^{p+1,q+1}(\mbb{G}_m\otimes R) \to E^{p,q}(R) \to 0
\]
is naturally split for all integers $p,q$.
\end{lemma}
\begin{proof}
Choose a section $\nu$ of the canonical map $\partial^*\colon E^{S^1\otimes\mbb{G}_m}\to E^{\mbb{P}^1}$.
Then it induces a split
\[
	\nu \colon E^{p,q}(R) \to E^{p+1,q+1}(\mbb{G}_m\otimes R)
\]
for $2q-p\ge 0$, which is natural in $R$, and we extend it to all integers $p,q$ by induction by the commutative diagram
\[
\xymatrix{
	E^{p+1,q+1}(\mbb{A}^1\otimes R)^{\oplus 2} \ar[r] \ar[d]^\nu &
		E^{p+1,q+1}(\mbb{G}_m\otimes R) \ar[r] \ar[d]^\nu & E^{p,q}(R) \ar[r] \ar@{.>}[d]^\nu & 0 \\
	E^{p+2,q+2}(\mbb{G}_m\otimes\mbb{A}^1\otimes R)^{\oplus 2} \ar[r] &
		E^{p+2,q+2}(\mbb{G}_m\otimes\mbb{G}_m\otimes R) \ar[r] & E^{p+1,q+1}(\mbb{G}_m\otimes R) \ar[r] & 0.
}
\]
Then it gives a desired split.
\end{proof}

\begin{lemma}\label{lem:FSt2}
Let $E$ be a fundamental motivic spectrum in $\V$.
Then, for each integer $q$, the family of functors $\{E^{p,q}\colon\SpP(\V)^\op\to\Ab\}_p$ together with isomorphisms $\delta\colon E^{p,q}\xrightarrow{\sim}E^{p+1,q}\circ\Sigma$ is a cohomology theory on $\SpP(\V)$ in the sense of \cite[1.4.1.6]{HA}.
\end{lemma}
\begin{proof}
We have to verify the following two properties:
\begin{enumerate}
\item $E^{p,q}$ preserves products.
\item For a cofiber sequence $A\to B\to C$ in $\SpP(\V)$, the induced sequence
\[
	E^{p,q}(C) \to E^{p,q}(B) \to E^{p,q}(A) 
\]
is exact.
\end{enumerate}
This is clear for $2q-p\ge 0$, and the general case follows by induction by the split exact sequence in Lemma \ref{lem:FSt}.
We remark that the splitting is important here; otherwise the induction step may not work.
\end{proof}

\begin{notation}
Let $\SpP(\V)^\fd$ denote the full subcategory of $\SpP(\V)$ spanned by fundamental motivic spectra in $\V$.
Note that $\SpP(\Sp(\V))^\fd$ makes sense by replacing $\V$ by its stabilization $\Sp(\V)$.
\end{notation}

\begin{theorem}\label{thm:FSt}
The adjunction
\[
	\Sigma^\infty \colon \SpP(\V) \rightleftarrows \SpP(\Sp(\V)) \colon \Omega^\infty
\]
restricts to an adjoint equivalence
\[
	\Sigma^\infty \colon \SpP(\V)^\fd \overset{\sim}{\rightleftarrows} \SpP(\Sp(\V))^\fd \colon \Omega^\infty.
\]
\end{theorem}
\begin{proof}
We first prove that the functor
\[
	\Omega^\fd := \Omega\vert_{\SpP(\V)^\fd} \colon \SpP(\V)^\fd \to \SpP(\V)^\fd
\]
is an equivalence.
Note that $\Omega^\fd$ is an equivalence if and only if it induces an equivalence on homotopy categories.
It follows from Lemma \ref{lem:FSt2} and the Brown representability (\cite[1.4.1.10]{HA}) that, for each fundamental motivic spectrum $E$ in $\V$, the functor $E^{1,0}\colon\SpP(\V)^\op\to\Ab$ is representable.
We claim that the assignment $E\mapsto E^{1,0}$ gives a homotopy inverse of $\Omega^\fd$.
Let us first verify that $F:=E^{1,0}$ is fundamental.
For this, note that we have a commutative diagram
\[
\xymatrix{
	E^{p+2,q+1}(\mbb{A}^1\otimes\mbb{P}^1\otimes R)^{\oplus 2} \ar[r] \ar[d]^\nu &
		E^{p+2,q+1}(\mbb{G}_m\otimes\mbb{P}^1\otimes R) \ar[r] \ar[d]^\nu & F^{p,q}(\mbb{P}^1\otimes R) \ar[r] \ar@{.>}[d]^\nu & 0 \\
	E^{p+1,q}(\mbb{A}^1\otimes\mbb{G}_m\otimes R)^{\oplus 2} \ar[r] &
		E^{p+1,q}(\mbb{G}_m\otimes\mbb{G}_m\otimes R) \ar[r] & F^{p-1,q}(\mbb{G}_m\otimes R) \ar[r] & 0
}
\]
for integers $p,q$ with $2q-p\ge 0$, where $\nu$ is a natural section supplied by the fundamentality of $E$.
Hence, we obtain a section $\nu\colon F^{\mbb{P}^1}\to F^{S^1\otimes\mbb{G}_m}$, which shows that $F$ is fundamental.
The assignment $E\mapsto E^{1,0}$ is clearly a section of $\Omega^\fd$.
To show that it is a retraction, we have to show that there is a natural equivalence $(\Omega E)^{1,0}\simeq E$ for a fundamental motivic spectrum $E$, and this follows from the commutative diagram
\[
\xymatrix{
	(\Omega E)^{2,1}(\mbb{A}^1\otimes R)^{\oplus 2} \ar[r] \ar[d]^\simeq & 
		(\Omega E)^{2,1}(\mbb{G}_m\otimes R) \ar[r] \ar[d]^\simeq & (\Omega E)^{1,0}(R) \ar[r] \ar@{.>}[d]^\simeq & 0 \\
	E^{1,1}(\mbb{A}^1\otimes R)^{\oplus 2} \ar[r] & 
		E^{1,1}(\mbb{G}_m\otimes R) \ar[r] & E^{0,0}(R) \ar[r] & 0,
}
\]
where each row is exact and every map is natural in $R$.

Next we show that the canonical functor
\[
	\SpP(\Sp(\V))^\fd \to \lim(\cdots\xrightarrow{\Omega}\SpP(\V)^\fd\xrightarrow{\Omega}\SpP(\V)^\fd\xrightarrow{\Omega}\SpP(\V)^\fd).
\]
is an equivalence.
In other words, a motivic spectrum $E$ in $\Sp(\V)$ is fundamental if and only if $\Omega^{\infty-i}E$ is fundamental for every $i$ as a motivic spectrum in $\V$.
The ``only if'' part is obvious.
To show the ``if'' part, suppose we are given a motivic spectrum $E$ in $\Sp(\V)$ such that $E^{(i)}:=\Omega^{\infty-i}E$ is fundamental for every $i$.
We have seen that $E^{(i)}\simeq (E^{(0)})^{i,0}$ for each $i$ and that a section $\nu_0\colon(E^{(0)})^{\mbb{P}^1}\to(E^{(0)})^{S^1\otimes\mbb{G}_m}$ induces a section $\nu_i\colon(E^{(i)})^{\mbb{P}^1}\to(E^{(i)})^{S^1\otimes\mbb{G}_m}$.
Then we have $\Omega\nu_i=\nu_{i-1}$ for each $i\ge 1$.
Therefore, we obtain a morphism
\[
	\nu \colon E^{\mbb{P}^1}\to E^{S^1\otimes\mbb{G}_m}
\]
in $\Sp(\SpP(\V))\simeq \SpP(\Sp(\V))$ such that $\partial^*\circ\nu$ is an equivalence.
By replacing $\nu$ if necessary, we conclude that $E$ is a fundamental motivic spectrum in $\Sp(\V)$.

By Corollary \ref{cor:FMS2}, the adjunction $(\Sigma^\infty,\Omega^\infty)$ induces an adjunction
\[
	\Sigma^\infty \colon \SpP(\V_S)^\fd \rightleftarrows \SpP(\Sp(\V))^\fd \colon \Omega^\infty.
\]
The above argument proves that the right adjoint $\Omega^\infty$ is an equivalence, and therefore we obtain a desired adjoint equivalence.
\end{proof}

\begin{remark}
By Theorem \ref{thm:FSt}, we can naturally regard a fundamental motivic spectrum $E$ in $\V$ as a motivic spectrum in $\Sp(\V)$.
In particular, the cohomology groups $E^{p,q}(R)$ are well-defined for all integers $p,q$ and they coincide with the groups defined in Construction \ref{cons:FSt}.
\end{remark}

\begin{corollary}\label{cor:FSt}
Suppose that $\V$ is multiplicative.
Let $E$ be a fundamental motivic ring spectrum in $\V$ and $R$ a motivic coring spectrum in $\V$.
Then
\[
	E^{*,*}(R) := \bigoplus_{p,q}E^{p,q}(R)
\]
forms a graded ring.
\end{corollary}

\begin{corollary}\label{cor:FSt2}
Suppose that $\V$ is multiplicative.
Let $E$ be a fundamental motivic $\mbb{E}_1$-ring spectrum.
Then the $\infty$-category $\LMod_E(\SpP(\V))$ is stable.
\end{corollary}
\begin{proof}
Since $\SpP(\V)^\fd$ is an ideal of $\SpP(\V)$, it inherits a symmetric monoidal structure and the induced functor
\[
	\LMod_E(\SpP(\V)^\fd) \to \LMod_E(\SpP(\V))
\]
is fully faithful.
Furthermore, it is essentially surjective, since every $E$-module in $\SpP(\V)$ is fundamental by Corollary \ref{cor:FMS}.
Hence, we have equivalences
\[
	\LMod_E(\SpP(\V)) \simeq \LMod_E(\SpP(\V)^\fd) \simeq \LMod_E(\SpP(\Sp(\V))^\fd) \simeq \LMod_E(\SpP(\Sp(\V))),
\]
where the second equivalence is by Theorem \ref{thm:FSt}.
Note that the last $\infty$-category is stable.
\end{proof}

\newpage

\section{Orientations and projective bundle formula}\label{OPB}

In this section, we develop a theory of orientation for motivic spectra.
We say that a motivic spectrum $E$ is orientable if the map
\[
	[\mcal{O}(1)]\otimes\id_E \colon \mbb{P}^1\otimes E\to \Pic\otimes E
\]
admits a retraction.
Note that if a motivic spectrum is orientable then it is fundamental.
We formulate projective bundle formula for oriented motivic spectra and relate it to \textit{elementary blowup excision}, which is an excision condition with respect to the blowup square
\[
\xymatrix{
	\mbb{P}^{n-1} \ar[r] \ar[d] & Q \ar[d] \\
	\{0\} \ar[r] & \mbb{A}^n.
}
\]
More precisely, we prove that an oriented motivic spectrum satisfies projective bundle formula if and only if it satisfies elementary blowup excision (Lemma \ref{lem:EBE2}).
This is convenient because elementary blowup excision is formulated without orientation nor $\mbb{P}^1$-spectrum structure.

\subsection{Orientation}\label{Ori}

\begin{p}
Let $\V$ be an $\infty$-category presentably tensored over $\St$ as before, cf.\ \ref{p:DMS}.
We assume that $\V$ is compactly generated and that $\mbb{P}^1$ is compact in $\V$. 
\end{p}

\begin{definition}[Orientation]\label{def:Ori}
Let $E$ be a motivic spectrum in $\V$.
We say that $E$ is \textit{orientable} if the map
\[
	[\mcal{O}(1)]\otimes\id_E \colon \mbb{P}^1\otimes E \to \Pic\otimes E
\]
admits a retraction.
When we choose such a retraction, we call it an \textit{orientation} of $E$.
An \textit{oriented motivic spectrum in $\V$} is a motivic spectrum in $\V$ equipped with an orientation.
\end{definition}

\begin{remark}
Let $E$ be an oriented motivic spectrum in $\V$.
Then the orientation is identified with a morphism in $\St_*$
\[
	c_1 \colon \Pic \to \OmegaP^\infty \intMap(E,\SigmaP E).
\]
In particular, for each line bundle $\mcal{L}$ on a stack $X$, we obtain a map
\[
	c_1(\mcal{L}) \colon E \to (\SigmaP E)^{X_+},
\]
which we call the \textit{first Chern class} of $\mcal{L}$.
When $E$ is defined over a qcqs derived scheme $S$, the first Chern class of a line bundle on an $S$-stack is well-defined.
Since the map $c_1$ is pointed, $c_1(\mcal{O})$ is the zero map.
The first Chern class $c_1(\mcal{O}(1))$ of the line bundle $\mcal{O}(1)$ on $\mbb{P}^1$ is identified with the canonical equivalence $E\xrightarrow{\sim}\SigmaP E^{\mbb{P}^1}$.
Note that $c_1$ is identified with the first Chern class of the universal line bundle on $\Pic$.
\end{remark}

\begin{lemma}\label{lem:Ori}
Let $E$ be a motivic spectrum in $\V$.
Then the following are equivalent:
\begin{enumerate}
\item $E$ is orientable.
\item There exists an element $c\in E^1(\Pic\otimes E)$ which lifts the identity $\id_E\in E^0(E)$ via the map
\[
	E^1(\Pic\otimes E) \xrightarrow{[\mcal{O}(1)]^*} E^1(\mbb{P}^1\otimes E) \simeq E^0(E).
\]
\item The map $E^\Pic\to E^{\mbb{P}^1}$ admits a section.
\item The motivic spectrum $\intMap(E,E)$ over $\Spec(\mbb{Z})$ is orientable.
\end{enumerate}
Suppose that $\V$ is multiplicative and that $E$ is a motivic ring spectrum in $\V$, then these are further equivalent to the following:
\begin{enumerate}[resume]
\item The map $\mbb{P}^1\otimes E\to\Pic\otimes E$ admits a retraction as a morphism of left (or right) $E$-modules.
\item There exists an element $c\in\tilde{E}^1(\Pic_\V)$ which lifts the unit $\eta\in E^0(\mbf{1}_\V)$ via the map
\[
	\tilde{E}^1(\Pic_\V) \xrightarrow{[\mcal{O}(1)]^*} \tilde{E}^1(\mbb{P}^1_\V) \simeq E^0(\mbf{1}_\V).
\]
\end{enumerate}
\end{lemma}
\begin{proof}
This is proved in the same way as Lemma \ref{lem:FMS}.
\end{proof}

\begin{remark}
Suppose that $\V$ is multiplicative.
Let $E$ be an orientable motivic ring spectrum in $\V$.
Then we can choose an orientation of $E$ as a morphism of right $E$-modules
\[
	c_1 \colon \Pic\otimes E\to\mbb{P}^1\otimes E.
\]
We call such an orientation a \textit{linear orientation}.
In other words, if $E$ is an orientable motivic ring spectrum, then it always has a linear orientation.
Note that a linear orientation is determined by a map
\[
	c_1 \colon \Pic \to \OmegaP^{\infty-1}E
\]
and that, for a line bundle $\mcal{L}$ on a stack $X$, the first Chern class $c_1(\mcal{L})\colon E\to(\SigmaP E)^{X_+}$ is reconstructed by the left multiplication by $c_1(\mcal{L})\in E^1(X)$.
All orientations we choose for orientable motivic ring spectra will be linear orientations and will be referred to simply as orientations unless there is a possibility of confusion.
\end{remark}

\begin{remark}
Let $E$ be an orientable motivic ring spectrum.
Then every left $E$-module $M$ in $\SpP(\V)$ is orientable and a linear orientation $c_1$ of $E$ induces an orientation of $M$ by
\[
	c_1\otimes\id_M \colon \Pic\otimes E\otimes_EM \to \mbb{P}^1\otimes E\otimes_EM.
\]
Furthermore, every orientation $c_1$ of an orientable motivic spectrum $F$ in $\V$ arises in this way;
indeed $c_1$ is induced from a linear orientation of the motivic ring spectrum $\intMap(F,F)$ over $\Spec(\mbb{Z})$.
\end{remark}

\begin{lemma}\label{lem:Ori2}
An orientable motivic spectrum in $\V$ is fundamental.
\end{lemma}
\begin{proof}
Indeed, the map $[\mcal{O}(1)]\colon\mbb{P}^1\to\Pic$ factors through the canonical map $\partial\colon\mbb{P}^1\to S^1\otimes\mbb{G}_m$.
\end{proof}

\begin{lemma}\label{lem:Ori3}
Let $S$ be a qcqs derived scheme, $E$ be an oriented motivic ring spectrum over $S$, and $X\in\Sm_S$.
Suppose we are given line bundles $\mcal{L}_1,\dotsc,\mcal{L}_n,\mcal{L}_1',\dotsc,\mcal{L}_n'$ on $X$ and an open covering $\{U_1,\dotsc,U_n\}$ of $X$ such that $\mcal{L}_i\vert_{U_i}\simeq\mcal{L}_i'\vert_{U_i}$.
Then
\[
	\prod_{i=1}^n(c_1(\mcal{L}_i)-c_1(\mcal{L}_i')) = 0
\]
in $E^*(X)$.
In particular:
\begin{enumerate}[label=\upshape{(\roman*)}]
\item $c_1(\mcal{L})$ is nilpotent in $E^*(X)$ for every line bundle $\mcal{L}$ on $X$.
\item $c_1(\mcal{O}(1))^{n+1}=0$ in $E^*(\mbb{P}^n)$.
\end{enumerate}
\end{lemma}
\begin{proof}
Since $\gamma_i:=c_1(\mcal{L}_i)-c_1(\mcal{L}_i')$ is sent to zero in $E^1(U_i)$ for each $i$, it lifts to
\[
	\tilde{\gamma}_i \in E^1(X,U_i) := \pi_0\fib(\SigmaP E(X)\to \SigmaP E(U_i)).
\]
Therefore, $\prod\gamma_i$ lifts to
\[
	\prod_{i=1}^n\tilde{\gamma}_i \in E^n(X,\bigcup U_i) = E^n(X,X) = 0
\]
and we conclude $\prod\gamma_i=0$.
\end{proof}

\begin{p}[Relation to pbf-local sheaves with transfers]
Let $S$ be a qcqs derived scheme.
Let $\Sh^\mrm{tr}_\pbf(\Sch_S)$ be the $\infty$-category of pbf-local sheaves with transfers in the sense of \cite{AI}.
Then there is a canonical symmetric monoidal left adjoint $\St_{S*}\to\Sh^\mrm{tr}_\pbf(\Sch_S)$, which carries $\mbb{P}^1$ to an invertible object.
Therefore, we obtain a commutative diagram
\[
\xymatrix{
	\SpP(\Sh^\mrm{tr}_\pbf(\Sch_S)) \ar[r] \ar[d]^{\OmegaP^\infty}_\simeq & \SpP(S) \ar[d]^{\OmegaP^\infty} \\
	\Sh^\mrm{tr}_\pbf(\Sch_S) \ar[r] & \St_S.
}
\]
The first Chern class of the universal line bundle on $\Pic$ is defined for pbf-local sheaves with transfers as in the proof of \cite[Lemma 3.2]{AI}, which gives an orientation of the associated motivic spectra in the sense of Definition \ref{def:Ori}.
\end{p}

\subsection{Projective bundle formula}\label{PBF}

\begin{definition}[Projective bundle formula]\label{def:PBF}
Let $E$ be an oriented motivic spectrum in $\V$.
We say that $E$ \textit{satisfies projective bundle formula} if the map
\[
	\sum_{i=1}^nc_1(\mcal{O}(1))^i \colon \bigoplus_{i=1}^n\SigmaP^{-i}E \to E^{\mbb{P}^n}
\]
is an equivalence for every $n\ge 1$.
\end{definition}

\begin{remark}
An oriented motivic spectrum $E$ in $\V$ satisfies projective bundle formula if and only if the map
\[
	\sum_{i=1}^nc_1(\mcal{O}(1))^i \colon \bigoplus_{i=1}^nE^{p-2i,q-i}(X) \to E^{p,q}(\mbb{P}^n\otimes X_+)
\]
is an equivalence of spectra for every $p,q,n$ and $X\in\V$.
It is because that $\SpP(\V)$ is generated under colimits by $\SigmaP^{\infty-i}X_+$ for $X\in\V$ and $i\ge 0$, cf.\ Lemma \ref{lem:FPr2}.
\end{remark}

\begin{lemma}\label{lem:PBF}
Suppose that $\V$ is defined over a qcqs derived scheme $S$.
Let $\mcal{E}$ be a vector bundle of rank $r$ on an $S$-stack $X$ and $E$ an oriented motivic spectrum in $\V$ which satisfies projective bundle formula.
Then the map
\[
	\sum_{i=0}^{r-1}c_1(\mcal{O}(1))^i \colon \bigoplus_{i=0}^{r-1}\SigmaP^{-i}E^{X_+} \to E^{\mbb{P}(\mcal{E})_+}
\]
is an equivalence.
\end{lemma}
\begin{proof}
We reduce to the case $X$ is representable by a limit argument and reduce to the case $\mcal{E}$ is trivial by Zariski descent.
Then it is immediate from the definition.
\end{proof}

\begin{corollary}\label{cor:PBF}
Suppose that $\V$ is multiplicative.
Let $E$ be an oriented motivic ring spectrum in $\V$ and $R$ a motivic coring spectrum in $\V$.
Assume that $E$ satisfies projective bundle formula and that $\xi:=c_1(\mcal{O}(1))$ is in the center of the ring $E^{*,*}(\mbb{P}^n_+\otimes R)$.
Then we have a ring isomorphism
\[
	E^{*,*}(\mbb{P}^n_+\otimes R) \simeq E^{*,*}(R)[\xi]/\xi^{n+1}.
\]
\end{corollary}
\begin{proof}
This follows from Lemma \ref{lem:Ori3} and Lemma \ref{lem:PBF}.
\end{proof}

\subsection{Elementary blowup excision}\label{EBE}

\begin{definition}[Elementary blowup excision]\label{def:EBE}
Let $F$ be a presheaf on $\Sm$.
We say that $F$ \textit{satisfies elementary blowup excision} if $F$ carries the blowup square
\[
\xymatrix{
	\mbb{P}^{n-1}_X \ar[r] \ar[d] & Q \ar[d] \\
	\{0\}_X \ar[r] & \mbb{A}^n_X
}
\]
to a cartesian square for every $X\in\Sm$ and $n\ge 1$.
\end{definition}

\begin{notation}
Let $\St^\ex$ denote the full subcategory of $\St$ spanned by sheaves satisfying the elementary blowup excision.
For an $\infty$-category $\V$ presentably tensored over $\St$, we write $\V^\ex:=\V\otimes_\St\St^\ex$.
Then $\V^\ex$ is an accessible localization of $\V$.
We say that an object in $\V$ \textit{satisfies elementary blowup excision} if it belong to $\V^\ex$.
Note that $\SpP(\V)^\ex\simeq\SpP(\V^\ex)$.
\end{notation}

\begin{remark}
Note that the inclusion $\V^\ex\to\V$ preserves filtered colimits, and thus the localization $\V\to\V^\ex$ preserves compact objects.
In particular, if $\V$ is compactly generated and $\mbb{P}^1$ is compact in $\V$, then the same holds for $\V^\ex$.
\end{remark}

\begin{lemma}\label{lem:EBE}
Let $E$ be a motivic spectrum in $\V$.
Then the following are equivalent:
\begin{enumerate}
\item $E$ satisfies elementary blowup excision.
\item $\OmegaP^{\infty-i}E$ satisfies elementary blowup excision for every $i\ge 0$.
\item For every $n$, the square
\[
\xymatrix{
	E^{\mbb{A}^n_+} \ar[r] \ar[d] & E^{Q_+} \ar[d] \\
	E \ar[r] & E^{\mbb{P}^{n-1}_+}
}
\]
is cartesian.
\end{enumerate}
\end{lemma}
\begin{proof}
Let $\chi_n$ denote the canonical map $Q\oplus_{\mbb{P}^{n-1}}\{0\}\to\mbb{A}^n$ in $\St$.
Then $\SpP(\V)^\ex$ is a localization of $\SpP(\V)$ with respect to the maps $R\otimes\chi_n$ for all $n$ and motivic spectra $R$ in $\V$; from which the equivalence (i)$\Leftrightarrow$(iii) follows.
Since $\SpP(\V)$ is generated under colimits by $\SigmaP^{\infty-i}X_+$ for $X\in\V$ and $i\ge 0$ (Lemma \ref{lem:FPr2}), we see that $\SpP(\V)^\ex$ is a localization of $\SpP(\V)$ with respect to the maps $\SigmaP^{\infty-i}(\chi_n\otimes X_+)$ for all $i,n$ and $X\in\V$; from which the equivalence (i)$\Leftrightarrow$(ii) follows.
\end{proof}

\begin{lemma}\label{lem:EBE2}
Let $E$ be an oriented motivic spectrum in $\V$.
Then $E$ satisfies projective bundle formula if and only if $E$ satisfies elementary blowup excision.
\end{lemma}
\begin{proof}
Consider the blowup square
\[
\xymatrix{
	D \ar[r] \ar[d] & Q \ar[d]^\pi \\
	\{0\} \ar[r] & \mbb{P}^n.
}
\]
There is a canonical projection $q\colon Q\to\mbb{P}^{n-1}$ which makes $Q$ a $\mbb{P}^1$-bundle.
The associated twisting sheaf $\mcal{O}_q(1)$ is isomorphic to $\pi^*\mcal{O}_{\mbb{P}^n}(1)$.
Since $\mcal{O}_q(1)$ is trivial in a neighborhood of $D$ and isomorphic to $q^*\mcal{O}_{\mbb{P}^{n-1}}(1)$ outside $D$, we have
\[
	c_1(\mcal{O}_q(1))\cdot q^*c_1(\mcal{O}_{\mbb{P}^{n-1}}(1)) = c_1(\mcal{O}_q(1))^2
\]
in $\intMap(E,E)^*(Q)$ by Lemma \ref{lem:Ori3}.
It follows that the diagram
\[
\xymatrix@C+2.5pc{
	\SigmaP^{-1}E(\mbb{P}^{n-1}_+\otimes R) \ar[r]^-{c_1(\mcal{O}_q(1))\cdot q^*}_-\simeq & E(Q_+\otimes R,D_+\otimes R) \\
	\bigoplus_{i=0}^{n-1}\SigmaP^{-i-1}E(R) \ar[u]^{\sum_{i=0}^{n-1}c_1(\mcal{O}_{\mbb{P}^{n-1}}(1))^i}
		\ar[r]^-{\sum_{i=0}^{n-1}c_1(\mcal{O}_{\mbb{P}^n}(1))^{i+1}} & E(\mbb{P}^n\otimes R) \ar[u]^{\pi^*}
}
\]
is commutative.
The top horizontal map is an equivalence in general.
Then the assertion is immediate from this diagram.
\end{proof}

\begin{lemma}\label{lem:EBE3}
Suppose that $\V$ is multiplicative.
Let $E$ be a homotopy commutative motivic ring spectrum in $\V$ and $R$ a homotopy cocommutative motivic coring spectrum in $\V$.
Assume that $E$ is orientable and satisfies elementary blowup excision.
Then, for $x\in E^{p,q}(R)$ and $y\in E^{p',q'}(R)$, we have
\[
	xy = (-1)^{pp'}yx.
\]
in $E^{p+p',q+q'}(R)$.
\end{lemma}
\begin{proof}
Choose an orientation of $E$.
Then it satisfies projective bundle formula by Lemma \ref{lem:EBE2}.
Let $\tau\in E^0(\mbf{1})\simeq E^2(\mbb{P}^1\otimes\mbb{P}^1)$ be the class of the permutation $\mbb{P}^1\otimes\mbb{P}^1\to\mbb{P}^1\otimes\mbb{P}^1$.
Then, for $x\in E^{p,q}(R)$ and $y\in E^{p',q'}(R)$, we have in general
\[
	xy = (-1)^{pp'}\tau^{qq'}yx.
\]
Hence, it suffices to show that $\tau=1$.
Note that we have
\[
	c_1(\mcal{O}(1))^2 = \tau c_1(\mcal{O}(1))^2
\]
in $E^2(\mbb{P}^2)$.
Then the projective bundle formula implies $\tau=1$.
\end{proof}

\begin{remark}
In particular, if $E$ and $R$ are homotopy (co)commutative, then the centrality assumption in Corollary \ref{cor:PBF} is always satisfied.
\end{remark}

\newpage

\section{Cohomology of the moduli stack of vector bundles}\label{Vec}

The goal of this section is to calculate the cohomology of the moduli stack $\Vect_n$ of rank $n$ vector bundles.
Let $E$ be a homotopy commutative oriented motivic ring spectrum which satisfies projective bundle formula.
Then we prove a ring isomorphism 
\[
	E^{*,*}(\Vect_{n,X}) \simeq E^{*,*}(X)[[c_1,\dotsc,c_n]]
\]
for every stack $X$ (Corollary \ref{cor:Che2}).
This is a refinement of the main theorem in \cite{AI}, which assumes the existence of transfers.
The proof is mostly parallel to that of \cite{AI}.\footnote{A gap has been found in the proof of \cite{AI}, but it is corrected in this section.}
The main new aspect is Construction \ref{cons:Lif}, which works thanks to Lemma \ref{lem:Ori3}.

Along the way, we develop a theory of Chern classes and formal group laws and establish their standard properties such as splitting principle (Lemma \ref{lem:Che}).
Those facts have been well known for special types of cohomology theory as already proved in \cite{SAG6}, but our results generalize them completely in a way that only depends on projective bundle formula.

\subsection{Pretheory}\label{PTh}

In order to prove the main results of this section, it is sometimes necessary to work on $\mbb{P}^1$-spectra in presheaves.
In this subsection, we introduce some auxiliary notions related to this.
These notions are used only in this section and are discussed with the minimum generality necessary for our purpose.

\begin{notation}
We consider the $\infty$-category $\PSh(\Sm;\Sp_{\ge 0})$ of presheaves of connective spectra on $\Sm$ and let
\[
	\SpP^\pre := \SpP(\PSh(\Sm;\Sp_{\ge 0})).
\]
Then $\SpP^\fd$ is a full subcategory of $\SpP^\pre$ as
\[
	\SpP^\fd \simeq \SpP(\Sh(\Sm;\Sp_{\ge 0}))^\fd \hookrightarrow \SpP^\pre,
\]
where the equivalence is by Theorem \ref{thm:FSt}.
The coefficients were made connective spectra for technical reasons for later use.
\end{notation}

\begin{notation}
We write
\[
	\OmegaP^\infty \colon \SpP^\pre \to \PSh(\Sm).
\]
for the right adjoint that extends $\OmegaP^\infty$ on $\SpP^\fd$.
For $E\in\SpP^\pre$ and $X\in\PSh(\Sm)$, we set
\[
	E(X) := \Map(X,\OmegaP^\infty E) \qquad
	E^{p,q}(X) := \pi_{2q-p}\Map(X,\OmegaP^\infty E).
\]
These are compatible with Notations \ref{not:DMS} when applied to $E\in\SpP^\fd$.
\end{notation}

\begin{definition}\label{def:PTh}
Let $E$ be an oriented motivic ring spectrum and $M$ a left $E$-module in $\SpP^\pre$.
We say that $M$ \textit{satisfies projective bundle formula} if the map
\[
	\sum_{i=1}^n c_1(\mcal{O}(1))^i \colon \bigoplus_{i=1}^n \SigmaP^{-i}M \to M^{\mbb{P}^n}
\]
is an equivalence for every $n\ge 1$.
\end{definition}

\begin{remark}\label{rem:PTh}
Let $E$ be an oriented motivic $\mbb{E}_1$-ring spectrum.
Then we define $\LMod^\pbf_E(\SpP^\pre)$ to be the full subcategory of $\LMod_E(\SpP^\pre)$ spanned left $E$-modules which satisfy projective bundle formula.
Then $\LMod^\pbf_E(\SpP^\pre)$ is an accessible localization of $\LMod_E(\SpP^\pre)$ and we denote the localization functor by
\[
	L_\pbf \colon \LMod_E(\SpP^\pre) \to \LMod^\pbf_E(\SpP^\pre).
\]
It follows from Lemma \ref{lem:EBE2} that the $\infty$-category $\LMod_E(\SpP^\ex)$ is identified with the full subcategory of $\LMod^\pbf_E(\SpP^\pre)$ spanned by $E$-modules which satisfy Zariski descent.
In particular, the composite localization
\[
	\LMod_E(\SpP^\pre) \to \LMod^\pbf_E(\SpP^\pre) \to \LMod_E(\SpP^\ex)
\]
does not depend on the module structure and is induced from the functor $\SpP^\pre\to\SpP^\ex$, which is the localization with respect to Zariski descent and elementary blowup excision followed by the infinite delooping.
\end{remark}

\subsection{Lifting lemma}\label{Lif}

\begin{p}[Globally generated vector bundle]
For $n\ge 1$, we define $\Vect_n^\flat$ to be the subpresheaf of $\Vect_n$ spanned by globally generated vector bundles of rank $n$.
We write $\Pic^\flat:=\Vect_1^\flat$.
Then $\Pic^\flat$ inherits the $\mbb{E}_\infty$-monoid structure from $\Pic$ and $\Vect^\flat$ is a module over $\Pic^\flat$.
Note that the inclusion
\[
	\Vect_n^\flat \to \Vect_n
\]
is a Zariski local equivalence.
\end{p}

\begin{notation}
For $n\ge 0$, we set
\[
	\bar{\Delta}^n := \Proj\biggl(\frac{\mbb{Z}[U,T_0,\dotsc,T_n]}{U-\sum_{i=0}^nT_i}\biggr).
\]
The assignment $n\mapsto\bar{\Delta}^n$ forms a semi-cosimplicial scheme $\bar{\Delta}^\bullet$ in a standard way.
Note that, for each $l\ge 1$, the twisting sheaves $\mcal{O}(l)$ on $\bar{\Delta}^*$ define a point of the semi-simplicial presheaf $\Pic^\flat(\bar{\Delta}^\bullet)$.
\end{notation}

\begin{lemma}\label{lem:Lif}
Let $k,n$ be non-negative integers.
Then the diagram in $\PSh(\Sm)$
\[
\xymatrix@C+1pc{
	& \lvert \cosk_k(\Gr_n^{\bar{\Delta}^\bullet}) \rvert \ar[d] \\
	\Vect_n^\flat \ar[r]^-{\mcal{O}(k+1)\otimes p^*} \ar@{.>}@/^1pc/[ru] & \lvert \cosk_k(\Vect_n^{\flat\bar{\Delta}^\bullet}) \rvert
}
\]
admits a lift as indicated, where $p^*$ denotes the pullback along the obvious projection.
\end{lemma}
\begin{proof}
This is a refinement of \cite[Lemma 3.3 (A)]{AI}.
As in the proof there, it is reduced to the following lifting property.
Let $\mcal{E}$ a globally generated vector bundle of rank $n$ on $X\in\Sm$.
Suppose that we are given a surjection $\alpha\colon\mcal{O}^{\oplus\infty}_{\partial\bar{\Delta}^m_X}\to\mcal{E}(k+1)\vert_{\partial\bar{\Delta}^m_X}$ for $1\le m\le k$.
Then we would like to show that it lifts to a surjection $\alpha'\colon\mcal{O}^{\oplus\infty}_{\bar{\Delta}^m_X}\to\mcal{E}(k+1)$.

Since the fiber of $\mcal{E}(k+1)\to\mcal{E}(k+1)\vert_{\partial\bar{\Delta}^m_X}$ is $\mcal{E}(k-m)$ and it is globally generated, if we could find any lift $\alpha'\colon\mcal{O}^{\oplus\infty}_{\bar{\Delta}^m_X}\to\mcal{E}(k+1)$ of $\alpha$ then we can add extra sections to ensure that $\alpha'$ is surjective.
Let $p$ be the projection $\bar{\Delta}^m_X\to X$ and $i$ the inclusion $\partial\bar{\Delta}^m_X\to\bar{\Delta}^m_X$.
Then the map
\[
	p_*\mcal{E}(k+1) = \mcal{E}\otimes(\mcal{O}_X[T_1,\dotsc,T_m])_{(k+1)} \to
	p_*i_*\mcal{E}(k+1) = \mcal{E}\otimes(\mcal{O}_X[T_1,\dotsc,T_m]/T_1\dotsc T_m)_{(k+1)}
\]
has a section as $\mcal{O}_X$-modules, and thus it remains surjective after taking the global section functor $\Gamma(X,-)$.
This proves the existence of $\alpha'$ and finishes the proof.
\end{proof}

\begin{construction}\label{cons:Lif}
Let $E$ be an oriented motivic ring spectrum and $M$ a left $E$-module in $\SpP^\pre$.
The first Chern class $\xi:=c_1(\mcal{O}(1))\in E^1(\bar{\Delta}^\bullet_+)$ is well-defined and we have a morphism of semi-simplicial objects in $\SpP^\pre$
\[
	\xi^n \colon \SigmaP^{-n}M \to M^{\bar{\Delta}^\bullet_+},
\]
where the left hand side is a constant semi-simplicial object.
By Lemma \ref{lem:Ori3}, the map factors through the $n$-th b\^ete truncation
\[
\xymatrix{
	\SigmaP^{-n}M \ar[r]^{\xi^n} \ar[d] & M^{\bar{\Delta}^\bullet_+} \\
	\sigma_{\ge n}(\SigmaP^{-n}M). \ar@{.>}[ru] &
}
\]
By taking the coproduct with respect to all $n\ge 0$, we obtain a map
\[
	\sum_{n=0}^\bullet \xi^n \colon \bigoplus_{n=0}^\bullet \SigmaP^{-n}M \to M^{\bar{\Delta}^\bullet_+},
\]
which is a level-wise equivalence if $M$ satisfies projective bundle formula.
We set
\[
	M^{\bar{\Delta}^\bullet_+}/M^{\bar{\Delta}^\bullet_{\infty+}}
	:= \cofib\biggl( \bigoplus_{n=1}^\bullet \SigmaP^{-n}M \xrightarrow{\sum_{n=1}^\bullet \xi^n} M^{\bar{\Delta}^\bullet_+} \biggr).
\]
Then the canonical map
\[
	M \to M^{\bar{\Delta}^\bullet_+}/M^{\bar{\Delta}^\bullet_{\infty+}}
\]
is a level-wise equivalence if $M$ satisfies projective bundle formula.
\end{construction}

\begin{p}[Truncation]
For each $k\ge 0$, the truncation $(-)_{\le k}\colon\Sp_{\ge 0}\to(\Sp_{\ge 0})_{\le k}$ is symmetric monoidal.
We set
\[
	(\SpP^\pre)_{\le k} := \SpP(\PSh(\Sm;(\Sp_{\ge 0})_{\le k})).
\]
Then the associated localization
\[
	(-)_{\le k} \colon \SpP^\pre \to (\SpP^\pre)_{\le k}
\]
is just the component-wise $k$-truncation.
Note that if a left $E$-module $M$ in $\SpP^\pre$ satisfies projective bundle formula then so does $M_{\le k}$.
\end{p}

\begin{lemma}\label{lem:Lif2}
Let $E$ be an oriented motivic $\mbb{E}_1$-ring spectrum.
Then, for every $n,k\ge 0$, the canonical map
\[
	(L_\pbf(\Gr_{n+}\otimes^\pre E))_{\le k} \to (L_\pbf(\Vect_{n+}^\flat\otimes^\pre E))_{\le k}
\]
admits a section, where $\otimes^\pre$ denotes the tensor product in $\SpP^\pre$.
\end{lemma}
\begin{proof}
In the proof, we write $\otimes:=\otimes^\pre$ for simplicity.
By Construction \ref{cons:Lif}, for each left $E$-module $M$ which satisfies projective bundle formula, we have a morphism of semi-simplicial objects
\[
	M^{\bar{\Delta}^\bullet_+} \to M^{\bar{\Delta}^\bullet_+}/M^{\bar{\Delta}^\bullet_{\infty+}} \simeq M.
\]
Combining it with Lemma \ref{lem:Lif}, we obtain a diagram in $\PSh(\Sm)$
\[
\xymatrix{
	& & \lvert\cosk_k(\OmegaP^\infty(L_\pbf(\Gr_{n+}\otimes E))^{\bar{\Delta}^\bullet_+})\rvert \ar[d] \ar[r] &
		\OmegaP^\infty(L_\pbf(\Gr_{n+}\otimes E))_{\le k} \ar[d] \\
	\Vect^\flat_n \ar[rr]^-{\mcal{O}(k+1)\otimes p^*} \ar@/^1pc/[rru] & & 
		\lvert\cosk_k(\OmegaP^\infty(L_\pbf(\Vect^\flat_{n+}\otimes E))^{\bar{\Delta}^\bullet_+})\rvert \ar[r] &
			 \OmegaP^\infty(L_\pbf(\Vect^\flat_{n+}\otimes E))_{\le k}.
}
\]
By taking the adjunctions, we obtain a diagram in $\SpP^\pre$
\[
\xymatrix{
	& (L_\pbf(\Gr_{n+}\otimes E))_{\le k} \ar[d] \\
	(L_\pbf(\Vect^\flat_{n+}\otimes E))_{\le k} \ar[r] \ar@/^1pc/[ru] &
		(L_\pbf(\Vect^\flat_{n+}\otimes E))_{\le k}.
}
\]
Here we used the fact that truncation preserves projective bundle formula.
Hence, it remains to show that the bottom map is homotopic to the identity.
By construction, it suffices to show that the composite
\[
	\Vect^\flat_n
		\xrightarrow{\mcal{O}(k+1)\otimes p^*} \OmegaP^\infty(L_\pbf(\Vect^\flat_{n+}\otimes E))^{\bar{\Delta}^k_+}
		\to \frac{\OmegaP^\infty(L_\pbf(\Vect^\flat_{n+}\otimes E))^{\bar{\Delta}^k_+}}
			{\OmegaP^\infty(L_\pbf(\Vect^\flat_{n+}\otimes E))^{\bar{\Delta}^k_{\infty+}}}
		\xleftarrow{\sim} \OmegaP^\infty L_\pbf(\Vect^\flat_{n+}\otimes E)
\]
is homotopic to the canonical map.
An inverse of the last equivalence is given by the pullback along the inclusion $i\colon\{*\}\to\bar{\Delta}^k$ for some point $*\in\bar{\Delta}^k$ not meeting $\bar{\Delta}^k_\infty$.
Since $i^*\mcal{O}(k+1)$ is trivial, the assertion follows.
This completes the proof.
\end{proof}

\begin{remark}\label{rem:Lif}
The following variant of Lemma \ref{lem:Lif2} will be used later.
Let $S$ be a small set of morphisms in $\LMod_E(\SpP^\pre)$ such that $S$-local objects are stable under truncations and contains all $L_\pbf$-local objects and that $S$-equivalences are stable under $X\otimes^\pre-$ for every $X\in\SpP^\pre$.
Then the proof of Lemma \ref{lem:Lif2} goes through and the canonical map
\[
	(L_S(\Gr_{n+}\otimes^\pre E))_{\le k} \to (L_S(\Vect_{n+}^\flat\otimes^\pre E))_{\le k}
\]
admits a section, where $L_S$ denotes the Bousfield localization with respect to $S$.
\end{remark}

\subsection{Cohomology of the Picard stack}\label{Pic}

\begin{p}
Let $\V$ be an $\infty$-category presentably tensored over $\St$ as before, cf.\ \ref{p:DMS}.
We assume that $\V$ is compactly generated and that $\mbb{P}^1$ is compact in $\V$. 
Let $L_\ex$ denote the localization $\SpP(\V)\to\SpP(\V^\ex)$.
\end{p}

\begin{proposition}\label{prop:Pic}
Let $E$ be an orientable motivic spectrum in $\V$.
Then the canonical map
\[
	L_\ex(\mbb{P}^\infty_+\otimes E) \to L_\ex(\Pic_+\otimes E)
\]
is an equivalence, where the tensor products are taken in $\SpP(\V)$.
\end{proposition}
\begin{proof}
By replacing $E$ by $\intMap(E,E)$ and by choosing the orientation, we may assume that $E$ is an oriented motivic $\mbb{E}_1$-ring spectrum over $\Spec(\mbb{Z})$.
A little stronger, we prove that the canonical map
\[
	L_\pbf(\mbb{P}^\infty_+\otimes^\pre E) \to L_\pbf(\Pic^\flat_+\otimes^\pre E)
\]
is an equivalence in $\SpP^\pre$.
Thanks to Lemma \ref{lem:Lif2}, it suffices to show that the map admits a retraction.
We set $E':=L_\pbf(\mbb{P}^\infty\otimes^\pre E)$.
We claim that the map
\[
	\prod_{i=0}^\infty c_1^i \colon E'(\mbb{P}^\infty) \to E'(\Pic^\flat)
\]
is well-defined, i.e., $E'(\Pic^\flat)$ is a $c_1$-complete $E(\Pic^\flat)$-module.
Note that $E'(\Pic^\flat)$ is a limit of $E'(X)$ with $X\in\Sm$ and that $c_1$ is nilpotent in $E'(X)$ by Lemma \ref{lem:Ori3}.
Since a limit of complete modules is complete, the claim follows.
Now we have a commutative diagram 
\[
\xymatrix{
	\prod_{i=1}^\infty \SigmaP^{-i}E'(\Spec(\mbb{Z})) \ar[rd]^\simeq \ar[d]_{\prod_{i=1}^\infty c_1^i} & \\
	E'(\Pic^\flat) \ar[r] & E'(\mbb{P}^\infty)
}
\]
and the diagonal arrow is an equivalence by the projective bundle formula.
In particular, the canonical map $\mbb{P}^\infty\to\OmegaP^\infty E'$ lifts to a map $\Pic^\flat\to\OmegaP^\infty E'$, which gives a desired retraction.
\end{proof}

\begin{corollary}\label{cor:Pic}
Suppose that $\V$ is multiplicative.
Let $E$ be a homotopy commutative motivic ring spectrum in $\V$.
Assume that $E$ is oriented and satisfies projective bundle formula.
Then we have a ring isomorphism
\[
	E^{*,*}(\Pic_+\otimes R) \simeq E^{*,*}(R)[[c_1]]
\]
for every homotopy cocommutative motivic coring spectrum $R$ in $\V$.
\end{corollary}
\begin{proof}
This follows from Corollary \ref{cor:PBF} and Proposition \ref{prop:Pic}.
\end{proof}

\subsection{Chern classes and formal group laws}\label{Che}

\begin{definition}[Higher Chern class]\label{def:Che}
Let $E$ be a motivic ring spectrum over a qcqs derived scheme $S$.
Assume that $E$ is (linearly) oriented and satisfies projective bundle formula.
Let $\mcal{E}$ be a vector bundle of rank $r\ge 1$ on an $S$-stack $X$.
For $1\le i\le r$, we define the \textit{$i$-th Chern class} $c_i(\mcal{E})\in E^i(X)$ to be the unique element which satisfies the formula
\[
	\sum_{i=0}^r (-1)^i c_1(\mcal{O}(1))^i \cdot p^*c_{r-i}(\mcal{E}) = 0
\]
in $E^r(\mbb{P}(\mcal{E}))$ with the convention $c_0(\mcal{E})=1$, cf.\ Lemma \ref{lem:PBF}.
We write $c(\mcal{E}):=\sum_{i=0}^rc_i(\mcal{E})t^i$ and call it the \textit{total Chern class}.
\end{definition}

\begin{p}[Formal group law]
Let $E$ be a homotopy commutative motivic ring spectrum over $S$.
Assume that $E$ is oriented and satisfies projective bundle formula.
Let $m\colon\Pic\times\Pic\to\Pic$ be the map classifying the tensor products of line bundles.
Consider the induced map
\[
	m^* \colon E(\Pic_S) \to E(\Pic_S\times\Pic_S)
\]
and let $f_\univ$ be the image of the universal first Chern class $c_1$ in
\[
	E^*(\Pic_S\times\Pic_S) \simeq E^*(S)[[x,y]],
\] 
where the isomorphism is by Corollary \ref{cor:Pic}.
Then $f_\univ$ is a formal group law over $E^*(S)$.
Since first Chern classes on qcqs derived schemes are nilpotent by Lemma \ref{lem:Ori3}, for every pair of line bundles $\mcal{L},\mcal{L}'$ on $X\in\Sm_S$, we have
\[
	c_1(\mcal{L}\otimes\mcal{L}') = f_\univ(c_1(\mcal{L}),c_1(\mcal{L}'))
\]
in $E^*(X)$.
\end{p}

\begin{lemma}\label{lem:Che}
Let $E$ be a homotopy commutative motivic ring spectrum over a qcqs derived scheme $S$.
Assume that $E$ is oriented and satisfies projective bundle formula.
Let $\mcal{E}$ be a vector bundle of rank $r$ on $X\in\Sm_S$.
Then:
\begin{enumerate}[label=\upshape{(\roman*)}]
\item $c_i(\mcal{E})$ is nilpotent in $E^*(X)$ for every $i\ge 1$.
\item If $\mcal{E}$ admits a filtration
\[
	0 = \mcal{E}_0 \subset \mcal{E}_1 \subset \cdots \subset \mcal{E}_r=\mcal{E}
\]
such that $\mcal{L}_i=\mcal{E}_i/\mcal{E}_{i-1}$ is a line bundle for $1\le i\le r$, then 
\[
	c(\mcal{E}) = \prod_{i=1}^r(1+c_1(\mcal{L}_i)t)
\]
in $E^*(X)[t]$.
\item If we have a fiber sequence
\[
	\mcal{E}' \to \mcal{E} \to \mcal{E}''
\]
of vector bundles on $X$, then $c(\mcal{E})=c(\mcal{E}')\cdot c(\mcal{E}'')$ in $E^*(X)[t]$.
\end{enumerate} 
\end{lemma}
\begin{proof}
By taking the pullback of $\mcal{E}$ to the derived scheme representing full flags of $\mcal{E}$, (iii) is reduced to (ii).
Similarly, (i) follows from (ii) and the fact that first Chern classes are nilpotent.
To prove (ii), we are reduced to the case $\mcal{E}=\bigoplus_{i=1}^r\mcal{L}_i$ by the splitting trick as in \cite[Lemma 4.4]{AI}.
Consider the universal quotient $\mcal{E}\to\mcal{O}(1)$ on $\mbb{P}(\mcal{E})$.
The induced map $\mcal{L}_i\to\mcal{O}(1)$ gives a global section $s_i$ of $\mcal{L}_i^{-1}(1)$, and let $D_i\subset\mbb{P}(\mcal{E})$ be the derived vanishing locus of $s_i$.
Then the intersection of all $D_i$ with $1\le i\le r$ is empty, and thus we get $\prod_ic_1(\mcal{L}_i^{-1}(1))=0$ by Lemma \ref{lem:Ori3}.
By the formal group law, we have
\[
	c_1(\mcal{O}(1)) = c_1(\mcal{L}_i\otimes\mcal{L}_i^{-1}(1))
					 = c_1(\mcal{L}_i) + c_1(\mcal{L}_i^{-1}(1)) + \sum_{p,q\ge 1}a_{pq}c_1(\mcal{L}_i)^pc_1(\mcal{L}_i^{-1}(1))^q
\]
for some $a_{pq}\in E^*(S)$.
Therefore, we have $\prod_i(c_1(\mcal{O}(1))-c_1(\mcal{L}_i))=0$, which implies the desired formula.
\end{proof}

\begin{lemma}\label{lem:Che2}
Suppose that $\V$ is defined over a qcqs derived scheme $S$.
Let $E$ be a homotopy commutative oriented motivic ring spectrum over $S$.
Let $M$ be an $E$-module in $\SpP(\V)$ which satisfies projective bundle formula.
Let $\mcal{E}$ be a vector bundle of rank $r$ on an $S$-stack $X$ and $\mcal{Q}$ the universal quotient vector bundle of $\mcal{E}$ on the grassmannian $\Gr_n(\mcal{E})$.
Then the map
\[
	\sum_\alpha c(\mcal{Q})^\alpha \colon \bigoplus_\alpha \SigmaP^{-\lVert\alpha\rVert}M^{X_+} \to M^{\Gr_n(\mcal{E})_+}
\]
is an equivalence, where $\alpha$ runs over all $n$-tuples of non-negative integers with $|\alpha|\le r-n$.
For an $n$-tuple $\alpha=(\alpha_1,\dotsc,\alpha_n)$ of non-negative integers, we write $|\alpha|:=\sum\alpha_i$, $\lVert\alpha\rVert:=\sum i\alpha_i$ and $c^\alpha:=\prod c_i^{\alpha_i}$.
\end{lemma}
\begin{proof}
The proof of \cite[Lemma 4.5]{AI} works as it is under the validity of Lemma \ref{lem:Che}.
\end{proof}

\begin{corollary}\label{cor:Che}
Suppose that $\V$ is multiplicative.
Let $E$ be a homotopy commutative oriented motivic ring spectrum in $\V$ which satisfies projective bundle formula.
Then we have a ring isomorphism
\[
	E^*(\Gr_{n,+}\otimes R) \simeq E^*(R)[[c_1,\dotsc,c_n]]
\]
for every homotopy cocommutative motivic coring spectrum $R$ in $\V$.
\end{corollary}
\begin{proof}
This follows from Lemma \ref{lem:Che2} as in \cite[Corollary 4.6]{AI}.
\end{proof}

\begin{theorem}\label{thm:Che}
Suppose that $\V$ is multiplicative.
Let $E$ be a homotopy commutative orientable motivic $\mbb{E}_1$-ring spectrum in $\V$.
Then the canonical map
\[
	L_\ex(\Gr_n\otimes E) \to L_\ex(\Vect_n\otimes E)
\]
is an equivalence, where the tensor products are taken in $\SpP(\V)$.
\end{theorem}
\begin{proof}
We may assume that $\V=\St$.
We fix an orientation of $E$.
We say that a left $E$-module $M$ in $\SpP^\pre$ \textit{satisfies grassmannian formula} if the map
\[
	\sum_\alpha c(\mcal{Q})^\alpha \colon \bigoplus_\alpha \SigmaP^{-\lVert\alpha\rVert}M^{X_+} \to M^{\Gr_n(\mcal{O}^N)_+}
\]
is an equivalence for every $n\ge 1$ and $N\ge n$, where $\alpha$ runs over all $n$-tuples of non-negative integers with $|\alpha|\le N-n$, cf.\ Lemma \ref{lem:Che2}.
Let $\LMod_E^\grf(\SpP^\pre)$ be the full subcategory $\LMod_E(\SpP^\pre)$ spanned by left $E$-modules which satisfy grassmannian formula and $L_\grf$ the localization $\LMod_E(\SpP^\pre)\to\LMod_E^\grf(\SpP^\pre)$.
We claim that the canonical map
\[
	\phi \colon L_\grf(\Gr_n\otimes^\pre E) \to L_\grf(\Vect_n^\flat\otimes^\pre E)
\]
is an equivalence.
Then the theorem follows from this claim since the desired equivalence is obtained as a further localization of $\phi$ by Lemma \ref{lem:Che2}.

By Lemma \ref{lem:Lif2} and Remark \ref{rem:Lif}, the map $\phi$ admits a section after finite truncations.
Hence, it suffices to show that it admits a retraction.
Let $\mcal{E}_\univ$ be the universal vector bundle of rank $n$ on $\Vect_n$ and $\mcal{Q}$ the universal quotient vector bundle on $\Gr_n$.
Note that $c_i(\mcal{E}_\univ)$ lifts $c_i(\mcal{Q})$ via the canonical map $E^*(\Vect_n)\to E^*(\Gr_n)$.
We set $E':=L_\grf(\Gr_n\otimes^\pre E)$.
Then it follows from Lemma \ref{lem:Che} that $E'(\Vect_n^\flat)$ is a complete $E(\Vect_n^\flat)$-module along $(c_1(\mcal{E}_\univ),\dotsc,c_n(\mcal{E}_\univ))$.
Hence, we have a commutative diagram
\[
\xymatrix{
	\prod_\alpha \SigmaP^{-\lVert\alpha\rVert}E'(\Spec(\mbb{Z})) \ar[rd]^\simeq \ar[d]_{\prod_\alpha c(\mcal{E}_\univ)^\alpha} & \\
	E'(\Vect_n^\flat) \ar[r] & E'(\Gr_n).
}
\]
and the diagonal arrow is an equivalence by the grassmannian formula.
In particular, the canonical map $\Gr_n\to\OmegaP^\infty E'$ lifts to a map $\Vect_n\to\OmegaP^\infty E'$, which gives a desired retraction.
This completes the proof.
\end{proof}

\begin{corollary}\label{cor:Che2}
Suppose that $\V$ is multiplicative.
Let $E$ be a homotopy commutative oriented motivic $\mbb{E}_1$-ring spectrum in $\V$ which satisfies projective bundle formula.
Then we have a ring isomorphism
\[
	E^{*,*}(\Vect_{n,+}\otimes R) \simeq E^{*,*}(R)[[c_1,\dotsc,c_n]]
\]
for every homotopy cocommutative motivic coring spectrum $R$ in $\V$.
\end{corollary}
\begin{proof}
This follows from Corollary \ref{cor:Che} and Theorem \ref{thm:Che}.
\end{proof}

\begin{remark}[Syntomic cohomology]
The results in this section can be applied to syntomic cohomology in the sense of \cite{BL} and reprove and generalize some of the results in \cite[\S9]{BL} (assuming projective bundle formula \cite[Theorem 9.1.1]{BL}).
\end{remark}

\begin{remark}
Theorem \ref{thm:Che} is generalized to an equivalence $L_\ex\Gr_n\simeq L_\ex\Vect_n$ in \cite[Theorem 5.3]{AHI}.
\end{remark}

\newpage

\section{Applications to $K$-theory}\label{MSn}

In this section, we apply the results we obtained so far to algebraic $K$-theory.
The main result is a universality of $K$-theory as an $\mbb{S}[\Pic]$-module: We prove that the non-connective $K$-theory is a universal $\mbb{S}[\Pic]$-module which satisfies projective bundle formula and Zariski descent (Theorem \ref{thm:AlK}).
We also discuss the Selmer $K$-theory introduced in \cite{Cla}.
We prove that the Selmer $K$-theory is a universal $\mbb{S}[\Pic]$-module which satisfies projective bundle formula and \'etale descent (Theorem \ref{thm:SeK}).

As a by-product, we see that giving an additive morphism $K\to K$, where $K$ denotes the $K$-theory stack, is equivalent to giving an arbitrary morphism of stacks $\Pic\to K$ (Theorem \ref{thm:CoK}).
This would be helpful for studying cohomology operations in $K$-theory.

\subsection{Cohomology of the $K$-theory stack}\label{CoK}

We apply Corollary \ref{cor:Che2} to study cohomology of the $K$-theory stack.
Theorem \ref{thm:CoK} below is a version of \cite[Proposition 2.27]{GS} and \cite[Proposition 5.1.1]{Rio}.
Our proof is inspired by their proofs.

\begin{p}
Let $\V$ be an $\infty$-category presentably tensored over $\St$ as before, cf.\ \ref{p:DMS}.
We assume that $\V$ is compactly generated and that $\mbb{P}^1$ is compact in $\V$.
\end{p}

\begin{p}[$K$-theory]
Let $K$ denote the $K$-theory stack on qcqs derived schemes, which yields an $S$-stack for each qcqs derived scheme $S$.
Note that the $K$-theory stack is left Kan extended from smooth $\mbb{Z}$-algebras as proved by Bhatt and Lurie, cf.\ \cite[Appendix A]{EHKSY}.
Therefore, the base change functor $\St\to\St_S$ carries $K$ to $K$.
\end{p}

\begin{notation}
For stacks $X,Y$, we write $[X,Y]$ for the set of homotopy classes of morphisms in $\St$.
When $X,Y$ are pointed, we write $[X,Y]_*$ for the set of homotopy classes of morphisms in $\St_*$.
\end{notation}

\begin{theorem}\label{thm:CoK}
Suppose that $\V$ is multiplicative.
Let $E$ be a homotopy commutative orientable motivic $\mbb{E}_1$-ring spectrum in $\V$.
Then the canonical map
\[
	L_\ex(\Pic_+\otimes E) \to L_\ex(K\otimes E)
\]
admits a retraction $s$ such that, for every $E$-module $M$ in $\SpP(\V^\ex)$, the map
\[
	s^* \colon M(\Pic) = \Map(\Pic_+,\OmegaP^\infty M) \to \tilde{M}(K) = \Map(K,\OmegaP^\infty M)
\]
identifies $M^0(\Pic)$ with the subset of $\tilde{M}^0(K)=[K,\OmegaP^\infty M]_*$ consisting of additive morphisms.
\end{theorem}
\begin{proof}
We may assume that $\V=\St$.
Let $\Add(-,\OmegaP^\infty M)$ denote the subset of $[-,\OmegaP^\infty M]_*$ consisting of additive morphisms.
We only have to show that the pre-composition by the canonical map $\Pic_+\to K$ induces an isomorphism
\[
	\Add(K,\OmegaP^\infty M) \xrightarrow{\sim} M^0(\Pic).
\]
Indeed, if we take $L_\ex(\Pic_+\otimes E)$ as $M$, then the canonical map $\Pic_+\to\OmegaP^\infty L_\ex(\Pic_+\otimes E)$ lifts to an additive morphism $K\to\OmegaP^\infty L_\ex(\Pic_+\otimes E)$ that yields a desired retraction $s$.

Let $\tilde{K}$ denote the reduced $K$-theory.
Consider the commutative diagram
\[
\xymatrix{
	\tilde{M}(\tilde{K}) \ar[r] \ar[d] & \tilde{M}(K) \ar[r] \ar[d] & \tilde{M}(\mbb{Z}) \ar[d] \\
	\tilde{M}(\Pic) \ar[r] & M(\Pic) \ar[r] & M(S),
}
\]
where the bottom sequence is a split fiber sequence.
Since the induced sequence
\[
	\Add(\tilde{K},\OmegaP^\infty M) \to \Add(K,\OmegaP^\infty M) \to \Add(\mbb{Z},\OmegaP^\infty M)
\]
is a split exact sequence and $\Add(\mbb{Z},\OmegaP^\infty M)=M^0(S)$, we are reduced to showing that the map
\[
	\Add(\tilde{K},\OmegaP^\infty M) \to \tilde{M}^0(\Pic)
\]
is an isomorphism.

Note that $\tilde{K}$ is equivalent to the plus construction of $\Vect_\infty:=\colim_n\Vect_n$.
Since $\OmegaP^\infty M$ is an infinite loop space by Theorem \ref{thm:FSt}, we have
\[
	M(\tilde{K}) \simeq M(\Vect_\infty^+) \simeq M(\Vect_\infty).
\]
We see that $\Add(\tilde{K},\OmegaP^\infty M)$ is identified with the subgroup of the coalgebra $M^0(\tilde{K})$ consisting of primitive elements as in \cite[Lemma 2.25]{GS}.
Then the comultiplication $\Delta$ of $M^0(\tilde{K})$ is identified with the limit of the canonical maps
\[
\xymatrix@R-1pc{
	M^0(S)[[t_1,\dotsc,t_{n+m}]]^{\Sigma_{n+m}} \ar[r] \ar@{=}[d]_\wr & M^0(S)[[t_1,\dotsc,t_{n+m}]]^{\Sigma_n\times\Sigma_m} \ar@{=}[d]_\wr \\
	M^0(\Vect_{n+m}) \ar[r]^-\Delta & M^0(\Vect_n\times\Vect_m)
}
\]
by Corollary \ref{cor:Che2}.
Here $(t_1,t_2,\dotsc)$ are unique variables such that the $n$-th Chern class $c_n$ is the $n$-th elementary polynomial of them.
Therefore, primitive elements $f$ in $M^0(\tilde{K})$ are completely determined by their images $f_0$ in $M^0(\Pic)$.
More precisely, 
\[
	f = \bigl\{ {\textstyle\sum_{i=1}^nf_0(t_i)} \bigr\}_n \in \lim_n M^0(\Vect_n) = M^0(\tilde{K}).
\]
This proves the desired isomorphism $\Add(\tilde{K},\OmegaP^\infty M)\simeq\tilde{M}^0(\Pic)$.
\end{proof}

\begin{remark}[Adams operation]
By Theorem \ref{thm:CoK}, the pre-composition by the canonical map $\Pic\to K$ restricts to an isomorphism
\[
	\Add(K,K) \simeq [\Pic,K].
\]
In particular, for each positive integer $k$, we obtain a unique additive morphism $\psi^k\colon K\to K$ which restricts to a morphism $\Pic\to K$ sending $\mcal{L}$ to $\mcal{L}^{\otimes k}$.
This is exactly the $k$-th Adams operation on $K$-theory.
\end{remark}

\subsection{$\mbb{P}^1$-periodicity}\label{Per}

\begin{notation}[Bott element]
We write $Q(\Pic):=\Omega^\infty(\Sigma^\infty\Pic_+)$ and regard it as an $\mbb{E}_\infty$-algebra in $\St_*$.
Let $\beta$ be a morphism in $\St_*$ defined by
\[
	\beta := 1-[\mcal{O}(-1)] \colon \mbb{P}^1 \to Q(\Pic),
\]
which we refer to as the \textit{Bott element}.
\end{notation}

\begin{definition}[$\mbb{P}^1$-periodicity]\label{def:Per}
We say that a $Q(\Pic)$-module $E$ in $\V_*$ \textit{satisfies $\mbb{P}^1$-periodicity} if the map
\[
	\beta \colon E \to E^{\mbb{P}^1}
\]
is an equivalence.
Let $\Mod_{Q(\Pic)}^{\mbb{P}^1}(\V_*)$ denote the full subcategory of $\Mod_{Q(\Pic)}(\V_*)$ spanned by $Q(\Pic)$-modules which satisfy $\mbb{P}^1$-periodicity.
\end{definition}

\begin{remark}
The $\infty$-category $\Mod_{Q(\Pic)}^{\mbb{P}^1}(\V_*)$ is an accessible localization of $\Mod_{Q(\Pic)}(\V_*)$.
Let $L_{\mbb{P}^1}$ denote the localization
\[
	L_{\mbb{P}^1} \colon \Mod_{Q(\Pic)}(\V_*) \to \Mod_{Q(\Pic)}^{\mbb{P}^1}(\V_*).
\]
If $\V$ is multiplicative, then $\Mod_{Q(\Pic)}^{\mbb{P}^1}(\V_*)$ admits a unique presentably symmetric monoidal structure for which the localization $L_{\mbb{P}^1}$ is symmetric monoidal.
In general, $\Mod_{Q(\Pic)}^{\mbb{P}^1}(\V_*)$ is presentably tensored over $\Mod_{Q(\Pic)}^{\mbb{P}^1}(\St_*)$ and we have an equivalence
\[
	\Mod_{Q(\Pic)}^{\mbb{P}^1}(\V_*) \simeq \Mod_{Q(\Pic)}^{\mbb{P}^1}(\St_*)\otimes_\St\V,
\]
where the tensor product is taken in $\Mod_\St(\Pr^L)$.
\end{remark}

\begin{remark}
Since $\mbb{P}^1$ is invertible in $\Mod_{Q(\Pic)}^{\mbb{P}^1}(\V_*)$, we have an equivalence
\[
	\Mod_{Q(\Pic)}^{\mbb{P}^1}(\V_*) \simeq \SpP(\Mod_{Q(\Pic)}^{\mbb{P}^1}(\V_*)).
\]
In particular, if $E$ is an $Q(\Pic)$-module in $\V_*$ which satisfies $\mbb{P}^1$-periodicity, then it canonically yields a motivic spectrum in $\V$, which we denote by $E$.
The motivic spectrum $E$ in $\V$ is \textit{periodic}, i.e., $\SigmaP E\simeq E$.
Moreover, the map
\[
	\beta \colon E \xrightarrow{\sim} E^{\mbb{P}^1}
\]
canonically factors through $E^\Pic$ and it gives an orientation of $E$ in the sense of Definition \ref{def:Ori}.
Then the Bott element $\beta$ is recovered as the first Chern class $c_1(\mcal{O}(1))\colon E\to\SigmaP E^{\mbb{P}^1}\simeq E^{\mbb{P}^1}$.
\end{remark}

\begin{lemma}\label{lem:Per}
The $\infty$-category $\Mod_{Q(\Pic)}^{\mbb{P}^1}(\V_*)$ is stable.
\end{lemma}
\begin{proof}
This follows from Theorem \ref{thm:FSt}.
\end{proof}

\begin{lemma}\label{lem:Per2}
Let $E$ be a $Q(\Pic)$-module in $\SpP(\V)$.
Then there is a natural equivalence
\[
	L_{\mbb{P}^1}E \simeq E[\beta^{-1}] \simeq \colim(E\xrightarrow{\beta}\SigmaP^{-1}E\xrightarrow{\beta}\SigmaP^{-2}E\xrightarrow{\beta}\cdots).
\] 
\end{lemma}
\begin{proof}
Since $\mbb{P}^1$ is invertible in $\SpP(\V)$, the localization $L_{\mbb{P}^1}$ on $\Mod_{Q(\Pic)}(\SpP(\V))$ is exactly the inversion of the Bott element $\beta$, i.e., $L_{\mbb{P}^1}\colon E\mapsto E[\beta^{-1}]$.
Since $\pi_1\Map(-,Q(\Pic))$ is abelian, this localization has the desired description, cf.\ \cite[Appendix C]{BNT}.
\end{proof}

\begin{corollary}\label{cor:Per}
Let $E$ be a $Q(\Pic)$-module in $\V_*$.
Then there is a natural equivalence
\[
	L_{\mbb{P}^1} E \simeq
	\colim(\SigmaP^\infty E\xrightarrow{\beta}\SigmaP^{\infty-1}E\xrightarrow{\beta}\SigmaP^{\infty-2}E\xrightarrow{\beta}\cdots).
\] 
\end{corollary}
\begin{proof}
Note that we have a commutative diagram
\[
\xymatrix{
	\SpP(\Mod_{Q(\Pic)}(\V_*)) \ar[r]^-{L_{\mbb{P}^1}} & \SpP(\Mod_{Q(\Pic)}^{\mbb{P}^1}(\V_*))  \\
	\Mod_{Q(\Pic)}(\V_*) \ar[r]^-{L_{\mbb{P}^1}} \ar[u]^{\SigmaP^\infty} & \Mod_{Q(\Pic)}^{\mbb{P}^1}(\V_*). \ar[u]^{\SigmaP^\infty}_\simeq
}
\]
The assertion is immediate from this diagram and Lemma \ref{lem:Per2}.
\end{proof}

\begin{p}[Pbf-localization]
Let $E$ be a $Q(\Pic)$-module in $\V_*$ which satisfies $\mbb{P}^1$-periodicity.
Then it follows from Lemma \ref{lem:EBE2} that $E$ satisfies elementary blowup excision if and only if it satisfies projective bundle formula, i.e., the map
\[
	\sum_{i=1}^n \beta^i \colon \bigoplus_{i=1}^n E \to E^{\mbb{P}^n}
\]
is an equivalence for every $n\ge 1$.
We write $\Mod_{Q(\Pic)}^\pbf(\V_*):=\Mod_{Q(\Pic)}^{\mbb{P}^1}(\V^\ex_*)$; then it is identified with the full subcategory of $\Mod_{Q(\Pic)}(\V_*)$ spanned by $Q(\Pic)$-modules which satisfies (periodic) projective bundle formula.
We consider the localizations
\[
\xymatrix@1{
	\Mod_{Q(\Pic)}(\V_*) \ar[r]_-{L_{\mbb{P}^1}} \ar@/^1.5pc/[rr]^-{L_\pbf} & \Mod^{\mbb{P}^1}_{Q(\Pic)}(\V_*) \ar[r]_-{L_\ex} & \Mod_{Q(\Pic)}^\pbf(\V_*)
}
\]
and refer to $L_\pbf$ as the \textit{pbf-localization}.
In this section, $L_\pbf$ consistently refers to the localization introduced here, not the one in Remark \ref{rem:PTh} (which was temporarily used).
\end{p}

\begin{remark}
When $\V$ is stable, we could instead consider $\mbb{S}[\Pic]$-modules and the localizations
\[
\xymatrix@1{
	\Mod_{\mbb{S}[\Pic]}(\V) \ar[r]_-{L_{\mbb{P}^1}} \ar@/^1.5pc/[rr]^-{L_\pbf} &
		\Mod^{\mbb{P}^1}_{\mbb{S}[\Pic]}(\V) \ar[r]_-{L_\ex} & \Mod_{\mbb{S}[\Pic]}^\pbf(\V),
}
\]
but it does not make any difference, namely, the diagram
\[
\xymatrix{
	\Mod_{\mbb{S}[\Pic]}(\V) \ar[r]^-{L_{\mbb{P}^1}} \ar[d] &
		\Mod^{\mbb{P}^1}_{\mbb{S}[\Pic]}(\V) \ar[r]^-{L_\ex} \ar[d] & \Mod_{\mbb{S}[\Pic]}^\pbf(\V) \ar[d] \\
	\Mod_{Q(\Pic)}(\V) \ar[r]^-{L_{\mbb{P}^1}} & \Mod^{\mbb{P}^1}_{Q(\Pic)}(\V) \ar[r]^-{L_\ex} & \Mod_{Q(\Pic)}^\pbf(\V)
}
\]
is commutative, where the vertical functors are the restrictions of scalars.
\end{remark}

\subsection{Universality of $K$-theory}\label{AlK}

\begin{p}
Let $\mbb{S}[\Pic]$ be the stabilization $\Sigma^\infty\Pic_+$, which yields an $\mbb{E}_\infty$-algebra in $\Sp(\St_S)$ for each qcqs derived scheme $S$.
We consider the pbf-localization $L_\pbf\mbb{S}[\Pic]$ of the $\mbb{S}[\Pic]$-module $\mbb{S}[\Pic]$ in $\Sp(\St_S)$.
\end{p}

\begin{p}[$K$-theory]
Let $S$ be a qcqs derived scheme.
We have an evident morphism $Q(\Pic)\to K$ of $\mbb{E}_\infty$-algebras in $\St_{S*}$ and $K$ satisfies projective bundle formula with respect to this $Q(\Pic)$-module structure, cf.\ \cite[Theorem B]{Kha}.
In particular, $K$ canonically yields an oriented motivic $\mbb{E}_\infty$-ring spectrum $K$ over $S$, and it is uniquely lifted to a motivic $\mbb{E}_\infty$-ring spectrum $K^\Bass$ in $\Sp(\St_S)$ by Theorem \ref{thm:FSt}, which recovers the Bass non-connective $K$-theory as in \cite{TT,BGT}.
Since the $K$-theory stack is left Kan extended from smooth $\mbb{Z}$-algebras, the base change functor $\SpP\to\SpP(S)$ carries $K$ to $K$.
\end{p}

\begin{theorem}\label{thm:AlK}
For every qcqs derived scheme $S$, the canonical map
\[
	L_\pbf\mbb{S}[\Pic] \to K^\Bass
\]
is an equivalence of $\mbb{E}_\infty$-algebras in $\Sp(\St_S)$.
\end{theorem}
\begin{proof}
We may assume that $S=\Spec(\mbb{Z})$.
We work over the $\infty$-category $\St^\ex$ consistently.
Then, considering each universal construction in that sense, the assertion is equivalent to saying that the canonical map
\[
	L_{\mbb{P}^1}\Sigma^\infty_+\Pic \to K^\Bass
\]
is an equivalence; where $L_{\mbb{P}^1}$ is the left adjoint to $\Mod_{Q(\Pic)}^{\mbb{P}^1}(\Sp(\St^\ex))\to\Mod_{Q(\Pic)}(\Sp(\St^\ex))$ and $\Sigma^\infty_+$ is the stabilization $\St^\ex\to\Sp(\St^\ex)$.

We consider the square in $\Sp(\St^\ex)$.
\[
\xymatrix{
	X \ar[r] \ar[d] & \Sigma^\infty_+\Pic \ar[d] \\
	\Sigma^\infty K \ar[r] & K^\Bass,
}
\]
where $X$ is defined to be the pullback.
We endow $\Sigma^\infty K$ with the $Q(\Pic)$-module structure which comes form the canonical map $Q(\Pic)\to K$.
Then the square canonically lifts a cartesian square in $\Mod_{Q(\Pic)}(\Sp(\St^\ex))$.
Warn that $\Sigma^\infty K$ has another $Q(\Pic)$-module structure which comes from the canonical $\mbb{S}[\Pic]$-module structure on $\Sigma^\infty K$, but it does not work in later steps.

The goal is to show that the right vertical map is an $L_{\mbb{P}^1}$-equivalence, and in fact we prove that each map in the diagram is an $L_{\mbb{P}^1}$-equivalence; the right vertical equivalence will be deduced from the other equivalences at the end.
We first show that $\Sigma^\infty K\to K^\Bass$ is an $L_{\mbb{P}^1}$-equivalence.
Note that $L_{\mbb{P}^1}\Sigma^\infty K$ is identified with the image of $K$ under the composite functor
\[
	\Mod_{Q(\Pic)}(\St^\ex_*)
		\xrightarrow{L_{\mbb{P}^1}} \Mod_{Q(\Pic)}^{\mbb{P}^1}(\St^\ex_*)
			\xrightarrow{\Sigma^\infty} \Mod_{Q(\Pic)}^{\mbb{P}^1}(\Sp(\St^\ex)).
\]
Since $K$ satisfies $\mbb{P}^1$-periodicity, the first functor carries $K$ to $K$.
Since the second functor is an equivalence by Lemma \ref{lem:Per}, it carries $K$ to $K^\Bass$.
Then it follows that the canonical map $\Sigma^\infty K\to K^\Bass$ is an $L_{\mbb{P}^1}$-equivalence, and thus so is $X\to\Sigma^\infty_+\Pic$.

We show that $L_{\mbb{P}^1}X\to L_{\mbb{P}^1}\Sigma^\infty K$ is an equivalence as motivic spectra in $\Sp(\St^\ex)$.
Let $E$ be a homotopy commutative orientable motivic $\mbb{E}_1$-ring spectrum in $\Sp(\St^\ex)$.
It suffices to show that the induced map
\[
	\Map(L_{\mbb{P}^1}\Sigma^\infty K,M) \to \Map(L_{\mbb{P}^1}X,M)
\]
is an equivalence for each $E$-module $M$ in $\SpP(\Sp(\St^\ex))$.
By Corollary \ref{cor:Per}, the map is identified with the limit of the map of towers
\[
\xymatrix{
	\cdots \ar[r]^-\beta & \Map(\Sigma^\infty K,\OmegaP^{\infty-2}M) \ar[r]^-\beta \ar[d] &
		\Map(\Sigma^\infty K,\OmegaP^{\infty-1}M) \ar[r]^-\beta \ar[d] & \Map(\Sigma^\infty K,\OmegaP^\infty M) \ar[d] \\
	\cdots \ar[r]^-\beta & \Map(X,\OmegaP^{\infty-2}M) \ar[r]^-\beta &
		\Map(X,\OmegaP^{\infty-1}M) \ar[r]^-\beta & \Map(X,\OmegaP^\infty M) \\
}
\]
Little stronger, we show that the induced maps
\[
	\lim_n\!{}^i\pi_*\Map(\Sigma^\infty K,\OmegaP^{\infty-n}M) \to \lim_n\!{}^i\pi_*\Map(X,\OmegaP^{\infty-n}M)
\]
are equivalences for $i\ge 0$.

Here are some auxiliary observations.
We have seen that the canonical map
\[
\xymatrix{
	\Map(L_{\mbb{P}^1}\Sigma^\infty_+\Pic,M) \ar@{}[r]|-\simeq \ar[d]^\wr & \lim\Bigl(\cdots \ar[r]^-\beta &
		\Map(\Sigma^\infty_+\Pic,\OmegaP^{\infty-1}M) \ar[r]^-\beta \ar[d] & \Map(\Sigma^\infty_+\Pic,\OmegaP^\infty M)\Bigr) \ar[d] \\
	\Map(L_{\mbb{P}^1}X,M) \ar@{}[r]|-\simeq & \lim\Bigl(\cdots \ar[r]^-\beta &
		\Map(X,\OmegaP^{\infty-1}M) \ar[r]^-\beta & \Map(X,\OmegaP^\infty M)\Bigr)
}
\]
is an equivalence.
The canonical map $X\to\Sigma^\infty_+\Pic$ admits a section, and thus $\Map(\Sigma^\infty_+\Pic,-)\to\Map(X,-)$ admits a retraction.
Then it follows from the Milnor sequence that the canonical maps
\[
	\lim_n\!{}^i\pi_*\Map(\Sigma^\infty_+\Pic,\OmegaP^{\infty-n}M) \xrightarrow{\sim} \lim_n\!{}^i\pi_*\Map(X,\OmegaP^{\infty-n}M)
\]
are equivalences for $i\ge 0$.

Getting back on the track, we will construct a map of towers of abelian groups
\[
	\tilde{s}^* \colon \{\pi_*\Map(X,\OmegaP^{\infty-n}M)\}_n \to \{\pi_*\Map(\Sigma^\infty K,\OmegaP^{\infty-n}M)\}_n
\]
such that it is levelwise a retraction of the canonical map and that the composite
\[
	\{\pi_*\Map(\Sigma^\infty_+\Pic,\OmegaP^{\infty-n}M)\}_n \to
		\{\pi_*\Map(X,\OmegaP^{\infty-n}M)\}_n \xrightarrow{\tilde{s}^*}
			\{\pi_*\Map(\Sigma^\infty K,\OmegaP^{\infty-n}M)\}_n
\]
induces isomorphisms on higher limits.
Then we conclude by the two out of three property.
The map $\tilde{s}^*$ is constructed as follows.
By Theorem \ref{thm:CoK}, there is a morphism $s$
\[
\xymatrix{
	X\otimes E \ar[r] \ar[d] & \Sigma^\infty_+\Pic\otimes E \ar[d] \\
	\Sigma^\infty K\otimes E \ar[r] \ar[ru]^s \ar@{.>}@/^1pc/[u]^{\tilde{s}} & K^\Bass\otimes E,
}
\]
that respects $Q(\Pic)$-module structures and that makes the bottom triangle commutative.
Hence, we obtain a $Q(\Pic)$-linear section $\tilde{s}$, from which the desired map $\tilde{s}^*$ is contravariantly induced.
It is clear that $\tilde{s}^*$ satisfies the first requirement, i.e., levelwise a retraction of the canonical map.
It remains to show that the maps
\[
	s^* \colon \lim_n\!{}^i\pi_*\Map(\Pic_+,\Omega^\infty\OmegaP^{\infty-n}M) \to \lim_n\!{}^i\pi_*\Map(K,\Omega^\infty\OmegaP^{\infty-n}M)
\]
are equivalences.
It is reduced to the case $\pi_0$ by replacing $M$.
By Theorem \ref{thm:CoK}, the map $s^*$ induces an isomorphism
\[
	[\Pic_+,\Omega^\infty\OmegaP^{\infty-n}M]_* \simeq \Add(K,\Omega^\infty\OmegaP^{\infty-n}M).
\]
We claim that the map $\beta\colon[K,\Omega^\infty\OmegaP^{\infty-n-1}M]_*\to[K,\Omega^\infty\OmegaP^{\infty-n}M]_*$ factors through the subset of additive morphisms (the case $n=0$ suffices).
Indeed, if we are given a map $\alpha\colon K\to\Omega^\infty\OmegaP^{\infty-1} M$, then it fits into a commutative diagram
\[
\xymatrix{
	\Omega K^{\mbb{G}_m} \ar[r] \ar[d]^{\Omega\alpha^{\mbb{G}_m}} & K^{\mbb{P}^1} \ar[r]^-{\beta^{-1}}_-\simeq &  K \ar[d]^{\beta\cdot\alpha} \\
	\Omega^{\infty+1}\OmegaP^{\infty-1}M^{\mbb{G}_m} \ar[rr] & & \Omega^\infty\OmegaP^\infty M
}
\]
and horizontal maps have sections as motivic spectra.
Hence, $\beta\cdot\alpha$ is additive.

Abstractly, we have shown the following.
Let $A_n:=[\Pic_+,\Omega^\infty\OmegaP^{\infty-n}M]_*$ and $B_n:=[K,\Omega^\infty\OmegaP^{\infty-n}M]_*$.
Then $s^*$ induces a levelwise injection $A_\bullet\to B_\bullet$ and that $B_{n+1}\to B_n$ factors through $A_n$.
This means that $s^*$ exhibits $A_\bullet$ as a cofinal subsystem of $B_\bullet$, and thus $s^*\colon R\lim A_\bullet\xrightarrow{\sim}R\lim B_\bullet$.
This completes the proof.
\end{proof}

\begin{remark}
The proof of Theorem \ref{thm:AlK} may look roundabout.
Let us try to clarify why.
The goal is to show that the canonical map $L_{\mbb{P}^1}\Sigma^\infty_+\Pic\to K^\Bass$ is an equivalence.
The essential point is that the $K$-theory spectrum $K^\Bass$ is build from the $K$-theory spaces $(K,K,\cdots)$ as a motivic spectrum and that we understand well the cohomology of $K$ by Theorem \ref{thm:CoK}.
More precisely, we have an equivalence
\[
	K^\Bass \simeq \colim_n \Sigma^\infty\SigmaP^{\infty-n} K
\]
as motivic spectra by Corollary \ref{cor:Per} ($\V=\St^\ex$), where the map $\Sigma^\infty\SigmaP^{\infty-n-1}K\to\Sigma^\infty\SigmaP^{\infty-n}K$ is the infinity suspension of the multiplication by the Bott element $\beta$ of $K$.
On the other hand, we have an equivalence
\[
	L_{\mbb{P}^1}\Sigma^\infty_+\Pic \simeq \colim_n \SigmaP^{\infty-n}\Sigma^\infty_+\Pic
\]
again by Corollary \ref{cor:Per} but this time $\V=\Sp(\St^\ex)$.
We would like to compare these two telescopes.
However, the canonical map $\Sigma^\infty_+\Pic\to\Sigma^\infty K$ is incompatible with the telescopic structures, because it does not preserve the Bott element.\footnote{This point seems to be overlooked in \cite{GS} and \cite[Corollary 4.8]{GS} is probably not true.}
To handle this issue, we note that the map $s\colon\Sigma^\infty K\to\Sigma^\infty_+\Pic$ as in Theorem \ref{thm:CoK} does preserve the Bott element, and thus yields a map of telescopes, but rather non-canonically.
Furthermore, it induces an isomorphism
\[
	\lim_n\!{}^iE^{*+n}(\Pic) \simeq \lim_n\!{}^i\tilde{E}^{*+n}(K)
\]
for a homotopy commutative orientable motivic ring spectrum $E$.
This is the computational input.
Then one way to organize the proof is to introduce the pullback $X$ and prove the $L_{\mbb{P}^1}$-equivalence of the canonical map $X\to\Sigma^\infty K$ using the less canonical isomorphism above.
\end{remark}

\subsection{Universality of Selmer $K$-theory}\label{SeK}

\begin{p}[Selmer $K$-theory]
Let $K^\et$ denote the \'etale sheafification of the $K$-theory stack, which yields an \'etale $S$-stack for each qcqs derived scheme $S$.
We have an evident morphism $Q(\Pic)\to K^\et$ of $\mbb{E}_\infty$-algebras in $\St_{S*}$ and $K^\et$ satisfies projective bundle formula with respect to this $Q(\Pic)$-module structure by \cite[Theorem 1.1]{CM}.
In particular, $K^\et$ canonically yields an oriented motivic $\mbb{E}_\infty$-ring spectrum over $S$, and it is uniquely lifted to a motivic $\mbb{E}_\infty$-ring spectrum $K^\Sel$ in $\Sp(\St_S)$ by Theorem \ref{thm:FSt}, which recovers the Selmer $K$-theory as in \cite{Cla,CM}.
\end{p}

\begin{remark}
The $p$-adic Selmer $K$-theory, as presheaves of $p$-complete spectra on animated rings, is left Kan extended from smooth $\mbb{Z}$-algebras, because so is the $K(1)$-local $K$-theory and the topological cyclic homology, cf.\ \cite[Appendix A]{EHKSY} and \cite[Theorem G]{CMM}.
The rational Selmer $K$-theory is the rational $K$-theory, so that it is stable under base changes as motivic spectra.
It follows that the base change functor $\SpP\to\SpP(S)$ carries $K^\et$ to $K^\et$.
\end{remark}

\begin{notation}
We write $L_{\pbf,\et}$ for the composition
\[
\xymatrix{
	& \Mod_{Q(\Pic)}(\Sp(\St^\et_S)) \ar[rd]^{L_\pbf} & \\
	\Mod_{Q(\Pic)}(\Sp(\St_S)) \ar[ru]^{L_\et} \ar[rd]_{L_\pbf} \ar[rr]^-{L_{\pbf,\et}} & & \Mod_{Q(\Pic)}^\pbf(\Sp(\St^\et_S)) \\
	& \Mod_{Q(\Pic)}^\pbf(\Sp(\St_S)). \ar[ru]_{L_\et} &
}
\]
\end{notation}

\begin{theorem}\label{thm:SeK}
For every qcqs derived scheme $S$, the canonical map
\[
	L_{\pbf,\et}\mbb{S}[\Pic] \to K^\Sel
\]
is an equivalence of $\mbb{E}_\infty$-algebras in $\Sp(\St_S)$.
\end{theorem}
\begin{proof}
We apply the functor
\[
	L_\et \colon \Mod_{Q(\Pic)}^\pbf(\Sp(\St_S)) \to \Mod_{Q(\Pic)}^\pbf(\Sp(\St^\et_S))
\]
to the equivalence $L_\pbf\mbb{S}[\Pic]\xrightarrow{\sim}K^\Bass$ in Theorem \ref{thm:AlK}.
Then the left hand side becomes $L_{\pbf,\et}\mbb{S}[\Pic]$, and thus it remains to show that this functor carries $K^\Bass$ to $K^\Sel$.
Consider the commutative diagram
\[
\xymatrix{
	\Mod_{Q(\Pic)}^\pbf(\Sp(\St_S)) \ar[r]^-{L_\et} & \Mod_{Q(\Pic)}^\pbf(\Sp(\St^\et_S)) \\
	\Mod_{Q(\Pic)}^\pbf(\St_{S*}) \ar[r]^-{L_\et} \ar[u]^{\Sigma^\infty}_\simeq & \Mod_{Q(\Pic)}^\pbf(\St^\et_{S*}), \ar[u]^{\Sigma^\infty_\et}
}
\]
where the left vertical functor is an equivalence by Lemma \ref{lem:Per} (we do not know if the right vertical functor is an equivalence since the $\infty$-category $\St^\et_S$ may not be compactly generated).
Hence, it suffices to show that $\Sigma^\infty_\et L_\et K$ is the Selmer $K$-theory, and indeed we have
\[
	\Sigma^\infty_\et L_\et K \simeq \Sigma^\infty_\et K^\et
							\simeq L_\et\Sigma^\infty K^\et
							\simeq L_\et K^\Sel
							\simeq K^\Sel,
\]
where each functor is taken in the sense of the commutative diagram above.
The first equivalence follows from the fact that $K^\et$ satisfies projective bundle formula (cf.\ \cite[Theorem 1.1]{CM}), the second one is obvious, the third one follows from the fact $K^\Sel$ is the unique infinite delooping of $K^\et$ as pbf-local $Q(\Pic)$-modules, and the last one follows from the fact that $K^\Sel$ is an \'etale sheaf.
\end{proof}

\newpage

\appendix

\section{Categorical toolbox}\label{CPr}

\subsection{Modules over commutative algebras}\label{Mod}

The operad $\mbb{LM}^\otimes$ defined in \cite{HA} plays a fundamental role in the theory of left modules over algebras.
More specifically, $\mbb{LM}$-monoidal $\infty$-categories control the theory of $\infty$-categories left-tensored over monoidal $\infty$-categories and left module objects in an $\mbb{LM}$-monoidal $\infty$-category $\mcal{M}^\otimes$ are precisely morphisms of $\infty$-operads $\mbb{LM}^\otimes\to\mcal{M}^\otimes$.
In the body of this paper, we mainly employ the cases where monoidal $\infty$-categories underlie symmetric monoidal $\infty$-categories.
While this can be dealt with as special cases, it is also possible to replace the operad $\mbb{LM}^\otimes$ by a simpler operad $\mbb{M}^\otimes$ and develop the theory in parallel with some simplification.
The second viewpoint is sometimes more convenient and we lay down some of its foundations in this subsection.

\begin{p}
Let $\mcal{O}^\otimes$ be an $\infty$-operad.
Recall from \cite{HA} that an $\mcal{O}$-monoidal $\infty$-category is an $\infty$-operad $\mcal{C}^\otimes$ equipped with a cocartesian fibration of $\infty$-operads $\mcal{C}^\otimes\to\mcal{O}^\otimes$.
For $\mcal{O}$-monoidal $\infty$-categories $\mcal{C}^\otimes$ and $\mcal{D}^\otimes$, let $\Fun^\lax_\mcal{O}(\mcal{C},\mcal{D})$ denote the $\infty$-category of lax $\mcal{O}$-monoidal functors and let $\Fun^\otimes_\mcal{O}(\mcal{C},\mcal{D})$ denote its full subcategory spanned by $\mcal{O}$-monoidal functors.
\end{p}

\begin{definition}[The operad for modules]\label{def:Mod}
We define a category $\mbb{M}^\otimes$ as follows:
\begin{itemize}[label={---},leftmargin=*]
\item An object in $\mbb{M}^\otimes$ is a pair $(I,S)$ of a pointed finite set $I\in\Fin_*$ and a pointed subset $S$ of $I$.
\item A morphism from $(I,S)$ to $(J,T)$ in $\mbb{M}^\otimes$ is a morphism $\alpha\colon I\to J$ in $\Fin_*$ such that $S\subset\alpha^{-1}(T)$ and that it restricts to a bijection $\alpha^{-1}(T^\circ)\xrightarrow{\sim}T^\circ$.
\end{itemize}
Then the forgetful functor $\mbb{M}^\otimes\to\Fin_*$ exhibits $\mbb{M}^\otimes$ as an operad.
The underlying category $\mbb{M}$ has exactly two objects $\mfk{a}=(\langle 1\rangle,*)$ and $\mfk{m}=(\langle 1\rangle,\langle 1\rangle)$.
Note that there is a unique morphism $\Comm^\otimes\to\mbb{M}^\otimes$ of operads, which is given by $I\mapsto(I,*)$.
\end{definition} 

\begin{remark}
Let $\mcal{M}^\otimes$ be an $\mbb{M}$-monoidal $\infty$-category.
Then the underlying $\infty$-category of $\mcal{M}^\otimes$ is the disjoint coproduct $\mcal{M}_\mfk{a}\sqcup\mcal{M}_\mfk{m}$ and $\mcal{M}_\mfk{a}$ has a symmetric monoidal structure given by the base change $\mcal{M}^\otimes\times_{\mbb{M}^\otimes}\Comm^\otimes$.
The active morphism $(\langle 2\rangle,\{0,2\})\to(\langle 1\rangle,\{0,1\})$ in $\mbb{M}^\otimes$ induces a functor $\mcal{M}_\mfk{a}\times\mcal{M}_\mfk{m}\to\mcal{M}_\mfk{m}$, which we call the \textit{tensor product}.
\end{remark}

\begin{definition}[Tensored $\infty$-category]\label{def:Mod2}
Let $\mcal{C}^\otimes$ be a symmetric monoidal $\infty$-category.
We say that an $\infty$-category $\mcal{M}$ is \textit{tensored over $\mcal{C}$} if we are supplied with an $\mbb{M}$-monoidal $\infty$-category $\mcal{M}^\otimes$, an equivalence of symmetric monoidal $\infty$-categories $\mcal{C}^\otimes\simeq\mcal{M}_\mfk{a}^\otimes$, and an equivalence of $\infty$-categories $\mcal{M}\simeq\mcal{M}_\mfk{m}$.
Then we say that $\mcal{M}^\otimes$ \textit{exhibits $\mcal{M}$ as tensored over $\mcal{C}$}.
\end{definition}

\begin{remark}
Let us clarify our notational convention about $\mcal{M}\leftrightarrow\mcal{M}^\otimes$.
If $\mcal{M}^\otimes$ is an $\mbb{M}$-monoidal $\infty$-category, then we usually denote by $\mcal{M}$ the underlying $\infty$-category, or rather the pair $(\mcal{M}_\mfk{a},\mcal{M}_\mfk{m})$.
If $\mcal{M}$ is an $\infty$-category tensored over a symmetric monoidal $\infty$-category $\mcal{C}$, then we usually denote by $\mcal{M}^\otimes$ the $\mbb{M}$-monoidal $\infty$-category that exhibits $\mcal{M}$ as tensored over $\mcal{C}$.
In the second case, $\mcal{M}$ may also mean the pair $(\mcal{C},\mcal{M})$, but this abuse will not cause confusion by specifying the symmetric monoidal $\infty$-category $\mcal{C}$.
\end{remark}

\begin{remark}
Let $\mcal{C}^\otimes$ be a symmetric monoidal $\infty$-category.
Then the base change $\mcal{C}^\otimes\times_{\Comm^\otimes}\mbb{M}^\otimes$ is an $\mbb{M}$-monoidal $\infty$-category that exhibits $\mcal{C}$ as tensored over $\mcal{C}$.
We will sometimes regard $\mcal{C}^\otimes$ as an $\mbb{M}$-monoidal $\infty$-category in this way.
\end{remark}

\begin{definition}[Module object]\label{def:Mod3}
Let $\mcal{M}^\otimes$ be an $\mbb{M}$-monoidal $\infty$-category.
Then we define
\[
	\Mod(\mcal{M}) := \Alg_{/\mbb{M}}(\mcal{M})
\]
and call it the \textit{$\infty$-category of module objects in $\mcal{M}$}.
Note that the pullback along $\mcal{M}^\otimes_\mfk{a}\to\mcal{M}^\otimes$ induces a functor $\Mod(\mcal{M})\to\CAlg(\mcal{M}_\mfk{a})$.
For a commutative algebra object $A$ in $\mcal{M}_\mfk{a}$, we define an $\infty$-category $\Mod_A(\mcal{M})$ by the pullback
\[
\xymatrix{
	\Mod_A(\mcal{M}) \ar[r] \ar[d] & \Mod(\mcal{M}) \ar[d] \\
	{*} \ar[r]^-A & \CAlg(\mcal{M}_\mfk{a})
}
\]
and call it the \textit{$\infty$-category of $A$-module objects in $\mcal{M}$}.
\end{definition}

\begin{remark}
An $\mbb{M}$-monoidal $\infty$-category is canonically identified with a module object in $\Cat_\infty$.
In particular, for a symmetric monoidal $\infty$-category $\mcal{C}$, an $\infty$-category tensored over $\mcal{C}$ is identified with a $\mcal{C}$-module object in $\Cat_\infty$.
\end{remark}

\begin{definition}[Linear functor]\label{def:Mod4}
Let $\mcal{M}$ and $\mcal{N}$ be $\infty$-categories tensored over symmetric monoidal $\infty$-categories $\mcal{C}$ and $\mcal{D}$ respectively and let $F\colon\mcal{C}\to\mcal{D}$ be a lax symmetric monoidal functor.
Then we define an $\infty$-category $\Fun_F^\lax(\mcal{M},\mcal{N})$ by the pullback
\[
\xymatrix{
	\Fun_F^\lax(\mcal{M},\mcal{N}) \ar[r] \ar[d] & \Fun_\mbb{M}^\lax((\mcal{C},\mcal{M}),(\mcal{D},\mcal{N})) \ar[d] \\
	{*} \ar[r]^-F & \Fun^\lax(\mcal{C},\mcal{D})
}
\]
and call it the \textit{$\infty$-category of lax $F$-linear functors}.
If $F\colon\mcal{C}\to\mcal{D}$ is a symmetric monoidal functor, then we define an $\infty$-category $\Fun_F^\otimes(\mcal{M},\mcal{N})$ by the pullback
\[
\xymatrix{
	\Fun_F^\otimes(\mcal{M},\mcal{N}) \ar[r] \ar[d] & \Fun_\mbb{M}^\otimes((\mcal{C},\mcal{M}),(\mcal{D},\mcal{N})) \ar[d] \\
	{*} \ar[r]^-F & \Fun^\otimes(\mcal{C},\mcal{D})
}
\]
and call it the \textit{$\infty$-category of $F$-linear functors}.
If $\mcal{C}=\mcal{D}$ and $F$ is the identity functor on $\mcal{C}$, then we refer to a (lax) $F$-linear functor as a \textit{(lax) $\mcal{C}$-linear functor}.
\end{definition}

\begin{remark}
Let $\mcal{M}$ and $\mcal{N}$ be $\infty$-categories tensored over symmetric monoidal $\infty$-categories $\mcal{C}$ and $\mcal{D}$ respectively and let $F\colon\mcal{C}\to\mcal{D}$ be a lax symmetric monoidal functor.
Then a lax $F$-linear functor $\mcal{M}\to\mcal{N}$ induces a functor
\[
	F \colon \Mod_A(\mcal{M}) \to \Mod_{F(A)}(\mcal{N})
\]
for each commutative algebra object $A$ in $\mcal{C}$.
\end{remark}

\begin{remark}
Let $\mcal{M}$ and $\mcal{N}$ be $\infty$-categories tensored over symmetric monoidal $\infty$-categories $\mcal{C}$ and $\mcal{D}$ respectively and let $F\colon\mcal{C}\to\mcal{D}$ be a symmetric monoidal functor.
By \cite[4.2.3.2]{HA}, the canonical functor $\Mod(\Cat_\infty)\to\CAlg(\Cat_\infty)$ is cartesian and $F^*\mcal{N}^\otimes$ exhibits $\mcal{N}$ as tensored over $\mcal{C}$.
It follows that an $F$-linear functor $\mcal{M}\to\mcal{N}$ is identified with a $\mcal{C}$-linear functor $\mcal{M}\to\mcal{N}$, which is further identified with a morphism in $\Mod_\mcal{C}(\Cat_\infty)$.
In other words, we have an equivalence
\[
	\Fun_F^\otimes(\mcal{M},\mcal{N}) \simeq \Fun_\mcal{C}^\otimes(\mcal{M},\mcal{N})
\]
and its groupoid core is equivalent to $\Map_{\Mod_\mcal{C}(\Cat_\infty)}(\mcal{M},\mcal{N})$.
\end{remark}

\begin{p}[Adjunction]\label{p:Mod}
Let $F\colon\mcal{M}\to\mcal{N}$ be an $\mbb{M}$-monoidal functor between $\mbb{M}$-monoidal $\infty$-categories.
Suppose that the underlying functors $\mcal{M}_\mfk{a}\to\mcal{N}_\mfk{a}$ and $\mcal{M}_\mfk{m}\to\mcal{N}_\mfk{m}$ admit right adjoints.
Then, by \cite[7.3.2.7]{HA}, $F$ admits a right adjoint $G$ relative to $\mbb{M}^\otimes$ and $G$ is lax $\mbb{M}$-monoidal.
Consequently, we have an adjunction
\[
	F \colon \Mod(\mcal{M}) \rightleftarrows \Mod(\mcal{N}) \colon G.
\]
Furthermore, for a commutative algebra object $A$ in $\mcal{C}$, the induced functor $F\colon\Mod_A(\mcal{M})\to\Mod_{F(A)}(\mcal{N})$ is a left adjoint to the composition
\[
	\Mod_{F(A)}(\mcal{N}) \xrightarrow{G} \Mod_{GF(A)}(\mcal{M}) \to \Mod_A(\mcal{M}),
\]
where the second functor is the restriction of scalars along the unit map $A\to GF(A)$. 
\end{p}

\begin{p}[Monoidal enrichment]\label{p:Mod2}
Consider the functor
\[
	\Mod \colon \Mod(\Cat_\infty) \to \Cat_\infty.
\]
This functor preserves limits since it is co-representable by \cite[2.2.4.9]{HA}, and thus it is uniquely promoted to a symmetric monoidal functor with respect to the cartesian symmetric monoidal structures.
By applying $\Mod$ to this functor, we obtain a functor 
\[
	\Mod^\otimes \colon \Mod(\Mod(\Cat_\infty)) \to \Mod(\Cat_\infty).
\]
Since the wedge product $\wedge\colon\mbb{M}^\otimes\otimes\mbb{M}^\otimes\to\mbb{M}^\otimes$ induces a functor $\Mod(\Cat_\infty)\to\Mod(\Mod(\Cat_\infty))$, we in particular obtain an $\mbb{M}$-monoidal $\infty$-category $\Mod(\mcal{M})^\otimes$ for each $\mbb{M}$-monoidal $\infty$-category $\mcal{M}^\otimes$ and it exhibits $\Mod(\mcal{M})$ as tensored over $\Mod(\mcal{M}_\mfk{a})$.
Furthermore, the natural transformation $\Mod\to\CAlg\circ(-)_\mfk{a}$ yields an $\mbb{M}$-monoidal functor $\Mod(\mcal{M})^\otimes\to\CAlg(\mcal{M}_\mfk{a})^\otimes$.
\end{p}

\begin{lemma}\label{lem:Mod}
Let $\mcal{M}^\otimes$ be an $\mbb{M}$-monoidal $\infty$-category such that the underlying $\infty$-categories $\mcal{M}_\mfk{a}$ and $\mcal{M}_\mfk{m}$ admit geometric realizations and that the tensor products
\[
	\mcal{M}_\mfk{a}\times\mcal{M}_\mfk{a} \to \mcal{M}_\mfk{a} \qquad
	\mcal{M}_\mfk{a}\times\mcal{M}_\mfk{m} \to \mcal{M}_\mfk{m}
\]
preserve geometric realizations in each variable.
Then the $\mbb{M}$-monoidal functor $\Mod(\mcal{M})^\otimes\to\CAlg(\mcal{M}_\mfk{a})^\otimes$ is cocartesian.
Furthermore, the associated functor $\CAlg(\mcal{M}_\mfk{a})^\otimes\to\Cat_\infty$ is a lax cartesian structure, where we regard $\CAlg(\mcal{M}_\mfk{a})^\otimes$ as an $\mbb{M}$-monoidal $\infty$-category, and thus we obtain a lax $\mbb{M}$-monoidal functor
\[
	\Theta_\mcal{M} \colon \CAlg(\mcal{M}_\mfk{a})^\otimes \to \Cat_\infty^\times.
\]
\end{lemma}
\begin{proof}
The proof is parallel to \cite[4.5.3.1]{HA}.
The assertion that $\CAlg(\mcal{M}_\mfk{a})^\otimes\to\Cat_\infty$ is a lax cartesian structure is a formal consequence of the Segal condition for $\Mod(\mcal{M})^\otimes$.
\end{proof}

\begin{remark}
In the situation of Lemma \ref{lem:Mod}, we in particular obtain a functor
\[
	\Theta_\mcal{M} \colon \CAlg(\mcal{M}_\mfk{a}) \to \Mod(\Cat_\infty),
\]
which classifies module objects in $\mcal{M}$.
This functor carries a commutative algebra object $A$ in $\mcal{M}_\mfk{a}$ to an $\mbb{M}$-monoidal $\infty$-category $\Mod_A(\mcal{M})^\otimes$ that exhibits $\Mod_A(\mcal{M})$ as tensored over $\Mod_A(\mcal{M}_\mfk{a})$, where the tensor product is given by the relative tensor product $\otimes_A$, and carries a morphism $A\to B$ in $\CAlg(\mcal{M}_\mfk{a})$ to an $\mbb{M}$-monoidal functor
\[
	B\otimes_A- \colon \Mod_A(\mcal{M})^\otimes \to \Mod_B(\mcal{M})^\otimes,
\]
which we call the \textit{base change}.
\end{remark}

\begin{lemma}\label{lem:Mod2}
There is an approximation $(\Fin_*)_{\langle 1\rangle/}\to\mbb{M}^\otimes$ to $\mbb{M}^\otimes$ in the sense of \cite[2.3.3.6]{HA}.
\end{lemma}
\begin{proof}
The full subcategory of $\mbb{M}^\otimes$ spanned by $(I,S)$ with $|S|\le 1$ is canonically equivalent to $(\Fin_*)_{\langle 1\rangle/}$.
Then it is straightforward to check that the inclusion $(\Fin_*)_{\langle 1\rangle/}\to\mbb{M}^\otimes$ satisfies the condition in \cite[2.3.3.6]{HA}
\end{proof}

\begin{remark}
We set $\mbb{M}^\oslash:=(\Fin_*)_{\langle 1\rangle/}$.
We call a morphism in $\mbb{M}^\oslash$ \textit{inert} if it lies over an inert morphism in $\Fin_*$.
For an $\mbb{M}$-monoidal $\infty$-category $\mcal{M}^\otimes$, let $\Mod(\mcal{M})'$ be the full subcategory of $\Fun_{\mbb{M}^\otimes}(\mbb{M}^\oslash,\mcal{M}^\otimes)$ spanned by those functors $\mbb{M}^\oslash\to\mcal{M}^\otimes$ which carries inert morphisms to cocartesian morphisms over $\mbb{M}^\otimes$.
Then it follows from \cite[2.3.3.23]{HA} that the pre-composition by $\mbb{M}^\oslash\to\mbb{M}^\otimes$ induces an equivalence
\[
	\Mod(\mcal{M}) \xrightarrow{\sim} \Mod(\mcal{M})'.
\]
For a symmetric monoidal $\infty$-category $\mcal{C}$, the $\infty$-category $\Mod(\mcal{C})'$ is exactly $\Mod^\Comm(\mcal{C})$ in \cite[3.3.3.8]{HA}.
By applying this equivalence to $\Cat_\infty$, we see that giving an $\mbb{M}$-monoidal $\infty$-category $\mcal{M}^\otimes$ is equivalent to giving a cocartesian fibration $\mcal{M}^\oslash\to\mbb{M}^\oslash$ such that, for every $n\ge 0$ and $\alpha\in\mbb{M}^\oslash_{\langle n\rangle}$, the induced functor
\[
	\mcal{M}^\oslash_\alpha \to \prod_{1\le i\le n}\mcal{M}^\oslash_{\rho_i(\alpha)}
\]
is an equivalence, where $\rho_i\colon\langle n\rangle\to\langle 1\rangle$ is the inert morphism with $\rho_i^{-1}(1)=i$.
Concretely, $\mcal{M}^\oslash$ is given by the pullback $\mcal{M}^\oslash=\mcal{M}^\otimes\times_{\mbb{M}^\otimes}\mbb{M}^\oslash$.
\end{remark}

\begin{p}[Relation to left modules in \cite{HA}]
If we let $\mbb{LM}^\otimes$ be the operad in \cite[4.2.1.7]{HA}, then there is a canonical morphism of operads $\mbb{LM}^\otimes\to\mbb{M}^\otimes$ which makes the diagram 
\[
\xymatrix{
	\mbb{LM}^\otimes \ar[r] & \mbb{M}^\otimes \\
	\Assoc^\otimes \ar[u] \ar[r] & \Comm^\otimes \ar[u]
}
\]
commutative.
For an $\mbb{M}$-monoidal $\infty$-category $\mcal{M}^\otimes$, there is a canonical equivalence
\[
	\Mod(\mcal{M}) \simeq \LMod(\mcal{M})\times_{\Alg(\mcal{M}_\mfk{a})}\CAlg(\mcal{M}_\mfk{a}),
\]
where $\LMod(\mcal{M})$ is the $\infty$-category of left module objects in $\mcal{M}$, cf.\ \cite[4.2.1.13]{HA}.
\end{p}

\subsection{Presentably $\mbb{M}$-monoidal $\infty$-categories}\label{PrM}

\begin{p}
Let $\mcal{O}^\otimes$ be an $\infty$-operad.
A \textit{presentably $\mcal{O}$-monoidal $\infty$-category} is an $\mcal{O}$-monoidal $\infty$-category $\mcal{C}^\otimes$ such that, for each $x\in\mcal{O}$, the fiber $\mcal{C}_x$ is a presentable $\infty$-category and that the $\mcal{O}$-monoidal structure on $\mcal{C}$ is compatible with small colimits in the sense of \cite[3.1.1.18]{HA}.
A presentably $\mcal{O}$-monoidal $\infty$-category is identified with an $\mcal{O}$-algebra object in $\Pr^L$.
\end{p}

\begin{definition}[Presentably tensored $\infty$-category]\label{def:PrM}
Let $\mcal{C}^\otimes$ be a presentably symmetric monoidal $\infty$-category.
We say that an $\infty$-category $\mcal{M}$ is \textit{presentably tensored} over $\mcal{C}$ if we are supplied with a presentably $\mbb{M}$-monoidal $\infty$-category that exhibits $\mcal{M}$ as tensored over $\mcal{C}$. 
An $\infty$-category presentably tensored over $\mcal{C}$ is identified with a $\mcal{C}$-module object in $\Pr^L$.
\end{definition}

\begin{remark}
Let $\mcal{M}^\otimes$ be a presentably $\mbb{M}$-monoidal $\infty$-category.
By \cite[4.8.3.22]{HA}, $\Mod_A(\mcal{M})^\otimes$ is a presentably $\mbb{M}$-monoidal $\infty$-category for each commutative algebra object $A$ in $\mcal{M}_\mfk{a}$.
It follows that the functor $\Theta_\mcal{M}\colon\CAlg(\mcal{M}_\mfk{a})\to\Mod(\Cat_\infty)$ classifying module objects in $\mcal{M}$ (Lemma \ref{lem:Mod}) induces a functor
\[
	\Theta_\mcal{M} \colon \CAlg(\mcal{M}_\mfk{a}) \to \Mod(\Pr^L).
\]
Warn that the $\mbb{M}$-monoidal $\infty$-category $\Mod(\mcal{M})^\otimes$ is not presentably $\mbb{M}$-monoidal though the underlying $\infty$-categories are presentable, because the tensor products are not distributive with respect to coproducts.
\end{remark}

\begin{lemma}\label{lem:PrM}
Let $\mcal{M}^\otimes$ be a presentably $\mbb{M}$-monoidal $\infty$-category and $A$ a commutative algebra object in $\mcal{M}_\mfk{a}$.
Then there is a natural equivalence of $\mbb{M}$-monoidal $\infty$-categories
\[
	\Mod_A(\mcal{M}_\mfk{a})\otimes_{\mcal{M}_\mfk{a}}\mcal{M} \simeq \Mod_A(\mcal{M}),
\]
where the tensor product is taken in $\Mod_{\mcal{M}_\mfk{a}}(\Pr^L)$.
\end{lemma}
\begin{proof}
This is a special case of \cite[4.8.4.6]{HA}.
\end{proof}

\subsection{Constructions of $\mbb{M}$-monoidal structures}\label{CoM}

\begin{p}[Trivial $\mbb{M}$-monoidal structure]\label{p:CoM}
Let $\mcal{K}$ be an $\infty$-category.
Then there exists a unique $\mbb{M}$-monoidal $\infty$-category $\mcal{K}^\otimes$ that exhibits $\mcal{K}$ as tensored over $*$, since we have an equivalence $\Mod_*(\Cat_\infty)\simeq\Cat_\infty$.
We refer to this $\mbb{M}$-monoidal structure on $\mcal{K}$ as the \textit{trivial $\mbb{M}$-monoidal structure}.
Concretely, the cocartesian fibration $\mcal{K}^\otimes\to\mbb{M}^\otimes$ is classified by the left Kan extension of $\mcal{K}\colon\Triv^\otimes\to\Cat_\infty$ along $\Triv^\otimes\to\mbb{M}^\otimes$.
\end{p}

\begin{lemma}\label{lem:CoM}
Let $\mcal{K}$ be a small $\infty$-category, $\mcal{M}^\otimes$ an $\mbb{M}$-monoidal $\infty$-category, and $A$ a commutative algebra object in $\mcal{M}_\mfk{a}$.
Then there is a natural equivalence
\[
	\Fun^\lax_A(\mcal{K},\mcal{M}) \simeq \Fun(\mcal{K},\Mod_A(\mcal{M})),
\]
where $\mcal{K}$ is equipped with the trivial $\mbb{M}$-monoidal structure.
\end{lemma}
\begin{proof}
This is obvious when $\mcal{K}=*$.
We reduce the assertion to this case by showing that the contravariant functor $\mcal{K}\mapsto\Fun^\lax_A(\mcal{K},\mcal{M})$ carries colimits to limits.
By definition,
\[
	\Fun^\lax_A(\mcal{K},\mcal{M})=\Fun^\lax_\mbb{M}(\mcal{K},\mcal{M})\times_{\CAlg(\mcal{M}_\mfk{a})}\{A\},
\]
and thus it suffices to show that $\mcal{K}\mapsto\Fun^\lax_\mbb{M}(\mcal{K},\mcal{M})$ carries colimits to limits.
Note that $\Fun^\lax_\mbb{M}(\mcal{K},\mcal{M})$ is the mapping $\infty$-category of the $\infty$-category $(\Op_\infty)_{/\mbb{M}^\otimes}$ of $\infty$-operads over $\mbb{M}^\otimes$.
Hence, it suffices to show that the functor
\[
	\Cat_\infty \to (\Op_\infty)_{/\mbb{M}^\otimes} \qquad \mcal{K}\mapsto\mcal{K}^\otimes
\]
preserves colimits.
This follows from the following observation: Given a diagram $\{\mcal{K}_i\}_{i\in I}$ of small $\infty$-categories, the colimit $\colim\mcal{K}_i^\otimes$ taken as $\infty$-preoperad is an $\infty$-operad.
\end{proof}

\begin{construction}[Pointwise $\mbb{M}$-monoidal structure]\label{cons:CoM}
Consider the functor
\[
	\Fun(-,-) \colon (\Cat_\infty^\sm)^\op\times\Cat_\infty \to \Cat_\infty.
\]
Since $\Fun(\mcal{K},-)$ preserves limits for each small $\infty$-category $\mcal{K}$, we obtain an induced functor
\[
	\Fun(-,-)^\otimes \colon (\Cat_\infty^\sm)^\op\times\Mod(\Cat_\infty) \to \Mod(\Cat_\infty).
\]
In particular, if $\mcal{K}$ is a small $\infty$-category and $\mcal{M}^\otimes$ is an $\mbb{M}$-monoidal $\infty$-category, then $\Fun(\mcal{K},\mcal{M})^\otimes$ is an $\mbb{M}$-monoidal $\infty$-category that exhibits $\Fun(\mcal{K},\mcal{M}_\mfk{m})$ as tensored over $\Fun(\mcal{K},\mcal{M}_\mfk{a})$.
We refer to this $\mbb{M}$-monoidal structure on $\Fun(\mcal{K},\mcal{M})$ as the \textit{pointwise $\mbb{M}$-monoidal structure}.

We generalize this construction to relative functors.
Let $\mcal{M}^\otimes\to\mcal{N}^\otimes$ be an $\mbb{M}$-monoidal functor between $\mbb{M}$-monoidal $\infty$-categories and let $F\colon\mcal{K}\to\mcal{N}_\mfk{m}$ be an arbitrary functor.
Note that we can regard $F$ as an $\mbb{M}$-monoidal functor $\mbb{M}^\otimes\to\Fun(\mcal{K},\mcal{N})^\otimes$.
We define an $\mbb{M}$-monoidal $\infty$-category $\Fun_\mcal{N}(\mcal{K},\mcal{M})^\otimes$ by the pullback
\[
\xymatrix{
	\Fun_\mcal{N}(\mcal{K},\mcal{M})^\otimes \ar[r] \ar[d] & \Fun(\mcal{K},\mcal{M})^\otimes \ar[d] \\
	\mbb{M}^\otimes \ar[r]^-F & \Fun(\mcal{K},\mcal{N})^\otimes.
}
\]
Then $\Fun_\mcal{N}(\mcal{K},\mcal{M})^\otimes$ exhibits $\Fun_{\mcal{N}_{\mfk{m}}}(\mcal{K},\mcal{M}_{\mfk{m}})$ as tensored over $\Fun_{\mcal{N}_{\mfk{a}}}(\mcal{K},\mcal{M}_{\mfk{a}})$, where the structure functor $\mcal{K}\to\mcal{N}_\mfk{a}$ is the constant functor onto the unit in $\mcal{N}_\mfk{a}$.
\end{construction}

\begin{construction}[Fiberwise $\mbb{M}$-monoidal structure]\label{cons:CoM2}
Let $\mcal{K}$ be a small $\infty$-category and $\mcal{C}$ a symmetric monoidal $\infty$-category.
Suppose we are given a functor $X\colon\mcal{K}\to\Mod_\mcal{C}(\Cat_\infty)$ and let $\mcal{E}\to\mcal{K}$ be the cocartesian fibration classified by $X$.
Then $X$ lifts to a lax $\mbb{M}$-monoidal functor $\bar{X}\colon\mcal{K}^\otimes\to\Cat_\infty^\times$ by Lemma \ref{lem:CoM}.
Let $\mcal{E}^\otimes\to\mcal{K}^\otimes$ be the cocartesian fibration classified by $\bar{X}$.
Then $\mcal{E}^\otimes$ is an $\mbb{M}$-monoidal $\infty$-category that exhibits $\mcal{E}$ as tensored over $\mcal{C}$ and the functor $\mcal{E}^\otimes\to\mcal{K}^\otimes$ is $\mbb{M}$-monoidal.
We refer to this $\mbb{M}$-monoidal structure on $\mcal{E}$ as the \textit{fiberwise $\mbb{M}$-monoidal structure}.

Combining this with Construction \ref{cons:CoM}, we get an $\mbb{M}$-monoidal $\infty$-category $\Fun_{\mcal{K}}(\mcal{K},\mcal{E})^\otimes$ that exhibits $\Fun_{\mcal{K}}(\mcal{K},\mcal{E})$ as tensored over $\Fun(\mcal{K},\mcal{C})$.
\end{construction}

\subsection{Day convolution}\label{Day}

We reformulate the Day convolution monoidal structures as in \cite[2.2.6]{HA} in a way that is convenient for our purpose.

\begin{construction}[Day convolution]\label{cons:Day}
By \cite[4.8.1.3]{HA} (see also \cite[4.8.1.8]{HA}), we have a symmetric monoidal functor
\[
	\mcal{P} \colon \Cat_\infty^{\sm,\times} \to \Pr^{L,\otimes},
\]
which carries a small $\infty$-category $\mcal{K}$ to the $\infty$-category $\mcal{P}(\mcal{K})$ of presehaves on $\mcal{K}$ and a functor $f\colon \mcal{K}\to\mcal{L}$ between small $\infty$-categories to the left Kan extension $f_!\colon\mcal{P}(\mcal{K})\to\mcal{P}(\mcal{L})$.
Note that, for a small $\infty$-category $\mcal{K}$ and a presentable $\infty$-category $\mcal{C}$, we have a canonical equivalence
\[
	\Fun(\mcal{K},\mcal{C}) \simeq \mcal{P}(\mcal{K}^\op)\otimes\mcal{C},
\]
where the tensor product is taken in $\Pr^L$.
On the other hand, since the symmetric monoidal $\infty$-category $\Pr^{L,\otimes}$ can be regarded as an algebra object in $\CAlg(\Cat_\infty)$, we have a symmetric monoidal functor
\[
	\otimes \colon \Pr^{L,\otimes}\times_{\Fin_*}\Pr^{L,\otimes} \to \Pr^{L,\otimes},
\]
which lifts the usual tensor product.
By composing those two symmetric monoidal functors, we obtain a symmetric monoidal functor
\[
	\Fun(-,-)^\otimes \colon \Cat_\infty^{\sm,\times}\times_{\Fin_*}\Pr^{L,\otimes} \to \Pr^{L,\otimes},
\]
which lifts the functor $(\mcal{K},\mcal{C})\mapsto\Fun(\mcal{K},\mcal{C})$.
\end{construction}

\begin{remark}
By applying $\Mod$ to the symmetric monoidal functor $\Fun(-,-)^\otimes$, we obtain a functor
\[
	\Fun(-,-)^\otimes \colon \Mod(\Cat_\infty^\sm)\times\Mod(\Pr^L) \to \Mod(\Pr^L).
\]
For a small $\mbb{M}$-monoidal $\infty$-category $\mcal{K}^\otimes$ and a presentably $\mbb{M}$-monoidal $\infty$-category $\mcal{M}^\otimes$, 
the resulting presentably $\mbb{M}$-monoidal $\infty$-category $\Fun(\mcal{K},\mcal{M})^\otimes$ coincides with the one constructed in \cite[2.2.6]{HA} and we have equivalences
\begin{align*}
	\Mod(\Fun(\mcal{K},\mcal{M})) &\simeq \Fun^\lax_\mbb{M}(\mcal{K},\mcal{M}) \\
	\CAlg(\Fun(\mcal{K},\mcal{M})_\mfk{a}) &\simeq \Fun^\lax(\mcal{K}_\mfk{a},\mcal{M}_\mfk{a}) \\
	\Mod_F(\Fun(\mcal{K},\mcal{M})) &\simeq \Fun^\lax_F(\mcal{K}_\mfk{m},\mcal{M}_{\mfk{m}}),
\end{align*}
where $F$ is a lax symmetric monoidal functor $\mcal{K}_\mfk{a}\to\mcal{M}_\mfk{a}$.
We refer to this $\mbb{M}$-monoidal structure on $\Fun(\mcal{K},\mcal{M})$ as the \textit{Day convolution $\mbb{M}$-monoidal structure}.
\end{remark}

\begin{remark}
The Day convolution $\mbb{M}$-monoidal structure is compatible with the pointwise $\mbb{M}$-monoidal structure in the following cases:
\begin{enumerate}
\item If $\mcal{K}^\otimes$ is a small $\mbb{M}$-monoidal $\infty$-category with the trivial $\mbb{M}$-monoidal structure, then the Day convolution $\mbb{M}$-monoidal structure on $\Fun(\mcal{K},\mcal{M})$ is the restriction of scalars of the pointwise $\mbb{M}$-monoidal structure along the symmetric monoidal functor $\mcal{M}_\mfk{a}\to\Fun(\mcal{K},\mcal{M}_\mfk{a})$.
\item If $\mcal{K}^\otimes$ is a cocartesian symmetric monoidal $\infty$-category, then the Day convolution $\mbb{M}$-monoidal structure on $\Fun(\mcal{K},\mcal{M})$ is identified with the pointwise $\mbb{M}$-monoidal structure.
\end{enumerate}
\end{remark}

\begin{lemma}\label{lem:Day}
Let $\mcal{K}^\otimes$ be a small $\mbb{M}$-monoidal $\infty$-category and $\mcal{M}^\otimes$ a presentably $\mbb{M}$-monoidal $\infty$-category.
Then there is a natural equivalence of $\mbb{M}$-monoidal $\infty$-categories
\[
	\Fun(\mcal{K},\mcal{M}) \simeq \Fun(\mcal{K},\mcal{M}_\mfk{a})\otimes_{\mcal{M}_\mfk{a}}\mcal{M},
\]
where the tensor product is taken in $\Mod_{\mcal{M}_\mfk{a}}(\Pr^L)$ and $\Fun(\mcal{K},\mcal{M}_\mfk{a})$ is tensored over $\mcal{M}_\mfk{a}$ by the restriction of scalars along the the symmetric monoidal functor $\mcal{M}_\mfk{a}\to\Fun(\mcal{K}_\mfk{a},\mcal{M}_\mfk{a})$ obtained as the left Kan extension along $*\to\mcal{K}_\mfk{a}$.
\end{lemma}
\begin{proof}
This is immediate from the equivalence $\Fun(\mcal{K},\mcal{M})\simeq\mcal{P}(\mcal{K}^\op)\otimes\mcal{M}$.
\end{proof}

\subsection{Smashing localizations}\label{SmL}

A criterion for smashing localization (cf.\ \ref{CCo}) is given.

\begin{definition}[Ideal/co-ideal]\label{def:SmL}
Let $\mcal{C}$ be a symmetric monoidal $\infty$-category and $\mcal{I}$ a full subcategory of $\mcal{C}$.
Then:
\begin{enumerate}
\item $\mcal{I}$ is an \textit{ideal} if for every $c\in\mcal{C}$ and $x\in\mcal{I}$ we have $c\otimes x\in\mcal{I}$.
\end{enumerate}
Suppose further that $\mcal{C}$ is closed.
Then:
\begin{enumerate}[resume]
\item $\mcal{I}$ is a \textit{co-ideal} if for every $c\in\mcal{C}$ and $x\in\mcal{I}$ we have $\intMap(c,x)\in\mcal{I}$.
\end{enumerate}
\end{definition}

\begin{lemma}\label{lem:SmL}
Let $\mcal{C}$ be a symmetric monoidal $\infty$-category and $L\colon\mcal{C}\to\mcal{C}'$ a localization.
Then the following are equivalent:
\begin{enumerate}
\item $L$ is smashing.
\item $L$ is symmetric monoidal and $\mcal{C}'$ is an ideal of $\mcal{C}$
\end{enumerate}
Suppose that $\mcal{C}$ is closed, then these are further equivalent to:
\begin{enumerate}[resume]
\item $\mcal{C}'$ is an ideal and co-ideal of $\mcal{C}$.
\end{enumerate}
\end{lemma}
\begin{proof}
The implication (i)$\Rightarrow$(ii) is obvious.
Assume the condition (ii).
Then it follows from \cite[4.1.7.4]{HA} that the essential image $L\mcal{C}$ admits a unique symmetric monoidal structure for which the functor $L\colon\mcal{C}\to L\mcal{C}$ is promoted to a symmetric monoidal functor.
In particular, $A:=L(\mbf{1})$ is a unit object in $L\mcal{C}$ and thus $L\mcal{C}\simeq\Mod_A(L\mcal{C})\subset\Mod_A(\mcal{C})$.
Since $L\mcal{C}\subset\mcal{C}$ is an ideal, we see that $A$ is an idempotent algebra in $\mcal{C}$.
We claim $\Mod_A(\mcal{C})=L\mcal{C}$, which proves that $L$ is smashing.
We have already checked one inclusion.
To show the other inclusion note that, for $x\in\Mod_A(\mcal{C})$, we have
\[
	x\otimes A \simeq (x\otimes_AA)\otimes A \simeq x\otimes_A(A\otimes A) \simeq x\otimes_AA \simeq x.
\]
Since $L\mcal{C}\subset\mcal{C}$ is an ideal and $A\in L\mcal{C}$, we have $x\in L\mcal{C}$.
This proves the claim and thus (ii)$\Rightarrow$(i).

Suppose that $\mcal{C}$ is closed.
Then we have $\Map(x\otimes c,y)\simeq \Map(x,\intMap(c,y))$ for $x,y,c\in\mcal{C}$.
It follows from this equivalence that $c\otimes-$ preserves $L$-equivalences if and only if $\intMap(c,-)$ preserves $L$-local objects.
The later is exactly the condition $L\mcal{C}\subset\mcal{C}$ is a co-ideal and thus (ii)$\Leftrightarrow$(iii).
\end{proof}

\end{document}